\documentclass[letterpaper,11pt]{article}
\usepackage{graphicx,xcolor,url,subcaption,authblk,float}
\usepackage{amsmath,amsthm,amssymb,amsfonts,mathtools,bm}
\usepackage{enumitem}
\usepackage{algorithm,algcompatible}
\usepackage{tabularx,booktabs,multirow}
\newcolumntype{C}{>{\centering\arraybackslash}X}
\newcolumntype{R}{>{\raggedleft\arraybackslash}X}
\newcolumntype{L}{>{\raggedright\arraybackslash}X}
\makeatletter
\newcommand{\multiline}[1]{%
    \begin{tabularx}{\dimexpr\linewidth-\ALG@thistlm}[t]{@{}X@{}}
        #1
    \end{tabularx}
}
\makeatother
\algnewcommand{\IIf}[1]{\State\algorithmicif\ #1\ \algorithmicthen}
\algnewcommand{\EndIIf}{\unskip\ \algorithmicend\ \algorithmicif}
\algnewcommand{\IElse}{\unskip\ \algorithmicelse \unskip\ }

\DeclareMathOperator{\tr}{\mathrm{tr}}
\newcommand{\one}{\bm{\mathrm{e}}}
\newcommand{\eye}{\bm{\mathrm{I}}}
\allowdisplaybreaks

\usepackage[margin=1in]{geometry}

\newtheorem{theorem}{Theorem}
\newtheorem{remark}{Remark}
\newtheorem{lemma}{Lemma}
\begin{document}

\title{A Lagrangian Dual Method for Two-Stage Robust Optimization with Binary Uncertainties\thanks{This material is based upon work supported by the U.S. Department of Energy, Office of Science, Advanced Scientific Computing Research, under Contract DE-AC02-06CH11357.}}

\author{Anirudh Subramanyam}
\affil{Mathematics and Computer Science Division, Argonne National Laboratory, Lemont, IL}

\maketitle

\begin{abstract}
    This paper presents a new exact method to calculate worst-case parameter realizations in two-stage robust optimization problems with categorical or binary-valued uncertain data. Traditional exact algorithms for these problems, notably Benders decomposition and column-and-constraint generation, compute worst-case parameter realizations by solving mixed-integer bilinear optimization subproblems. However, their numerical solution can be computationally expensive not only due to their resulting large size after reformulating the bilinear terms, but also because decision-independent bounds on their variables are typically unknown. We propose an alternative Lagrangian dual method that circumvents these difficulties and is readily integrated in either algorithm. We specialize the method to problems where the binary parameters switch on or off constraints as these are commonly encountered in applications, and discuss extensions to problems that lack relatively complete recourse and to those with integer recourse. Numerical experiments provide evidence of significant computational improvements over existing methods.
    
\end{abstract}

\section{Motivation}\label{sec:motivation}

Over the last two decades, robust optimization has emerged as one of the most popular approaches to address optimization problems with uncertain parameters. The key idea is to replace a probabilistic description of the unknown problem data with a deterministic set-based model, and identify solutions that are optimal under the worst-case data realizations from the latter so-called uncertainty set.
The approach can not only circumvent some of the computational and modeling difficulties associated with traditional stochastic approaches, but can also offer probabilistic performance guarantees of its solutions via a careful choice of the problem formulation and uncertainty set.
It has been successfully applied to problems in a wide variety of domains including
energy \cite{conejo2021robust},
transportation \cite{subramanyam2018robust},
telecommunication \cite{koster2013robust},
engineering design \cite{beyer2007robust},
and finance \cite{fabozzi2007robust}, to name but a few.

The rich gamut of applications have been enabled by several recent developments in robust optimization theory and algorithms.
The rich methodology of static or one-stage robust optimization problems, where all decisions have to be taken before any uncertainties are resolved, is reviewed in \cite{ben2009robust,gabrel2014recent}.
However, dynamic or multi-stage robust optimization problems, where some subset of the decisions can adapt to observations of unknown problem data, especially those involving categorical or discrete quantities, continue to pose theoretical and algorithmic challenges.

The goal of this paper is to address some of these challenges. %
Our motivation stems from the broad modeling power of categorical or binary-valued random variables,
particularly in applications that are laced with uncertainty about inherently discrete or non-numerical quantities.
Examples include:
\begin{enumerate}
    \item \emph{Optimization over networks or graph structures.} Binary parameters are natural candidates to model random node or link failures in networks. Depending on the application, the latter can represent connections that are either physical, as in the case of electric power \cite{yuan2016robust}, transportation \cite{xu2020robust} and supply chain \cite{lu2015reliable} systems, or virtual, as in the case of telecommunication traffic engineering \cite{liu2014traffic} and industrial control systems \cite{wang2007robust}. %
    Related examples include adversarial network interdiction \cite{smith2013modern} %
    and the design of survivable networks \cite{kerivin2005design}.
    
    \item \emph{Unit or service availability in manufacturing and service operations.} The unavailability (or failure) of machines, equipment, and units in a production scheduling environment \cite{goren2008robustness,feng2019finite}, or of particular services in the context of service operations \cite{dong2020managing} are naturally modeled as binary random variables.
    
    \item \emph{Discrete demand and preference models.} In many applications, customer demands are naturally discrete quantities. For example, logistics operators may not know \emph{a priori} which customers to visit \cite{albareda2011facility,subramanyam2021robust}. %
    Another example is the phenomenon of no-shows in queuing or appointment systems with scheduled arrivals, such as in healthcare clinics \cite{jiang2017integer}.
    Finally, in discrete choice modeling and ranking-based portfolio management, the uncertain parameters naturally belong to discrete sets; for example, in the latter, unknown asset preferences can be modeled as uncertain rankings or permutations over all possible orderings \cite{nguyen2012robust}.
    
    \item \emph{Cardinality-constrained uncertainty models.} One of the most popular uncertainty sets is the cardinality-constrained or Bertsimas-Sim model \cite{bertsimas2004price}, which stipulates a budget on the number of uncertain parameters that can simultaneously deviate from their nominal values. This set has a natural interpretation as a discrete uncertainty set. In fact, any continuous-valued unknowns that are modeled using this set can be equivalently reformulated using a binary set.
    These models are similar to reliability and safety engineering concepts that are referred to as ``$N-k$ reliability'' \cite{bienstock2010the} or ``$N+k$ redundancy'' \cite{barroso2013datacenter}, where the goal is to design a system with $N$ components such that it remains operational under any combination of $k$ failures.
    
    \item \emph{Scenario-based models of continuous unknowns.} Scenario-based uncertainty sets have been widely used in robust combinatorial optimization \cite{kouvelis2013robust,garuba2020comparison} %
    as well as other applications such as power systems engineering \cite{jabr2014robust}.
    As before, such sets can be equivalently reformulated and interpreted as binary sets.
    
    \item \emph{Adversarial machine learning with categorical data.} Relevant applications of robust optimization with categorical uncertainty include logistic regression \cite{shafieezadeh2015distributionally} and binary classification problems with noisy labels \cite{caramanis201214}.
\end{enumerate}
Motivated by these applications, in Section~\ref{sec:formulation}, we first describe the problem formulations of interest, then provide an overview of existing algorithms for their solution, and then summarize the contributions of this paper.

\section{Problem Formulation and Contributions}\label{sec:formulation}
We study two different problem formulations.
The first one, which we denote by \ref{eq:two_stage_ro_general}, is fairly general where both the constraints and objective function are affected by uncertainty.
\begin{equation}\label{eq:two_stage_ro_general}\tag{$\mathcal{P}$}
    \begin{aligned}
        &\inf_{\bm{x} \in \mathcal{X}} \sup_{\bm{\xi} \in \Xi} \, \mathcal{Q}(\bm{x}, \bm{\xi}), \\
        & %
        \mathcal{Q}(\bm{x}, \bm{\xi}) = 
        \left[
        \begin{aligned}
            \mathop{\text{minimize}}_{\bm{y} \in \mathcal{Y}} \;\; & \bm{c}(\bm{\xi})^\top \bm{x}  + \bm{d}(\bm{\xi})^\top \bm{y}  \\
            \text{subject to} \; & \bm{T}\bm{x} + \bm{W} \bm{y} \geq \bm{h}(\bm{\xi})
        \end{aligned}
        \right].
    \end{aligned}
\end{equation}
Here, $\bm{x}$, $\bm{y}$ and $\bm{\xi}$ are the first- and second-stage decision variables, and the vector of uncertain parameters, respectively.
We will assume throughout the paper that the feasible decision sets $\mathcal{X} \subseteq \mathbb{R}^{n_1}$ and $\mathcal{Y} \subseteq \mathbb{R}^{n_2}$, and the uncertainty set $\Xi \subseteq \{0, 1\}^{n_p}$ are non-empty and mixed-integer linear programming (MILP) representable;
$\mathcal{X}$ is compact;
$\bm{c}: \{0, 1\}^{n_p} \to \mathbb{R}^{n_1}$, $\bm{d} : \{0, 1\}^{n_p} \to \mathbb{R}^{n_2}$ and $\bm{h} : \{0, 1\}^{n_p} \to \mathbb{R}^{m}$ are vector-valued affine functions; and $\bm{T} \in \mathbb{R}^{m \times n_1}$ and $\bm{W} \in \mathbb{R}^{m \times n_2}$ are fixed matrices.
Finally, we note that $\mathcal{Q} : \mathbb{R}^{n_1} \times \{0, 1\}^{n_p} \to \mathbb{R} \cup \{-\infty, +\infty\}$ is the second-stage value function, in which we also include the first-stage costs $\bm{c}(\bm{\xi})^\top \bm{x}$ for ease of exposition.
We adopt the convention that the optimal value of a minimization (maximization) problem is $-\infty$ ($+\infty$) if it is unbounded and $+\infty$ ($-\infty$) if it is infeasible.

The second formulation, which we denote by \ref{eq:two_stage_ro_indicator}, uses the uncertain parameters as binary indicators that switch on or off second-stage constraints.
\begin{equation}\label{eq:two_stage_ro_indicator}\tag{$\mathcal{P}_\mathcal{I}$}
    \begin{aligned}
        &\inf_{\bm{x} \in \mathcal{X}} \sup_{\bm{\xi} \in \Xi} \, \mathcal{Q}_\mathcal{I}(\bm{x}, \bm{\xi}), \\
        & %
        \mathcal{Q}_\mathcal{I}(\bm{x}, \bm{\xi}) = 
        \left[\begin{aligned}
            \mathop{\text{minimize}}_{\bm{y} \in \mathcal{Y}} \;\; & \bm{c}(\bm{\xi})^\top \bm{x}  + \bm{d}(\bm{\xi})^\top \bm{y}  \\
            \text{subject to} \; &  \bm{g}(\bm{x}, \bm{y}) \geq \bm{0} \\
            & \xi_j = 0 \implies g_i(\bm{x}, \bm{y}) = 0, \;\; i \in \mathcal{I}_j^0, \; j \in [n_p] \\
            & \xi_j = 1 \implies g_i(\bm{x}, \bm{y}) = 0, \;\; i \in \mathcal{I}_j^1, \; j \in [n_p]
        \end{aligned}\right].
    \end{aligned}
\end{equation}
Here, $\bm{g} : \mathcal{X} \times \mathcal{Y} \to \mathbb{R}^m$ is a vector-valued affine function, $\mathcal{I}_j^0, \mathcal{I}_j^1 \subseteq [m]$ are index sets. %
and $[n]$ denotes the set $\{1, 2, \ldots, n\}$ for any positive integer $n$.

The indicator formulation \ref{eq:two_stage_ro_indicator} is ubiquitous in applications where the uncertain presence or absence of a quantity can modify the geometry of the feasible set.
For example, in the context of network optimization, $g_i(\bm{x}, \bm{y})$ may represent flow of some commodity $i$ through the $j^\text{th}$ arc (or node) of the network, which must be set to $0$ whenever the arc (or node) becomes unavailable due to random failure, as indicated by $\xi_j = 1$.
As we highlighted previously, similar such structures arise, among others, in the context of unit (un-)availability in production scheduling, discrete demand locations in logistics, and customer no-shows in appointment scheduling.
Although the indicator formulation \ref{eq:two_stage_ro_indicator} can be viewed as a special case of the more general problem \ref{eq:two_stage_ro_general} after performing an appropriate reformulation (e.g., big-M), we preserve the indicator structure of \ref{eq:two_stage_ro_indicator}, since it is more natural and intuitive.
Perhaps more importantly, we will show that we can exploit the specific structure of \ref{eq:two_stage_ro_indicator} to obtain computational improvements.
Before we proceed, however, we make some important remarks about formulations \ref{eq:two_stage_ro_general} and \ref{eq:two_stage_ro_indicator}.

\begin{remark}
    Problem~\ref{eq:two_stage_ro_general} can also accommodate random technology and recourse matrices, $\bm{T} : \{0, 1\}^{n_p} \to \mathbb{R}^{m \times n_1}$ and $\bm{W} : \{0, 1\}^{n_p} \to \mathbb{R}^{m \times n_2}$, whenever they are matrix-valued affine functions and the resulting bilinear terms, $\bm{T}(\bm{\xi}) \bm{x}$ and $\bm{W}(\bm{\xi}) \bm{y}$, can be exactly linearized (e.g., using Glover inequalities~\cite{glover1975improved}). The reformulated problem can be readily expressed as an instance of~\ref{eq:two_stage_ro_general}.
\end{remark}

\begin{remark}
    The non-negativity constraint $\bm{g}(\bm{x}, \bm{y}) \geq \bm{0}$  in problem \ref{eq:two_stage_ro_indicator} does not restrict generality.
    Indeed, suppose that for some fixed $i \in  [m]$, the quantity $g_i(\bm{x}, \bm{y})$ is unconstrained, but must be set to $0$ whenever $\xi_j = 1$.
    In such cases, we can augment the second-stage decision vector $\bm{y}$ to $(\bm{y}, g_i^{+}, g_i^{-})$,
    replace the indicator constraint in \ref{eq:two_stage_ro_indicator} with
    $
    \xi_j = 1 \implies \left[g_i^{+} = g_i^{-} = 0 \right],
    $
    and impose the additional constraints:
    $
    g_i^{+}, g_i^{-} \geq 0
    $
    and
    $
    g_i(\bm{x}, \bm{y}) = g_i^{+} - g_i^{-}
    $.
\end{remark}

The problem \ref{eq:two_stage_ro_general} (or \ref{eq:two_stage_ro_indicator}) is said to satisfy the relatively complete recourse assumption whenever the second-stage problem is feasible for all first-stage decisions and uncertain parameter realizations; in other words, $\mathcal{Q}(\bm{x}, \bm{\xi}) < +\infty$ (or $\mathcal{Q}_\mathcal{I}(\bm{x}, \bm{\xi}) < +\infty$) for all $\bm{x} \in \mathcal{X}$ and $\bm{\xi} \in \Xi$.
Similarly, it is said to satisfy the sufficiently expensive recourse assumption whenever the second-stage value function is bounded from below for all first-stage decisions and parameters; in other words, $\mathcal{Q}(\bm{x}, \bm{\xi}) > -\infty$ (or $\mathcal{Q}_\mathcal{I}(\bm{x}, \bm{\xi}) > -\infty$) for all $\bm{x} \in \mathcal{X}$ and $\bm{\xi} \in \Xi$.
We do not need all of these assumptions but state them here for ease of exposition.
We list our assumptions in Section~\ref{sec:reformulation}.
Regarding notation,
we use $\mathbb{R}_{+}$ and $\mathbb{Z}_{+}$
to denote the non-negative reals and integers;
and $\one_j$, $\one$, $\eye$ and $\bm{0}$ to denote the $j^\text{th}$ canonical unit vector, the vector of ones, the identity matrix, and the zero vector/matrix, respectively; their dimensions should be clear from the context.

\subsection{Literature review}\label{sec:formulation:literature}
In the presence of binary uncertainties, the two-stage robust optimization problem \ref{eq:two_stage_ro_general} is NP-hard even if there are no first-stage decisions (i.e., $\mathcal{X}$ is a singleton) and the second-stage problem is a two-dimensional linear program with only uncertain objective coefficients (i.e., $\mathcal{Y} = \mathbb{R}^{2}_{+}$ and $\bm{h}(\bm{\xi})$ is deterministic) \cite{subramanyam2020data}.
Indeed, it is solvable in polynomial time only in few special cases, such as when the uncertainty set $\Xi$ has a small inner description (i.e., in terms of a polynomial number of extreme points) or when $\bm{h}(\bm{\xi}) = \bm{0}$ and the matrices describing the constraints of $\Xi$ and the slopes of the affine function $\bm{d}(\bm{\xi})$ are totally unimodular \cite{mittal2020robust,subramanyam2020data}. 

Problem \ref{eq:two_stage_ro_general} simplifies substantially in the absence of binary uncertainties, when the uncertainty set $\Xi$ is linear programming (LP) representable.
In this case, existing algorithms for its approximate solution include:
\begin{enumerate}[label=\itshape\roman*)]
    \item decision rule approaches \cite{ben2004adjustable,bertsimas2015design,zhen2018adjustable,georghiou2019decision}, which approximate the optimal second-stage decisions $\bm{y}$ using parametric classes of linear or nonlinear functions of the uncertain parameters $\bm{\xi}$;
    \item uncertainty set partitioning approaches \cite{postek2016multistage,bertsimas2016multistage}, which partition the uncertainty set $\Xi$ into simple subsets (e.g., hyperrectangles) and approximate $\bm{y}$ by affine or constant functions of $\bm{\xi}$ over each subset;
    \item finite- or $K$-adaptability approaches \cite{bertsimas2010finite,hanasusanto2015k,buchheim2017min,subramanyam2020k}, which calculate a small, fixed number of second-stage decisions $\bm{y}_1, \ldots, \bm{y}_K$ in the first-stage, and subsequently select the best of these in response to a given realization of $\bm{\xi}$ in the second-stage; and,
    \item convexification approaches \cite{hanasusanto2018conic,ardestani2020linearized,xu2018copositive}, which first reformulate the infinite-dimensional two-stage problem as a finite-dimensional (but still intractable) one-stage problem and then approximate the latter using lift-and-project techniques (e.g., reformulation-linearization technique or using copositive programming).
\end{enumerate}
Several of these methods, particularly ones that are based on $K$-adaptability \cite{goerigk2020min,subramanyam2020k} and convexification \cite{mittal2020robust,jiang2019data,subramanyam2020data}, have also been extended to problems with mixed-integer uncertainties.

In addition to the above approximate schemes, a number of exact methods have also been proposed for problem \ref{eq:two_stage_ro_general} under various simplifying conditions.
When uncertainty affects only the objective function (i.e., $\bm{h}(\bm{\xi})$ is deterministic), the first- and second-stage decisions are binary (i.e., $\mathcal{X} \subseteq \{0, 1\}^{n_1}$ and $\mathcal{Y} \subseteq \{0, 1\}^{n_2}$) and $\Xi$ is a polytope, \cite{kammerling2020oracle} have proposed efficient algorithms to calculate lower bounds on the optimal value of \ref{eq:two_stage_ro_general} that is then integrated in a branch-and-bound scheme.
A similar approach but allowing also some continuous decisions in the first- and second-stage has been proposed in \cite{arslan2021decomposition}.

When uncertainty affects only the second-stage constraints' right-hand sides (i.e., $\bm{d}(\bm{\xi})$ is deterministic), the second-stage decisions are continuous (i.e., $\mathcal{Y}$ is LP representable) and $\Xi$ is the projection of a binary set (e.g., a cardinality-constrained set), a number of schemes based on Benders decomposition have been developed; e.g., see \cite{thiele2009robust,ayoub2016decomposition}.
A variant of this scheme that can solve problems with an exponential but finite number of uncertain parameter realizations that are described implicitly via block-structured binary uncertainty sets, has been recently developed in \cite{hashemi2021exploiting}.
An alternative exact approach based on semi-infinite programming techniques and known as column-and-constraint-generation was proposed in \cite{zhao2012exact,zeng2013solving}. %
Notably, this algorithm can also address problems where both the second-stage feasible region $\mathcal{Y}$ and the uncertainty set $\Xi$ are MILP representable.

Before we proceed to discuss the Benders decomposition and column-and-constraint generation algorithms in some detail,
we note that alternative approximation schemes for \ref{eq:two_stage_ro_general} and \ref{eq:two_stage_ro_indicator} can be devised by first ignoring the binary restriction on the uncertain parameters, and then using any (exact or approximate) method that is designed to address problems with continuous uncertainties.
However, recent studies \cite{subramanyam2020data,mittal2020robust} have shown that such approximations result in unnecessarily conservative (i.e., suboptimal) solutions even when the first- and second-stage feasible regions $\mathcal{X}$ and $\mathcal{Y}$ are convex.

\subsubsection{Overview of Benders decomposition algorithm}\label{sec:formulation:literature:Benders}
Algorithm \ref{algo:benders} outlines a basic Benders decomposition scheme for solving \ref{eq:two_stage_ro_general} when the second-stage decisions are continuous ($\mathcal{Y} = \mathbb{R}^{n_2}_{+}$) %
and the relatively complete and sufficiently expensive recourse assumptions are satisfied.
The basic idea is to compute lower bounds $LB$ in line~\ref{algo:benders:lb-update} by iteratively refining an inner approximation of
$\sup_{\bm{\xi} \in \Xi} \mathcal{Q}(\bm{x}, \bm{\xi})$.
Indeed, the lower bounding problem~\eqref{eq:benders_lb_update} can be derived %
by noting that strong LP duality implies that we have:
\begin{align}
    &\sup_{\bm{\xi} \in \Xi} \mathcal{Q}(\bm{x}, \bm{\xi}) \notag \\
    =&
    \sup_{\bm{\xi} \in \Xi} \bm{c}(\bm{\xi})^\top \bm{x}  + \inf_{\bm{y} \in \mathcal{Y}}
    \left\{ \bm{d}(\bm{\xi})^\top \bm{y}
    : \bm{W} \bm{y} \geq \bm{h}(\bm{\xi}) - \bm{T}\bm{x}
    \right\} \notag \\
    =& \sup_{\bm{\xi} \in \Xi} \bm{c}(\bm{\xi})^\top \bm{x}  + \sup_{\bm{\mu} \in \mathbb{R}^m_{+}}
    \left\{\bm{\mu}^\top \left(\bm{h}(\bm{\xi}) - \bm{T}\bm{x}\right)
    : \bm{W}^\top \bm{\mu} \leq \bm{d}(\bm{\xi})
    \right\} \notag \\
    =& \sup_{\bm{\xi}, \bm{\mu}}
    \left\{
    \bm{c}(\bm{\xi})^\top\bm{x} +
    \bm{\mu}^\top \big(\bm{h}(\bm{\xi}) - \bm{T}\bm{x}\big) :
    \begin{aligned}
        &\bm{\xi} \in \Xi, \; \bm{\mu} \in \mathbb{R}^m_{+} \\
        &\bm{W}^\top \bm{\mu} \leq \bm{d}(\bm{\xi})
    \end{aligned}
    \right\}, \label{eq:worst_case_problem_duality} \tag{$\mathcal{WC}_\mathcal{D}$}
\end{align}
where the inner maximization over $\bm{\mu}$ is feasible for all $\bm{x} \in \mathcal{X}$ and $\bm{\xi} \in \Xi$
under the relatively complete and sufficiently expensive recourse assumptions.
Therefore, Algorithm~\ref{algo:benders} implicitly enumerates the extreme points of the feasible region of the worst-case problem~\ref{eq:worst_case_problem_duality} (the subscript $\mathcal{D}$ stands for ``duality'') in the set $\mathcal{O}$.
In most implementations, the worst-case parameter realization $\hat{\bm{\xi}}$ and associated dual values $\hat{\bm{\mu}}$ are not computed individually as indicated in line~\ref{algo:benders:worst-case-computation};
rather, they are simultaneously calculated by solving problem \eqref{eq:worst_case_problem_duality} directly.
The algorithm is guaranteed to terminate after finitely many iterations, since the feasible region of \ref{eq:worst_case_problem_duality} has a finite number of extreme points.
Finally, we note that the lower bounding problem~\eqref{eq:benders_lb_update} can be readily reformulated by introducing an epigraphical variable and solved using an MILP solver.

\begin{algorithm}[!htbp]
    \caption{Basic Benders decomposition scheme to solve \ref{eq:two_stage_ro_general} when $\mathcal{Y} = \mathbb{R}^{n_2}_{+}$ and the relatively complete recourse assumption is satisfied}
    \label{algo:benders}
    \begin{algorithmic}[1]
        \renewcommand{\algorithmicrequire}{\textbf{Input:}}
        \renewcommand{\algorithmicensure}{\textbf{Output:}}
        \REQUIRE Problem \ref{eq:two_stage_ro_general}, tolerance $\epsilon > 0$
        \ENSURE Optimal solution $\bm{x}^\star$ of \ref{eq:two_stage_ro_general}
        \STATE Initialize $LB = -\infty$, $UB = +\infty$, $\bm{x}^\star = \emptyset$, $\mathcal{O} = \emptyset$.
        \REPEAT
        \STATE Set $LB$ and $\hat{\bm{x}}$ as the optimal objective value and solution of problem \eqref{eq:benders_lb_update}, respectively:
        \begin{equation}\label{eq:benders_lb_update}
            \inf_{\bm{x} \in \mathcal{X}} \left\{ \sup_{(\bm{\xi}, \bm{\mu}) \in \mathcal{O}} \left\{\bm{c}(\bm{\xi})^\top\bm{x} + \bm{\mu}^\top\big(\bm{h}(\bm{\xi}) - \bm{T}\bm{x}\big)\right\} 
            \right\}
        \end{equation}
        \label{algo:benders:lb-update}
        \STATE Compute 
        $
        \hat{\bm{\xi}} \in \mathop{\arg\max}\limits_{\bm{\xi} \in \Xi} \mathcal{Q}(\hat{\bm{x}}, \bm{\xi})
        $
        \label{algo:benders:worst-case-computation}
        \STATE Update $\mathcal{O}\gets \mathcal{O} \cup \{(\hat{\bm{\xi}}, \hat{\bm{\mu}})\}$, where
        $
        \hat{\bm{\mu}} \in \mathop{\arg\max}\limits_{\bm{\mu} \in \mathbb{R}^m_{+}} \left\{
        \big(\bm{h}(\hat{\bm{\xi}}) - \bm{T}\hat{\bm{x}}\big)^\top \bm{\mu} :
        \bm{W}^\top \bm{\mu} \leq \bm{d}(\hat{\bm{\xi}})
        \right\}
        $
        \IIf {$UB > \mathcal{Q}(\hat{\bm{x}}, \hat{\bm{\xi}})$}
        update
        $UB \gets \mathcal{Q}(\hat{\bm{x}}, \hat{\bm{\xi}})$ and
        $\bm{x}^\star \gets \hat{\bm{x}}$.
        \EndIIf
        \UNTIL {$UB - LB \leq \epsilon$}
    \end{algorithmic}
\end{algorithm}

\subsubsection{Overview of column-and-constraint generation algorithm}\label{sec:formulation:literature:ccg}

Algorithm \ref{algo:ccg} describes a basic column-and-constraint generation scheme for solving \ref{eq:two_stage_ro_general}
when the relatively complete recourse assumption is satisfied.
The key difference with the Benders algorithm is that instead of %
inner approximating
$\sup_{\bm{\xi} \in \Xi} \mathcal{Q}(\bm{x}, \bm{\xi})$
using dual values,
the algorithm obtains lower bounds by adding second-stage variables $\bm{y}^{(\bm{\xi})} \in \mathcal{Y}$ and constraints $\bm{T}\bm{x} + \bm{W} \bm{y}^{(\bm{\xi})} \geq \bm{h}(\bm{\xi})$ for every worst-case parameter realization $\bm{\xi}$ recorded in the set $\mathcal{R}$ (see line~\ref{algo:ccg:lb-update}).
Notably, this algorithm is agnostic to the presence of %
mixed-integer recourse decisions, as long as the worst-case parameter realizations $\hat{\bm{\xi}}$ can be efficiently computed in line~\ref{algo:ccg:worst-case-computation}.
In particular, if this is done in finitely many iterations, then Algorithm~\ref{algo:ccg} also terminates finitely since $\Xi$ is a finite set.

\begin{algorithm}[!htbp]
    \caption{Basic column-and-constraint generation scheme to solve \ref{eq:two_stage_ro_general}
        when the relatively complete recourse assumption is satisfied}
    \label{algo:ccg}
    \begin{algorithmic}[1]
        \renewcommand{\algorithmicrequire}{\textbf{Input:}}
        \renewcommand{\algorithmicensure}{\textbf{Output:}}
        \REQUIRE Problem \ref{eq:two_stage_ro_general}, tolerance $\epsilon > 0$
        \ENSURE Optimal solution $\bm{x}^\star$ of \ref{eq:two_stage_ro_general}
        \STATE Initialize $LB = -\infty$, $UB = +\infty$, $\bm{x}^\star = \emptyset$, $\mathcal{R} = \emptyset$.
        \REPEAT
        \STATE Set $LB$ and $\hat{\bm{x}}$ as the optimal objective value and (projected) solution of problem \eqref{eq:ccg_lb_update}:
        \begin{equation}\label{eq:ccg_lb_update}
            \inf_{\bm{x}, \eta, \bm{y}} \left\{
            \eta :
            \begin{aligned}
                & \bm{x} \in \mathcal{X}, \;\; \eta \in \mathbb{R}, \;\; \bm{y}^{(\bm{\xi})} \in \mathcal{Y}, \; \bm{\xi} \in \mathcal{R}, \\
                & \eta \geq \bm{c}(\bm{\xi})^\top\bm{x} + \bm{d}(\bm{\xi})^\top\bm{y}^{(\bm{\xi})}, \; \bm{\xi} \in \mathcal{R}, \\
                & \bm{T}\bm{x} + \bm{W} \bm{y}^{(\bm{\xi})} \geq \bm{h}(\bm{\xi}), \; \bm{\xi} \in \mathcal{R}.
            \end{aligned}
            \right\}
        \end{equation}
        \label{algo:ccg:lb-update}
        \STATE Compute 
        $
        \hat{\bm{\xi}} \in \mathop{\arg\max}\limits_{\bm{\xi} \in \Xi} \mathcal{Q}(\hat{\bm{x}}, \bm{\xi})
        $
        \label{algo:ccg:worst-case-computation}
        \STATE Update $\mathcal{R} \gets \mathcal{R} \cup \{\hat{\bm{\xi}}\}$.
        \IIf {$UB > \mathcal{Q}(\hat{\bm{x}}, \hat{\bm{\xi}})$}
        update
        $UB \gets \mathcal{Q}(\hat{\bm{x}}, \hat{\bm{\xi}})$ and
        $\bm{x}^\star \gets \hat{\bm{x}}$.
        \EndIIf
        \UNTIL {$UB - LB \leq \epsilon$}
    \end{algorithmic}
\end{algorithm}

When the second-stage decisions are continuous ($\mathcal{Y} = \mathbb{R}^{n_2}_{+}$), %
the worst-case parameter realizations in line~\ref{algo:ccg:worst-case-computation} can be computed by solving the duality-based formulation \eqref{eq:worst_case_problem_duality} similar to the Benders scheme.
However, the original paper on column-and-constraint generation \cite{zhao2012exact} proposes an alternative formulation using the Karush-Kuhn-Tucker (KKT) conditions of the second-stage problem: %
\begin{align}
    &\sup_{\bm{\xi} \in \Xi} \mathcal{Q}(\bm{x}, \bm{\xi}) \notag \\
    =&
    \sup_{\bm{\xi} \in \Xi} \bm{c}(\bm{\xi})^\top \bm{x}  + \inf_{\bm{y} \in \mathcal{Y}}
    \left\{\bm{d}(\bm{\xi})^\top \bm{y}
    : \bm{W} \bm{y} \geq \bm{h}(\bm{\xi}) - \bm{T}\bm{x}
    \right\} \notag \\
    =& \sup_{\bm{\xi}, \bm{y}, \bm{\mu}} %
    \left\{
    \bm{c}(\bm{\xi})^\top \bm{x}  + \bm{d}(\bm{\xi})^\top \bm{y} :
    \begin{aligned}
        &\bm{\xi} \in \Xi, \; \bm{y} \in \mathcal{Y}, \; \bm{\mu} \in \mathbb{R}^m_{+}\\
        &\bm{T}\bm{x} + \bm{W} \bm{y} \geq \bm{h}(\bm{\xi}) \\
        &\bm{W}^\top \bm{\mu} \leq \bm{d}_0 \\
        &\bm{\mu} \circ \left[\bm{W} \bm{y} + \bm{T}\bm{x} - \bm{h}(\bm{\xi})\right] = \bm{0} \\
        &\bm{y} \circ \left[\bm{W}^\top \bm{\mu} - \bm{d}_0\right] = \bm{0}
    \end{aligned}
    \right\}. \label{eq:worst_case_problem_kkt} \tag{$\mathcal{WC}_\mathcal{K}$}
\end{align}
Here, we use $\bm{u} \circ \bm{v} = \left(u_1 v_1, \ldots, u_n v_n\right)$ to denote the Hadamard product between two vectors $\bm{u}, \bm{v} \in \mathbb{R}^n$, and the subscript $\mathcal{K}$ in \eqref{eq:worst_case_problem_kkt} stands for ``KKT''.
When the second-stage decisions are mixed-integer, \cite{zeng2013solving} have proposed a so-called inner-level column-and-constraint generation algorithm to compute worst-case parameter realizations.
However, this algorithm also requires the solution of a %
problem similar to \eqref{eq:worst_case_problem_duality} or \eqref{eq:worst_case_problem_kkt} for fixed values of the discrete second-stage decisions.
We elaborate on this further in Section~\ref{sec:algorithm}.

\subsubsection{Solution of the worst-case parameter realization problems}\label{sec:formulation:literature:bilinear}
Both the Benders decomposition and column-and-constraint generation algorithms require the solution of either the duality-based \eqref{eq:worst_case_problem_duality} or KKT-based \eqref{eq:worst_case_problem_kkt} optimization problems to compute worst-case parameter realizations (see line~\ref{algo:benders:worst-case-computation} of Algorithms \ref{algo:benders} and \ref{algo:ccg}).
These are mixed-integer bilinear problems and in practice, they are reformulated as MILP problems using one of two approaches.

In the first approach, they are reformulated using additional variables and indicator constraints.
For example, suppose that $\bm{h}(\bm{\xi}) = \bm{h}_0 + \bm{H} \bm{\xi}$.
Then, by introducing $\bm{\Omega} \in \mathbb{R}^{m \times n_p}$, the bilinear term appearing in the objective function of \eqref{eq:worst_case_problem_duality} can be reformulated as follows:
\begin{equation*}
    \bm{\mu}^\top \bm{h}(\bm{\xi})
    =
    \bm{\mu}^\top \bm{h}_0 +
    \tr(\bm{H}^\top \bm{\Omega})
\end{equation*}
where $\tr(\bm{A})$ denotes the trace of a square matrix $A$.
Additional indicator constraints enforcing the definition $\bm{\Omega} = \bm{\mu} \bm{\xi}^\top$ can be modeled as follows:
\begin{alignat*}{3}
    \xi_j = 0 &\implies \Omega_{ij} = 0, &\; (i, j) \in [m] \times [n_p], \\
    \xi_j = 1 &\implies \Omega_{ij} = \mu_i, &\; (i, j) \in [m] \times [n_p].
\end{alignat*}
The $\bm{d}(\bm{\xi})^\top \bm{y}$ term in the objective function of formulation \eqref{eq:worst_case_problem_kkt} can be reformulated similarly by adding $n_2\cdot n_p$ new variables and constraints.
The complementary slackness condition
$\bm{\mu} \circ \left[\bm{W} \bm{y} + \bm{T}\bm{x} - \bm{h}(\bm{\xi})\right] = \bm{0}$
in \eqref{eq:worst_case_problem_kkt}
can be reformulated by introducing $\bm{\alpha} \in \{0, 1\}^m$ and the indicator constraints:
\begin{alignat*}{3}
    \alpha_i = 0 &\implies \mu_i = 0, &\; i \in [m], \\
    \alpha_i = 1 &\implies \one_i^\top \left[\bm{W} \bm{y} + \bm{T}\bm{x} - \bm{h}(\bm{\xi})\right] = 0, &\; i \in [m].
\end{alignat*}
The condition
$\bm{y} \circ \left[\bm{W}^\top \bm{\mu} - \bm{d}(\bm{\xi})\right] = \bm{0}$
can be reformulated similarly by adding $n_2$ additional binary variables.

In the second approach, the mixed-integer bilinear problems are reformulated as MILP problems by using the same additional variables as in the indicator approach but the bilinear terms are linearized using big-M constants.
For example, we can model $\bm{\Omega} = \bm{\mu} \bm{\xi}^\top$ in \eqref{eq:worst_case_problem_duality} as follows:
\begin{equation*}
    \max\{\mu_i - M (1 - \xi_j), 0\} \leq \Omega_{ij} \leq \min\{M \xi_j, \mu_i\}, \; (i, j) \in [m] \times [n_p],
\end{equation*}
where $M$ is a valid upper bound on $\mu_i$ (which may not always be available or computable).
Similarly, the complementary slackness condition
$\bm{\mu} \circ [\bm{W} \bm{y} + \bm{T}\bm{x} - \bm{h}(\bm{\xi})] = \bm{0}$
in \eqref{eq:worst_case_problem_kkt}
can be linearized by adding $2m$ new constraints:
\begin{equation*}
    \bm{\mu} \leq M\bm{\alpha},
    \qquad
    \bm{W} \bm{y} + \bm{T}\bm{x} - \bm{h}(\bm{\xi}) \leq M(\bm{e} - \bm{\alpha}),
\end{equation*}
and
$\bm{y} \circ \left[\bm{W}^\top \bm{\mu} - \bm{d}(\bm{\xi})\right] = \bm{0}$
can be linearized similarly with $2n_2$ constraints.

\subsection{Contributions}\label{sec:formulation:contributions}
We present a new alternative method to compute worst-case parameter realizations %
that can be integrated in either the Benders or column-and-constraint generation algorithms. %
The method uses a Lagrangian dual reformulation that moves all uncertainties from the constraints to the objective function using a single (i.e., scalar-valued) and finite Lagrange multiplier.
The optimal value of this multiplier can be generically and efficiently computed for instances of both \ref{eq:two_stage_ro_general} and \ref{eq:two_stage_ro_indicator}.
Notably, our method does not suffer from the computational issues which plague the numerical solution of the aforementioned indicator or linearized MILP reformulations that are used in existing implementations.
Specifically, we highlight the following points.
\begin{enumerate}
    \item The proposed method uses a Lagrangian dual formulation that adds fewer variables and constraints compared to existing methods; see Table \ref{table:formulation_sizes} for a comparison when the second-stage decisions are continuous %
    and the relatively complete recourse assumption is satisfied.
    We note, however, that the method can also address the absence of relatively complete recourse, presence of mixed-integer recourse, and uncertain objective functions.
    \begin{table}[!htbp]
        \centering
        \caption{Formulation sizes of methods that compute worst-case parameter realizations when $\mathcal{Y} = \mathbb{R}^{n_2}_{+}$ %
            and the relatively complete recourse assumption is satisfied (excluding decision variables for $\bm{\xi}$, constraints in $\Xi$, and variable bounds).}
        \label{table:formulation_sizes}
        \begin{tabularx}{\textwidth}{lCCC}
            \toprule
            Method & Binary variables & Continuous variables & Constraints \\
            \midrule
            \ref{eq:worst_case_problem_kkt} (linearized) & $m+n_2$ & $m+n_2+n_2n_p$ & $3m+3n_2+3n_2n_p$ \\
            \ref{eq:worst_case_problem_kkt} (indicator) & $m+n_2$ & $m+n_2+n_2n_p$ & $3m+3n_2+2n_2n_p$ \\
            \ref{eq:worst_case_problem_duality} (linearized) & $-$ & $m + mn_p$ & $3mn_p+n_2$ \\
            \ref{eq:worst_case_problem_duality} (indicator) & $-$ & $m + mn_p$ & $2mn_p+n_2$ \\
            Proposed (problem \ref{eq:two_stage_ro_general}) & $-$ & $m+n_p$ & $m+n_2$ \\
            Proposed (problem \ref{eq:two_stage_ro_indicator}) & $-$ & $m$ & $n_2$ \\
            \bottomrule
        \end{tabularx}
    \end{table}
    
    \item The optimal Lagrange multiplier can be generically and efficiently computed for any instance of \ref{eq:two_stage_ro_general} and \ref{eq:two_stage_ro_indicator}. This is in contrast to the linearized approach where decision-independent %
    big-M upper bounds on the dual variables $\bm{\mu}$ are typically unknown.
    Indeed, the use of arbitrary big-M values here can be problematic: on the one hand, small values can restrict the true feasible region of $\{\mathcal{Q}(\bm{x}, \bm{\xi}) : \bm{\xi} \in \Xi\}$ and hence relax the two-stage problem, leading to suboptimal decisions; on the other hand, large values can cause numerical issues.
    The indicator approach circumvents these difficulties, but its solution can require higher computational effort.
    
    \item Since our method moves all uncertainty to the objective function, it can potentially enable existing methods \cite{kammerling2020oracle,arslan2021decomposition} that address uncertainty only in the objective to solve also problems with uncertainty-affected constraints.
    
    \item We provide evidence of the computational benefits of our proposed method over existing methods via numerical experiments on network design, facility location, and staff rostering problems that have been previously studied in the literature. We provide open source access to our code.
\end{enumerate}

The rest of this paper is organized as follows. Section~\ref{sec:reformulation} presents the Lagrangian dual formulation and the main analytical results that will lay the foundation for Section~\ref{sec:algorithm}, where we describe our proposed algorithms for computing worst-case parameter realizations and how they can be integrated in the Benders or column-and-constraint generation algorithms.
Section~\ref{sec:experiments} discusses numerical experiments and the associated findings.
Finally, Section~\ref{sec:conclusions} offers concluding remarks and directions for future work.

\section{Lagrangian Dual Formulation}\label{sec:reformulation}

Throughout the paper, we make the following assumptions about \ref{eq:two_stage_ro_general} and \ref{eq:two_stage_ro_indicator}.
\begin{enumerate}[label=\textbf{(A\arabic*)},leftmargin=*]
\item\label{assume:sufficiently_expensive_recourse_general} 
Problem \ref{eq:two_stage_ro_general} satisfies
$
\inf \big\{\bm{d}(\bm{\xi})^\top \bm{y} : \bm{y} \in \mathcal{Y} \big\} > -\infty
$
for all $\bm{\xi} \in \Xi$.

\item\label{assume:sufficiently_expensive_recourse_indicator} 
Problem \ref{eq:two_stage_ro_indicator} satisfies
$
\inf \big\{\bm{d}(\bm{\xi})^\top \bm{y} : \bm{y} \in \mathcal{Y}, \; \bm{g}(\bm{x}, \bm{y}) \geq \bm{0} \big\} > -\infty
$
for all $\bm{\xi} \in \Xi$ and $\bm{x} \in \mathcal{X}$.
\end{enumerate}
These assumptions essentially imply that both \ref{eq:two_stage_ro_general} and \ref{eq:two_stage_ro_indicator} have sufficiently expensive recourse, and they are always satisfied whenever $\mathcal{Y}$ is compact or when the objective function is bounded from below by some nominal value.
They mostly serve to simplify exposition and to eliminate pathological cases that don't typically arise in applications.
We emphasize that we do not assume relatively complete recourse for either problem.

\subsection{Main results}

Our central idea relies on constructing certain Lagrangian and penalty functions of the second-stage problems.
For problem \ref{eq:two_stage_ro_general}, we define the Lagrangian function
$
\mathcal{L} : \mathcal{X} \times \{0, 1\}^{n_p} \times \mathbb{R}_{+} \to \mathbb{R} \cup \{+\infty\}
$
and a corresponding penalty function
$
\phi : [0, 1]^{n_p} \times \{0, 1\}^{n_p} \to \mathbb{R}
$
as follows:
\begin{align}
    \mathcal{L}(\bm{x}, \bm{\xi}, \lambda)
    &=\left[
    \begin{aligned}
        \mathop{\text{minimize}}_{\bm{y} \in \mathcal{Y}, \bm{z} \in \mathbb{R}^{n_p}_{+}} \;\; &
        \bm{c}(\bm{\xi})^\top \bm{x}  + \bm{d}(\bm{\xi})^\top \bm{y} + \lambda \phi(\bm{z}, \bm{\xi}) \\
        \text{subject to} \;\; & \bm{T}\bm{x} + \bm{W} \bm{y} \geq \bm{h}(\bm{z}), \;\; \bm{z} \leq \bm{e}.
    \end{aligned}
    \right]
    \label{eq:lagrangian_general} \\
    \phi(\bm{z}, \bm{\xi})
    &= \one^\top \bm{z} + \one^\top \bm{\xi} - 2 \bm{z}^\top \bm{\xi}.
    \label{eq:penalty_general}
\end{align}
Similarly, for problem \ref{eq:two_stage_ro_indicator}, we define the Lagrangian
$
\mathcal{L}_\mathcal{I} : \mathcal{X} \times \{0, 1\}^{n_p} \times \mathbb{R}_{+} \to \mathbb{R} \cup \{+\infty\}
$
and its penalty function
$
\phi_\mathcal{I} : \mathcal{X} \times \mathcal{Y} \times \{0, 1\}^{n_p} \to \mathbb{R}
$
as follows:
\begin{align}
    \mathcal{L}_\mathcal{I}(\bm{x}, \bm{\xi}, \lambda)
    &=\left[
    \begin{aligned}
        \mathop{\text{minimize}}_{\bm{y} \in \mathcal{Y}} \;\; &
        \bm{c}(\bm{\xi})^\top \bm{x}  + \bm{d}(\bm{\xi})^\top \bm{y} + \lambda \phi_\mathcal{I}(\bm{x}, \bm{y}, \bm{\xi}) \\
        \text{subject to} \;\; & \bm{g}(\bm{x}, \bm{y}) \geq \bm{0}.
    \end{aligned}
    \right]
    \label{eq:lagrangian_indicator} \\
    \phi_\mathcal{I}(\bm{x}, \bm{y}, \bm{\xi})
    &= \sum_{j \in [n_p]} \sum_{i \in \mathcal{I}_j^1} \xi_j g_i(\bm{x}, \bm{y}) + \sum_{j \in [n_p]}\sum_{i \in \mathcal{I}_j^0} (1 - \xi_j) g_i(\bm{x}, \bm{y}).
    \label{eq:penalty_indicator}
\end{align}

Observe that the uncertain parameters $\bm{\xi}$ appear only in the objective function of $\mathcal{L}$ and $\mathcal{L}_\mathcal{I}$ via the penalty functions $\phi$ and $\phi_\mathcal{I}$, respectively.
In fact, the latter satisfy some desirable properties that explain why $\mathcal{L}$ and $\mathcal{L}_\mathcal{I}$ constitute Lagrangian relaxations of the second-stage problems, $\mathcal{Q}$ and $\mathcal{Q}_\mathcal{I}$, respectively.

\begin{lemma}[Properties of the penalty functions]\label{lemma:penalty_function_properties}
    For any $\bm{x} \in \mathcal{X}$ and $\bm{\xi} \in \Xi$, the penalty functions
    $\phi$ and $\phi_\mathcal{I}$
    satisfy the following properties.
    \begin{enumerate}
        \item $\phi(\bm{z}, \bm{\xi}) \geq 0$ for all $z \in [0, 1]^{n_p}$ and $\phi(\bm{z}, \bm{\xi}) = 0$ if and only if $\bm{z} = \bm{\xi}$.
        \item If $(\bm{y}, \bm{z})$ is feasible in problem~\eqref{eq:lagrangian_general} with $\phi(\bm{z}, \bm{\xi}) = 0$,
        then $\bm{y}$ is also feasible in problem $\mathcal{Q}(\bm{x}, \bm{\xi})$.
        \item If $\bm{y}$ is feasible in problem~\eqref{eq:lagrangian_indicator},
        then
        $\phi_\mathcal{I}(\bm{x}, \bm{y}, \bm{\xi}) \geq 0$ and $\phi_\mathcal{I}(\bm{x}, \bm{y}, \bm{\xi}) = 0$ if and only if $\bm{y}$ is feasible in problem $\mathcal{Q}_\mathcal{I}(\bm{x}, \bm{\xi})$.
    \end{enumerate}
\end{lemma}
\begin{proof}
    The first statement follows from the following observation:
    \[
    \phi(\bm{z}, \bm{\xi})
    = \one^\top \bm{z} + \one^\top \bm{\xi} - 2 \bm{z}^\top \bm{\xi}
    \geq 
    \bm{z}^\top \bm{z} + \bm{\xi}^\top \bm{\xi} - 2 \bm{z}^\top \bm{\xi}
    =
    \|\bm{z} - \bm{\xi}\|_2^2
    \geq 0,
    \]
    where the inequality follows from $\bm{z}, \bm{\xi} \in [0, 1]^{n_p}$.
    
    The second statement is a consequence of the first, because under the stated conditions, we have $\bm{T}\bm{x} + \bm{W} \bm{y} \geq \bm{h}(\bm{z})$ as well as $\bm{z} = \bm{\xi}$, which taken together imply that $\bm{T}\bm{x} + \bm{W} \bm{y} \geq \bm{h}(\bm{\xi})$; that is, $\bm{y}$ is feasible in problem $\mathcal{Q}(\bm{x}, \bm{\xi})$.
    
    Finally, the third statement follows directly from the observation that each summand in the expression defining $\phi_\mathcal{I}$ is non-negative by construction; therefore, it is equal to $0$ if and only if for each $j \in [n_p]$ and $i \in \mathcal{I}_j^1$: either $\xi_j = 0$ or $g_i(\bm{x}, \bm{y}) = 0$, which is equivalent to the first implication in problem $\mathcal{Q}_\mathcal{I}(\bm{x}, \bm{\xi})$; and 
    for each $j \in [n_p]$ and $i \in \mathcal{I}_j^0$: either $(1 - \xi_j) = 0$ or $g_i(\bm{x}, \bm{y}) = 0$, which is equivalent to the second implication in problem $\mathcal{Q}_\mathcal{I}(\bm{x}, \bm{\xi})$.
\end{proof}

\begin{lemma}[Weak duality]\label{lemma:weak_duality}
    For any $\bm{x} \in \mathcal{X}$ and $\bm{\xi} \in \Xi$, we have:
    \begin{gather*}
        \mathcal{Q}(\bm{x}, \bm{\xi}) \geq \sup_{\lambda \in \mathbb{R}_{+}} \mathcal{L}(\bm{x}, \bm{\xi}, \lambda) \\
        \mathcal{Q}_\mathcal{I}(\bm{x}, \bm{\xi}) \geq \sup_{\lambda \in \mathbb{R}_{+}} \mathcal{L}_\mathcal{I}(\bm{x}, \bm{\xi}, \lambda)
    \end{gather*}
\end{lemma}
\begin{proof}
    One can readily verify that if $\bm{y}$ is feasible in $\mathcal{Q}(\bm{x}, \bm{\xi})$, then $(\bm{y}, \bm{\xi})$ is feasible in $\mathcal{L}(\bm{x}, \bm{\xi}, \lambda)$ for any $\lambda \in \mathbb{R}_{+}$.
    Moreover, it has the same objective value since $\phi(\bm{\xi}, \bm{\xi}) = 0$. Therefore, $\mathcal{L}(\bm{x}, \bm{\xi}, \lambda)$ is a relaxation of $\mathcal{Q}(\bm{x}, \bm{\xi})$ for all $\lambda \in \mathbb{R}_{+}$, which proves the first inequality.
    To prove the second inequality, note that if $\bm{y}$ is feasible in $\mathcal{Q}_\mathcal{I}(\bm{x}, \bm{\xi})$, then it is clearly also feasible in $\mathcal{L}_\mathcal{I}(\bm{x}, \bm{\xi}, \lambda)$ with the same objective value due to the third part of Lemma~\ref{lemma:penalty_function_properties}.
\end{proof}

Our first main result establishes that the Lagrangian functions $\mathcal{L}$ and $\mathcal{L}_\mathcal{I}$ satisfy a strong duality property, even though the second-stage feasible region $\mathcal{Y}$ may be a non-convex MILP representable set.
\begin{theorem}[Strong duality]\label{theorem:strong_duality}
    For any $\bm{x} \in \mathcal{X}$ and $\bm{\xi} \in \Xi$, we have:
    \begin{gather}
        \mathcal{Q}(\bm{x}, \bm{\xi}) = \sup_{\lambda \in \mathbb{R}_{+}} \mathcal{L}(\bm{x}, \bm{\xi}, \lambda)
        \label{eq:strong_duality_general}
        \\
        \mathcal{Q}_\mathcal{I}(\bm{x}, \bm{\xi}) = \sup_{\lambda \in \mathbb{R}_{+}} \mathcal{L}_\mathcal{I}(\bm{x}, \bm{\xi}, \lambda)
        \label{eq:strong_duality_indicator}
    \end{gather}
\end{theorem}
\begin{proof}
    Denote the finite (possibly empty) set of extreme points of the feasible region of
    problem~\eqref{eq:lagrangian_general} %
    by 
    $
    \mathcal{V} = \mathop{\mathrm{ext}} \big\{
    (\bm{y}, \bm{z}) \in \mathcal{Y}\times [0, 1]^{n_p}:
    \bm{T}\bm{x} + \bm{W} \bm{y} \geq \bm{h}(\bm{z})
    \big\}
    $.
    Under assumption~\ref{assume:sufficiently_expensive_recourse_general}, problem~\eqref{eq:lagrangian_general} is bounded from below and therefore, it must attain its optimal value at an extreme point:
    \begin{align}\label{eq:temp:strong_duality}
        \mathcal{L}(\bm{x}, \bm{\xi}, \lambda) 
        &=
        \inf \left\{
        \bm{c}(\bm{\xi})^\top \bm{x}  + \bm{d}(\bm{\xi})^\top \bm{y} + \lambda \phi(\bm{z}, \bm{\xi})
        :
        (\bm{y}, \bm{z}) \in \mathcal{V}
        \right\}, \;\; \forall \lambda \in \mathbb{R}_{+}.
    \end{align}

    First, suppose that $\mathcal{Q}(\bm{x}, \bm{\xi}) = +\infty$.
    If $\mathcal{V} = \emptyset$, then $\mathcal{L}(\bm{x}, \bm{\xi}, \lambda) = +\infty$ as well and we are done;
    therefore, assume $\mathcal{V} \neq \emptyset$.
    Observe that we must have
    $\phi(\bm{z}, \bm{\xi}) > 0$ for all $(\bm{y}, \bm{z}) \in \mathcal{V}$; indeed, if $\phi(\bm{z}, \bm{\xi}) = 0$, then the second statement of Lemma~\ref{lemma:penalty_function_properties} implies that $\bm{y}$ is feasible in problem $\mathcal{Q}(\bm{x}, \bm{\xi})$, contradicting $\mathcal{Q}(\bm{x}, \bm{\xi}) = +\infty$.
    Therefore, $\hat{\phi} \coloneqq \min\{\phi(\bm{z}, \bm{\xi}) : (\bm{y}, \bm{z}) \in \mathcal{V}\} > 0$ and equation \eqref{eq:temp:strong_duality} implies:
    \begin{equation}\label{eq:temp:2:strong_duality}
    \mathcal{L}(\bm{x}, \bm{\xi}, \lambda) 
    \geq
    \min \left\{
    \bm{c}(\bm{\xi})^\top \bm{x}  + \bm{d}(\bm{\xi})^\top \bm{y}
    :
    (\bm{y}, \bm{z}) \in \mathcal{V}
    \right\}
    + \lambda \hat{\phi} \;\; \forall \lambda \in \mathbb{R}_{+}.
    \end{equation}
    The right-hand side of the above inequality goes to $+\infty$ as $\lambda \to +\infty$ and therefore, equation~\eqref{eq:strong_duality_general} is proved in the case when $\mathcal{Q}(\bm{x}, \bm{\xi}) = +\infty$.
    
    Now, suppose that $\mathcal{Q}(\bm{x}, \bm{\xi}) < +\infty$.
    Then, Lemma~\ref{lemma:weak_duality} implies $\mathcal{L}(\bm{x}, \bm{\xi}, \lambda) < +\infty$ for all $\lambda \in \mathbb{R}_{+}$.
    Therefore, equation~\eqref{eq:temp:strong_duality} implies that $\mathcal{V} \neq \emptyset$ and
    we can define
    $
    \hat{\mathcal{V}}(\lambda) = \mathop{\arg\min}_{(\bm{y}, \bm{z}) \in \mathcal{V}} \big\{
    \bm{c}(\bm{\xi})^\top \bm{x}  + \bm{d}(\bm{\xi})^\top \bm{y} + \lambda \phi(\bm{z}, \bm{\xi})
    \big\}
    $
    to be the set of extreme points that achieve the minimum on the right-hand side of~\eqref{eq:temp:strong_duality}.
    
    We now claim that there exists some finite $\bar{\lambda} > 0$ such that
    for all $\lambda \geq \bar{\lambda}$,
    we have
    $\phi(\bm{z}, \bm{\xi}) = 0$
    for all $(\bm{y}, \bm{z}) \in \hat{\mathcal{V}}(\lambda)$.
    Suppose this was not the case; that is,
    for all $\lambda \in \mathbb{R}_{+}$,
    there exists some 
    $(\hat{\bm{y}}, \hat{\bm{z}}) \in \hat{\mathcal{V}}(\lambda)$
    such that $\phi(\hat{\bm{z}}, \bm{\xi}) > 0$.
    Then, equation~\eqref{eq:temp:strong_duality} implies that
    \[
    \mathcal{L}(\bm{x}, \bm{\xi}, \lambda) 
    \geq
    \min \left\{
    \bm{c}(\bm{\xi})^\top \bm{x}  + \bm{d}(\bm{\xi})^\top \bm{y}
    :
    (\bm{y}, \bm{z}) \in \mathcal{V}
    \right\}
    + \lambda \phi(\hat{\bm{z}}, \bm{\xi}) \;\; \forall \lambda \in \mathbb{R}_{+}.
    \]
    The right-hand side of this inequality goes to $+\infty$ as $\lambda \to +\infty$ and therefore, it implies that $\sup_{\lambda \in \mathbb{R}_{+}} \mathcal{L}(\bm{x}, \bm{\xi}, \lambda) = +\infty$. But this contradicts $\mathcal{L}(\bm{x}, \bm{\xi}, \lambda) < +\infty$.
    Therefore, we have established that for all $\lambda \geq \bar{\lambda}$,
    we have
    $\phi(\bm{z}, \bm{\xi}) = 0$
    for all $(\bm{y}, \bm{z}) \in \hat{\mathcal{V}}(\lambda)$.
    The second part of Lemma~\ref{lemma:penalty_function_properties} then implies that $\bm{y}$ is feasible in problem $\mathcal{Q}(\bm{x}, \bm{\xi})$ for all $(\bm{y}, \bm{z}) \in \hat{\mathcal{V}}(\lambda)$.
    Taken together with equation~\eqref{eq:temp:strong_duality}, these statements imply that:
    \begin{align*}
        \mathcal{L}(\bm{x}, \bm{\xi}, \lambda) 
        &=
        \min \left\{
        \bm{c}(\bm{\xi})^\top \bm{x}  + \bm{d}(\bm{\xi})^\top \bm{y} + \lambda \phi(\bm{z}, \bm{\xi})
        :
        (\bm{y}, \bm{z}) \in \mathcal{V}
        \right\} \;\; \forall \lambda \geq \bar{\lambda}
        \\
        &=
        \min \left\{
        \bm{c}(\bm{\xi})^\top \bm{x}  + \bm{d}(\bm{\xi})^\top \bm{y}
        :
        (\bm{y}, \bm{z}) \in \hat{\mathcal{V}}(\lambda)
        \right\} \;\; \forall \lambda \geq \bar{\lambda}
        \\
        &\geq
        \mathcal{Q}(\bm{x}, \bm{\xi}).
    \end{align*}
    Together with Lemma~\ref{lemma:weak_duality}, this proves \eqref{eq:strong_duality_general} when $\mathcal{Q}(\bm{x}, \bm{\xi}) < +\infty$ as well.
    
    Equation~\eqref{eq:strong_duality_indicator} can be proved using exactly the same arguments, by replacing $\mathcal{V}$ with $\mathcal{V}_\mathcal{I}$ and $\phi$ with $\phi_\mathcal{I}$. We omit details in the interest of brevity.
\end{proof}

Our second main result establishes an even stronger duality property than Theorem~\ref{theorem:strong_duality}.
Whereas the latter shows that the maximization over $\lambda$ can be moved outside the inner minimization over the second-stage decisions, the following theorem shows that it can even be moved outside the minimization over the first-stage decisions, without sacrificing optimality.
This will turn out to be crucial for ensuring the validity of the lower bounds and overall correctness of the algorithms proposed in Section~\ref{sec:algorithm}.
\begin{theorem}[Strong duality for the two-stage problem]\label{theorem:strong_duality_two_stage_problem}
    \begin{gather}
        \inf_{\bm{x} \in \mathcal{X}} \sup_{\bm{\xi} \in \Xi} \mathcal{Q}(\bm{x}, \bm{\xi}) = \sup_{\lambda \in \mathbb{R}_{+}} \inf_{\bm{x} \in \mathcal{X}} \sup_{\bm{\xi} \in \Xi} \mathcal{L}(\bm{x}, \bm{\xi}, \lambda)
        \label{eq:strong_duality_general_two_stage}
        \\
        \inf_{\bm{x} \in \mathcal{X}} \sup_{\bm{\xi} \in \Xi} \mathcal{Q}_\mathcal{I}(\bm{x}, \bm{\xi}) = \sup_{\lambda \in \mathbb{R}_{+}} \inf_{\bm{x} \in \mathcal{X}} \sup_{\bm{\xi} \in \Xi} \mathcal{L}_\mathcal{I}(\bm{x}, \bm{\xi}, \lambda)
        \label{eq:strong_duality_indicator_two_stage}
    \end{gather}
\end{theorem}
\begin{proof}
    Suppose first that \ref{eq:two_stage_ro_general} is infeasible, so that the left-hand side of \eqref{eq:strong_duality_general_two_stage} is equal to $+\infty$.
    In other words, for all $\bm{x} \in \mathcal{X}$, there exists some $\bm{\xi} \in \Xi$ such that $\mathcal{Q}(\bm{x}, \bm{\xi}) = +\infty$.
    Using the same argument preceding equation \eqref{eq:temp:2:strong_duality} in the proof of Theorem~\ref{theorem:strong_duality}, we can define
    $
    \mathcal{V}(\bm{x}) \coloneqq \mathop{\mathrm{ext}} \big\{
    (\bm{y}, \bm{z}) \in \mathcal{Y}\times [0, 1]^{n_p}:
    \bm{T}\bm{x} + \bm{W} \bm{y} \geq \bm{h}(\bm{z})
    \big\}
    $ %
    and
    $\hat{\phi}(\bm{x}, \bm{\xi}) \coloneqq \inf \{\phi(\bm{z}, \bm{\xi}) : (\bm{y}, \bm{z}) \in \mathcal{V}(\bm{x})\} > 0$.
    Unlike in the proof of Theorem~\ref{theorem:strong_duality}, we make the dependence on $\bm{x}$ and $\bm{\xi}$ explicit.
    Observe that $\tilde{\phi} \coloneqq \inf_{\bm{x} \in \mathcal{X}} \sup_{\bm{\xi} \in \Xi}  \hat{\phi}(\bm{x}, \bm{\xi})$ is strictly positive since for all $\bm{\xi} \in \Xi$,
    $
    \hat{\phi}({\bm{x}}, \bm{\xi}) > 0
    $
    and finite
    if $\mathcal{V}(\bm{x}) \neq \emptyset$
    or 
    $
    \hat{\phi}({\bm{x}}, \bm{\xi}) = +\infty
    $
    otherwise.
    Now taking $\sup$ over $\bm{\xi}$ and then $\inf$ over $\bm{x}$ in equation~\eqref{eq:temp:2:strong_duality}, we have for all $\lambda \in \mathbb{R}_{+}$ that:
    \[
    \inf_{\bm{x} \in \mathcal{X}} \sup_{\bm{\xi} \in \Xi} \mathcal{L}(\bm{x}, \bm{\xi}, \lambda) 
    \geq
    \inf_{\bm{x} \in \mathcal{X}} \sup_{\bm{\xi} \in \Xi} \inf \left\{
    \bm{c}(\bm{\xi})^\top \bm{x}  + \bm{d}(\bm{\xi})^\top \bm{y}
    :
    (\bm{y}, \bm{z}) \in \mathcal{V}(\bm{x})
    \right\}
    + \lambda \tilde{\phi}.
    \]
    The first term in the right-hand side of the above inequality is equal to $+\infty$ if $\mathcal{V}(\bm{x}) = \emptyset$ for all $\bm{x} \in \mathcal{X}$; otherwise, it is finite, since $\mathcal{X}$ and $\Xi$ are compact.
    In either case, the right-hand side goes to $+\infty$ as $\lambda \to \infty$.
    
    Suppose now that \ref{eq:two_stage_ro_general} is feasible, so that the left-hand side of \eqref{eq:strong_duality_general_two_stage} is finite.
    It suffices to take the minimization over decisions $\bm{x} \in \mathcal{X}$ for which $\sup_{\bm{\xi} \in \Xi} \mathcal{Q}(\bm{x}, \bm{\xi}) < + \infty$.
    In this case, Lemma~\ref{lemma:weak_duality} implies that
    $\sup_{\bm{\xi} \in \Xi} \mathcal{L}(\bm{x}, \bm{\xi}, \lambda) < + \infty$ for all $\lambda \geq 0$, whereas the proof of Theorem~\ref{theorem:strong_duality} establishes that there exists some finite $\bar{\lambda}(\bm{x}, \bm{\xi}) > 0$ (we again make the dependence on $\bm{x}$ and $\bm{\xi}$ explicit here)
    such that for all feasible $\bm{x} \in \mathcal{X}$ and $\bm{\xi} \in \Xi$, the function $\mathcal{L}(\bm{x}, \bm{\xi}, \cdot)$ is maximized by any $\lambda \geq \bar{\lambda}(\bm{x}, \bm{\xi})$.
    By defining $\tilde{\lambda} \coloneqq \max\{\bar{\lambda}(\bm{x}, \bm{\xi}) : \bm{x} \in \mathcal{X}, \bm{\xi} \in \Xi\}$ which is finite since $\mathcal{X}$ and $\Xi$ are compact, we then note that $\tilde{\lambda}$ maximizes the function $\mathcal{L}(\bm{x}, \bm{\xi}, \cdot)$ for all $\bm{x}$ and $\bm{\xi}$.
    In other words, we have:
    \begin{align*}
        \inf_{\bm{x} \in \mathcal{X}} \sup_{\bm{\xi} \in \Xi} \sup_{\lambda \in \mathbb{R}_{+}} \mathcal{L}(\bm{x}, \bm{\xi}, \lambda)
        =  \inf_{\bm{x} \in \mathcal{X}} \sup_{\bm{\xi} \in \Xi} \mathcal{L}(\bm{x}, \bm{\xi}, \tilde{\lambda})
        \leq \sup_{\lambda \in \mathbb{R}_{+}} \inf_{\bm{x} \in \mathcal{X}} \sup_{\bm{\xi} \in \Xi} \mathcal{L}(\bm{x}, \bm{\xi}, \lambda),
    \end{align*}
    where the inequality follows by treating the middle expression as a function that is evaluated at $\lambda = \tilde{\lambda}$ and the right-most expression as maximizing this function over all $\lambda \in \mathbb{R}_{+}$.
    The max-min inequality, however, implies that:
    \begin{align*}
        \inf_{\bm{x} \in \mathcal{X}} \sup_{\bm{\xi} \in \Xi} \sup_{\lambda \in \mathbb{R}_{+}} \mathcal{L}(\bm{x}, \bm{\xi}, \lambda)
        =  \inf_{\bm{x} \in \mathcal{X}} \sup_{\lambda \in \mathbb{R}_{+}} \sup_{\bm{\xi} \in \Xi} \mathcal{L}(\bm{x}, \bm{\xi}, \lambda)
        \geq  \sup_{\lambda \in \mathbb{R}_{+}} \inf_{\bm{x} \in \mathcal{X}} \sup_{\bm{\xi} \in \Xi} \mathcal{L}(\bm{x}, \bm{\xi}, \lambda).
    \end{align*}
    Combining the previous two inequalities with Theorem~\ref{theorem:strong_duality}, we obtain:
    \begin{align*}
        \inf_{\bm{x} \in \mathcal{X}} \sup_{\bm{\xi} \in \Xi} \mathcal{Q}(\bm{x}, \bm{\xi})
        =
        \inf_{\bm{x} \in \mathcal{X}} \sup_{\bm{\xi} \in \Xi} \sup_{\lambda \in \mathbb{R}_{+}} \mathcal{L}(\bm{x}, \bm{\xi}, \lambda)
        = 
        \sup_{\lambda \in \mathbb{R}_{+}} \inf_{\bm{x} \in \mathcal{X}} \sup_{\bm{\xi} \in \Xi} \mathcal{L}(\bm{x}, \bm{\xi}, \lambda),
    \end{align*}
    which proves equation~\eqref{eq:strong_duality_general_two_stage}.
    
    As in the proof of Theorem~\ref{theorem:strong_duality}, equation~\eqref{eq:strong_duality_indicator} can be proved using the exact same arguments, and we again omit details for the sake of brevity. 
\end{proof}

\subsection{Optimal Lagrange multipliers}

Theorems~\ref{theorem:strong_duality} and~\ref{theorem:strong_duality_two_stage_problem} reduce the computation of a worst-case parameter realization from optimizing over the second-stage value function
to optimizing over the Lagrangian dual function.
We now characterize the structure of an optimal Lagrange multiplier for \ref{eq:two_stage_ro_general} and \ref{eq:two_stage_ro_indicator}. %
Before that, however, we elucidate a useful property of the worst-case Lagrangian functions.
\begin{lemma}[Properties of the worst-case Lagrangian functions]\label{lemma:wc_lagrangian_function_properties}
    For any $\bm{x} \in \mathcal{X}$, the functions
    $
    \sup \{ \mathcal{L}(\bm{x}, \bm{\xi}, \cdot) : \bm{\xi} \in \Xi \}
    $
    and
    $
    \sup \{ \mathcal{L}_\mathcal{I}(\bm{x}, \bm{\xi}, \cdot) : \bm{\xi} \in \Xi \}
    $
    are piecewise linear, continuous, and monotonically non-decreasing in $\lambda \in \mathbb{R}_{+}$.
\end{lemma}
\begin{proof}
    Observe that
    \textit{(i)} the objective function of
    $\mathcal{L}(\bm{x}, \bm{\xi}, \lambda)$
    is affine in $\lambda$;
    \textit{(ii)} its optimal solution
    is always attained at any of the finite number of extreme points
    of its feasible region in problem~\eqref{eq:lagrangian_general};
    and therefore,
    \textit{(iii)} the function
    $\sup \{ \mathcal{L}(\bm{x}, \bm{\xi}, \cdot) : \bm{\xi} \in \Xi \}$ is the supremum over a finite set ($\bm{\xi} \in \Xi$) of the infimum of a finite number of affine functions.
    This shows that it is piecewise linear and continuous.

    Its non-decreasing nature follows directly from the non-negativity of the slopes of the linear pieces; that is, of the penalty functions $\phi$ and $\phi_\mathcal{I}$ that was established in Lemma~\ref{lemma:penalty_function_properties}.
    Specifically, consider any $\lambda_1, \lambda_2 \in \mathbb{R}$ such that $\lambda_1 \leq \lambda_2$.
    Then, for any $\bm{y} \in \mathcal{Y}$, $\bm{z} \in [0, 1]^{n_p}$ and $\bm{\xi} \in \Xi$,
    we have
    $
    \bm{c}(\bm{\xi})^\top \bm{x}  + \bm{d}(\bm{\xi})^\top \bm{y} + \lambda_1 \phi(\bm{z}, \bm{\xi})
    \leq 
    \bm{c}(\bm{\xi})^\top \bm{x}  + \bm{d}(\bm{\xi})^\top \bm{y} + \lambda_2 \phi(\bm{z}, \bm{\xi}).
    $
    Therefore, $\sup_{\bm{\xi} \in \Xi} \mathcal{L}(\bm{x}, \bm{\xi}, \lambda_1) \leq \sup_{\bm{\xi} \in \Xi} \mathcal{L}(\bm{x}, \bm{\xi}, \lambda_2)$ implying that it is non-decreasing in $\lambda$.
    The argument for $\sup \{ \mathcal{L}_\mathcal{I}(\bm{x}, \bm{\xi}, \cdot) : \bm{\xi} \in \Xi \}$ is similar.
\end{proof}

A necessary condition for the optimality of a Lagrange multiplier directly falls out of Lemma~\ref{lemma:wc_lagrangian_function_properties}, which we present next.
\begin{theorem}[Necessary conditions for optimality of multipliers]\label{theorem:necessary_optimality_condition}
    \begin{enumerate}
        \item For any %
        $\bm{x} \in \mathcal{X}$ %
        such that
        $\mathcal{Q}(\bm{x}, \bm{\xi})$ is finite for all $\bm{\xi} \in \Xi$,
        we have that
        \begin{equation*}
            \bar{\lambda}
            \in
            \mathop{\arg\max}_{\lambda \in \mathbb{R}_{+}} \left\{
            \sup_{\bm{\xi} \in \Xi} \mathcal{L}(\bm{x}, \bm{\xi}, \lambda)
            \right\}
        \end{equation*}
        only if there exists
        $\bar{\bm{\xi}} \in \arg\max_{\bm{\xi} \in \Xi} \mathcal{L}(\bm{x}, \bm{\xi}, \bar{\lambda})$
        and
        an optimal solution $(\bar{\bm{y}}, \bar{\bm{z}})$ of
        $\mathcal{L}(\bm{x}, \bar{\bm{\xi}}, \bar{\lambda})$
        such that
        $\phi(\bar{\bm{z}}, \bar{\bm{\xi}}) = 0$.
        
        \item For any %
        $\bm{x} \in \mathcal{X}$ %
        such that
        $\mathcal{Q}_\mathcal{I}(\bm{x}, \bm{\xi})$ is finite for all $\bm{\xi} \in \Xi$,
        we have that
        \begin{equation*}
            \bar{\lambda}
            \in
            \mathop{\arg\max}_{\lambda \in \mathbb{R}_{+}} \left\{
            \sup_{\bm{\xi} \in \Xi} \mathcal{L}_\mathcal{I}(\bm{x}, \bm{\xi}, \lambda)
            \right\}
        \end{equation*}
        only if there exists
        $\bar{\bm{\xi}} \in \arg\max_{\bm{\xi} \in \Xi} \mathcal{L}_\mathcal{I}(\bm{x}, \bm{\xi}, \bar{\lambda})$
        and
        an optimal solution $\bar{\bm{y}}$ of
        $\mathcal{L}_\mathcal{I}(\bm{x}, \bar{\bm{\xi}}, \bar{\lambda})$
        such that
        $\phi_\mathcal{I}(\bm{x}, \bar{\bm{y}}, \bar{\bm{\xi}}) = 0$.
    \end{enumerate}
\end{theorem}
\begin{proof}
    Suppose that the stated conditions hold and $\bar{\lambda}$ is an optimal multiplier.
    Theorem~\ref{theorem:strong_duality} implies that $\sup_{\bm{\xi} \in \Xi} \mathcal{L}(\bm{x}, {\bm{\xi}}, \bar{\lambda})$ must be finite.
    Therefore, denote 
    $\hat{\mathcal{U}} \coloneqq \arg\max_{\bm{\xi} \in \Xi} \mathcal{L}(\bm{x}, \bm{\xi}, \bar{\lambda})$
    and for each
    $
    \bm{\xi} \in \hat{\mathcal{U}}
    $,
    denote the set of minimizers of problem $\mathcal{L}(\bm{x}, \bm{\xi}, \bar{\lambda})$ as
    $
    \hat{\mathcal{V}}(\bm{\xi})
    $.
    Suppose for the sake of contradiction, that for all
    ${\bm{\xi}} \in \hat{\mathcal{U}}$
    and for all
    $({\bm{y}}, {\bm{z}}) \in \hat{\mathcal{V}}({\bm{\xi}})$
    we have that
    $\phi({\bm{z}}, {\bm{\xi}}) > 0$.
    Define $\hat{\phi} = \min\{\phi({\bm{z}}, {\bm{\xi}}) : (\bm{y}, \bm{z}) \in \hat{\mathcal{V}}(\bm{\xi}), \bm{\xi} \in \hat{\mathcal{U}}\} > 0$.
    Then, by construction:
    \begin{align*}
        \sup_{\bm{\xi} \in \Xi} \mathcal{L}(\bm{x}, {\bm{\xi}}, \bar{\lambda})
        =
        \max_{\bm{\xi} \in \hat{\mathcal{U}}}
        \min_{(\bm{y}, \bm{z}) \in \hat{\mathcal{V}}(\bm{\xi})}
        \{
        \bm{c}(\bm{\xi})^\top \bm{x}  + \bm{d}(\bm{\xi})^\top \bm{y}
        + 
        \bar{\lambda} \phi({\bm{z}}, {\bm{\xi}})
        \}
    \end{align*}
    Continuity of the function on the left-hand side (with respect to $\lambda$) implies that for some sufficiently small $\epsilon > 0$,
    $\max_{\bm{\xi} \in \Xi} \mathcal{L}(\bm{x}, {\bm{\xi}}, \bar{\lambda} + \epsilon) \geq \max_{\bm{\xi} \in \Xi} \mathcal{L}(\bm{x}, {\bm{\xi}}, \bar{\lambda}) + \epsilon \hat{\phi}$,
    contradicting the optimality of $\bar{\lambda}$.
    The proof of the second statement is similar and omitted for the sake of brevity. 
\end{proof}

Although Theorem~\ref{theorem:necessary_optimality_condition} is only a necessary condition for optimality, we shall show in the next section that we can still use it in conjunction with Theorem~\ref{theorem:strong_duality_two_stage_problem} to design an algorithm for computing the optimal multiplier.
For now, we close this section by providing a sufficient condition under which we can compute a closed-form expression for the optimal Lagrange multiplier in problem \ref{eq:two_stage_ro_general}.
\begin{theorem}[Closed-form expression for optimal Lagrange multiplier]\label{theorem:optimal_multiplier_general}
    Suppose that for every $\bm{x} \in \mathcal{X}$, either there exists $\bm{\xi} \in \Xi$ such that $\mathcal{Q}(\bm{x}, \bm{\xi}) = +\infty$ or $\mathcal{Q}(\bm{x}, \bm{\xi}) < +\infty$ for all $\bm{\xi} \in \{0, 1\}^{n_p}$.
    Then, for any feasible first-stage decision $\bm{x} \in \mathcal{X}$;
    that is,
    for which $\sup \big\{\mathcal{Q}(\bm{x}, \bm{\xi}) : \bm{\xi} \in \Xi \big\} < +\infty$,
    we have that
    \begin{equation*}
        u(\bm{x}) - \ell(\bm{x})
        \in
        \mathop{\arg\max}_{\lambda \in \mathbb{R}_{+}} \left\{
        \max_{\bm{\xi} \in \Xi} \mathcal{L}(\bm{x}, \bm{\xi}, \lambda)
        \right\},
    \end{equation*}
    where $u(\bm{x})$ is any finite upper bound on
    $\sup \big\{\mathcal{Q}(\bm{x}, \bm{\xi}) : \bm{\xi} \in \Xi \big\}$
    and  $\ell(\bm{x})$ is any finite lower bound on
    $\inf \big\{\bm{c}(\bm{\xi})^\top \bm{x} + \bm{d}(\bm{\xi})^\top \bm{y} : \bm{\xi} \in \Xi, \bm{y} \in \mathcal{Y} \big\}$.
\end{theorem}
\begin{proof}
    Suppose that $\bm{x} \in \mathcal{X}$ is any feasible first-stage decision in \ref{eq:two_stage_ro_general}.
    Then, we claim that
    \[
    \mathrm{Proj}_{\bm{z}}
    \left( \left\{
    (\bm{y}, \bm{z}) \in \mathcal{Y}\times [0, 1]^{n_p}:
    \bm{T}\bm{x} + \bm{W} \bm{y} \geq \bm{h}(\bm{z})
    \right\}\right) = [0, 1]^{n_p}.
    \]
    The `$\subseteq$' inclusion is obvious.
    To show the reverse `$\supseteq$' inclusion, we note that the condition stated in the Theorem implies that since $\bm{x}$ is feasible and has a finite second-stage value $\mathcal{Q}(\bm{x}, \bm{\xi}) < +\infty$ for all $\bm{\xi} \in \Xi$, it also has a finite second-stage value for all $\bm{\xi} \in \{0, 1\}^{n_p}$.
    In other words, the set $\{\bm{y} \in \mathcal{Y}:  \bm{W} \bm{y} \geq \bm{h}(\bm{z}) - \bm{T}\bm{x} \}$ is non-empty for any setting of $\bm{z} = \bm{\xi} \in \{0, 1\}^{n_p}$. This proves the claim.
    
    Therefore, problem~\eqref{eq:lagrangian_general} always has an optimal solution $(\hat{\bm{y}}, \hat{\bm{z}})$ such that $\hat{\bm{z}} \in \{0, 1\}^{n_p}$, for any $\bm{\xi} \in \Xi$ and $\lambda \in \mathbb{R}_{+}$ (since its objective is linear in $\bm{z}$).
    In this case, the definition of $\phi(\hat{\bm{z}}, \bm{\xi})$ is equivalent to that of the Hamming distance between two binary vectors and therefore,
    we have
    $\phi(\hat{\bm{z}}, \bm{\xi}) \in \{0, 1, 2, \ldots, n_p\}$.
    
    We are now ready to prove the statement of the theorem.
    In particular, we aim to show that $\mathcal{Q}(\bm{x}, \bm{\xi}) = \mathcal{L}(\bm{x}, \bm{\xi}, \lambda)$ is satisfied for all $\lambda \geq u(\bm{x}) - \ell(\bm{x})$ and all $\bm{\xi} \in \Xi$.
    Note that along with Theorem~\ref{theorem:strong_duality}, this will prove the first part of this theorem.
    Suppose for the sake of contradiction that $\mathcal{Q}(\bm{x}, \hat{\bm{\xi}}) > \mathcal{L}(\bm{x}, \hat{\bm{\xi}}, \hat{\lambda})$ for some $\hat{\lambda} \geq u(\bm{x}) - \ell(\bm{x})$ and some $\hat{\bm{\xi}} \in \Xi$.
    Let $(\hat{\bm{y}}, \hat{\bm{z}})$ be an optimal solution of $\mathcal{L}(\bm{x}, \hat{\bm{\xi}}, \hat{\lambda})$ (i.e., problem~\eqref{eq:lagrangian_general}).
    Then, the second part of Lemma~\ref{lemma:penalty_function_properties} implies that we must have $\phi(\hat{\bm{z}}, \bm{\xi}) > 0$ (otherwise, $\hat{\bm{y}}$ would be feasible in $\mathcal{Q}(\bm{x}, \hat{\bm{\xi}})$ with the same objective value $\mathcal{L}(\bm{x}, \hat{\bm{\xi}}, \hat{\lambda})$ which contradicts the hypothesis).
    But since we have established in the previous paragraph that  $\phi(\hat{\bm{z}}, \bm{\xi}) \in \{0, 1, 2, \ldots, n_p\}$, this implies that we must have
    $\phi(\hat{\bm{z}}, \bm{\xi}) \geq 1$.
    Along with the definition of $\ell(\bm{x})$, this further implies the following (where we note that a finite value for $\ell(\bm{x})$ is guaranteed under assumption~\ref{assume:sufficiently_expensive_recourse_general}):
    \begin{align*}
    \mathcal{L}(\bm{x}, \hat{\bm{\xi}}, \hat{\lambda})
    &= \bm{c}(\hat{\bm{\xi}})^\top \bm{x} + \bm{d}(\hat{\bm{\xi}})^\top \hat{\bm{y}} + \hat{\lambda} \phi(\hat{\bm{z}}, \bm{\xi}) \\
    &\geq \inf_{\bm{\xi} \in \Xi, \; \bm{y} \in \mathcal{Y}} \{\bm{c}(\bm{\xi})^\top \bm{x} + \bm{d}(\bm{\xi})^\top \bm{y} \} + \hat{\lambda} \geq \ell(\bm{x}) + \hat{\lambda}.
    \end{align*}
    But from the hypothesis, and from the definition of $u(\bm{x})$, we have:
    \begin{equation*}
        \mathcal{L}(\bm{x}, \hat{\bm{\xi}}, \hat{\lambda})
        < \mathcal{Q}(\bm{x}, \hat{\bm{\xi}})
        \leq \sup_{\bm{\xi} \in \Xi} \mathcal{Q}(\bm{x}, \bm{\xi})
        \leq u(\bm{x}).
    \end{equation*}
    The last two inequalities together imply that $\hat{\lambda} < u(\bm{x}) - \ell(\bm{x})$ which is the desired contradiction.
\end{proof}

Theorem~\ref{theorem:optimal_multiplier_general} provides a closed-form expression for the optimal Lagrange multiplier of $\mathcal{L}(\bm{x}, \bm{\xi}, \lambda)$ in terms of any known upper bound $u(\bm{x})$ on the worst-case objective value of $\bm{x}$ and any known lower bound $\ell(\bm{x})$ on
the best-case objective value of $\bm{x}$.
The next theorem provides some sufficient conditions under which these quantities can be efficiently computed.
\begin{theorem}[Closed-form expressions for $\ell(\bm{x})$ and $u(\bm{x})$]\label{theorem:closed_form_lu}
    Suppose that the conditions of Theorem~\ref{theorem:optimal_multiplier_general} hold and that $\bm{x} \in \mathcal{X}$ is any feasible first-stage decision in \ref{eq:two_stage_ro_general}.
    Then, the quantities $\ell(\bm{x})$ and $u(\bm{x})$ admit closed-form expressions under any of the following conditions.
    \begin{enumerate}
        
        \item If $\mathcal{Y} \subseteq [{y}^\ell_1, {y}^u_1] \times [{y}_{2}, \bar{y}_{2}] \times \ldots \times [{y}^\ell_{n_2}, {y}^u_{n_2}]$ is compact, or more generally, if for all $\bm{\xi} \in \Xi$, an optimal solution $\hat{\bm{y}}$ of the problem $\mathcal{Q}(\bm{x}, \bm{\xi})$ is known to satisfy
        $\hat{\bm{y}} \in [{\bm{y}}^\ell, {\bm{y}}^u]$ for some ${\bm{y}}^\ell, {\bm{y}}^u \in \mathbb{R}^{n_2}$, then
        \begin{align*}
            \ell(\bm{x}) &= \inf_{\bm{\xi}, \bm{y}} \big\{
            \bm{c}(\bm{\xi})^\top \bm{x}  + \bm{d}(\bm{\xi})^\top {\bm{y}} : \bm{\xi} \in \{0, 1\}^{n_p}, y_i \in \{{{y}}^\ell_i, {{y}}^u_i \}, i \in [n_2]]
            \big \},
            \\
            u(\bm{x}) &= \sup_{\bm{\xi}, \bm{y}} \big\{
            \bm{c}(\bm{\xi})^\top \bm{x}  + \bm{d}(\bm{\xi})^\top {\bm{y}} : \bm{\xi} \in \{0, 1\}^{n_p}, y_i \in \{{{y}}^\ell_i, {{y}}^u_i \}, i \in [n_2]]
            \big \}.
        \end{align*}
    
        \item If the objective function $\bm{d}(\bm{\xi}) = \bm{d}_0 + \bm{D}\bm{\xi}$ satisfies $\bm{D}^\top \bm{y} \geq \bm{0}$ for all $\{ \bm{y} \in \mathcal{Y}: \bm{T}\bm{x} + \bm{W} \bm{y} \geq \bm{h}(\bm{\xi}) \}$ and all $\bm{\xi} \in \Xi$, then
        \begin{align*}
        \ell(\bm{x}) &= \inf_{\bm{\xi} \in [0, 1]^{n_p}} \bm{c}(\bm{\xi})^\top \bm{x} + \inf_{\bm{y} \in \mathcal{Y}} \bm{d}_0^\top \bm{y},
        \end{align*}
        If also %
        $\bm{h}(\bm{\xi}) = \bm{h}_0 + \bm{H}\bm{\xi}$ satisfies %
        $\bm{H}^\top \bm{y} \geq \bm{0}$ for all $\{ \bm{y} \in \mathcal{Y}: \bm{T}\bm{x} + \bm{W} \bm{y} \geq \bm{h}(\bm{\xi}) \}$ and all $\bm{\xi} \in \Xi$, then
        \[
        u(\bm{x}) = \sup_{\bm{\xi} \in \{0, 1\}^{n_p}} \bm{c}(\bm{\xi})^\top \bm{x} 
        +
        \bm{d}(\one)^\top \bm{y}^0,
        \]
        for any arbitrary $\bm{y}^0 \in \{ \bm{y} \in \mathcal{Y}: \bm{T}\bm{x} + \bm{W} \bm{y} \geq \bm{h}(\one) \}$.    
    
        \item If $\bm{x}$ is static robust feasible; that is, the set $\mathcal{Y}(\bm{x}) \coloneqq \{ \bm{y} \in \mathcal{Y}: \bm{T}\bm{x} + \bm{W} \bm{y} \geq \sup_{\bm{\xi} \in \Xi} \bm{h}(\bm{\xi}) \} \neq \emptyset$, then
        \[
        u(\bm{x}) = \sup_{\bm{\xi} \in \{0, 1\}^{n_p}} \big\{
            \bm{c}(\bm{\xi})^\top \bm{x}  + \bm{d}(\bm{\xi})^\top \bm{y}^0
            \big \},
        \]
        for any arbitrary $\bm{y}^0 \in \mathcal{Y}(\bm{x})$.
    \end{enumerate}
\end{theorem}
\begin{proof}
    The first statement follows by observing first the following:
    \begin{align*}
    \inf_{\bm{\xi} \in \Xi} \inf_{\bm{y} \in \mathcal{Y}} \left\{\bm{c}(\bm{\xi})^\top \bm{x} + \bm{d}(\bm{\xi})^\top \bm{y} \right\}
    &
    \geq
    \inf_{\bm{\xi} \in [0, 1]^{n_p}} \inf_{\bm{y} \in [{\bm{y}}^\ell, {\bm{y}}^u]} \left\{\bm{c}(\bm{\xi})^\top \bm{x} + \bm{d}(\bm{\xi})^\top \bm{y} \right\},
    \\
    \text{and}\; \sup_{\bm{\xi} \in \Xi} \mathcal{Q}(\bm{x}, \bm{\xi})
    &
    =
    \sup_{\bm{\xi} \in \Xi} \inf_{\bm{y} \in \mathcal{Y}} \left\{\bm{c}(\bm{\xi})^\top \bm{x}  + \bm{d}(\bm{\xi})^\top {\bm{y}} : \bm{T}\bm{x} + \bm{W} \bm{y} \geq \bm{h}(\bm{\xi})\right\}
    \\
    &
    \leq
    \sup_{\bm{\xi} \in \{0, 1\}^{n_p}} \inf_{\bm{y} \in \mathcal{Y}} \left\{\bm{c}(\bm{\xi})^\top \bm{x}  + \bm{d}(\bm{\xi})^\top {\bm{y}} : \bm{T}\bm{x} + \bm{W} \bm{y} \geq \bm{h}(\bm{\xi})\right\}
    \\
    &
    \leq
    \sup_{\bm{\xi} \in \{0, 1\}^{n_p}} \sup_{\bm{y} \in [{\bm{y}}^\ell, {\bm{y}}^u]} \left\{\bm{c}(\bm{\xi})^\top \bm{x}  + \bm{d}(\bm{\xi})^\top {\bm{y}} \right\},
    \end{align*}
    and then by noting that the optimal value of a disjoint bilinear program over polytopes is attained at extreme points of the individual polytopes.
    
    The expression for $\ell(\bm{x})$ in the second statement follows by noting that because $\bm{\xi} \geq \bm{0}$, we have
    $\bm{d}(\bm{\xi})^\top {\bm{y}} = \bm{d}_0^\top \bm{y} + \bm{\xi}^\top \bm{D}^\top \bm{y} \geq \bm{d}_0^\top \bm{y}$ for all $\bm{y} \in \mathcal{Y}$.
    Similarly, the expression for $u(\bm{x})$ can be obtained by noting that because $\bm{\xi} \leq \one$,
    we have
    $\bm{d}(\bm{\xi})^\top {\bm{y}} = \bm{d}_0^\top \bm{y} + \bm{\xi}^\top \bm{D}^\top \bm{y} \leq \bm{d}_0^\top \bm{y} + \one^\top \bm{D}^\top \bm{y} = \bm{d}(\one)^\top \bm{y}$ for all $\bm{y} \in \mathcal{Y}$.
    Moreover, any $\bm{y} \in \mathcal{Y}$ that is feasible in problem $\mathcal{Q}(\bm{x}, \one)$ is also feasible in $\mathcal{Q}(\bm{x}, \bm{\xi})$ for any $\bm{\xi} \in \Xi$, since $\bm{T}\bm{x} + \bm{W} \bm{y} \geq \bm{h}(\one) \geq \bm{h}(\bm{\xi})$.
    Therefore, we have:
    \begin{align*}
        \sup_{\bm{\xi} \in \Xi} \mathcal{Q}(\bm{x}, \bm{\xi})
        &
        =
        \sup_{\bm{\xi} \in \Xi} \bm{c}(\bm{\xi})^\top \bm{x} 
        +
        \sup_{\bm{\xi} \in \Xi} \inf_{\bm{y} \in \mathcal{Y}} \left\{\bm{d}(\bm{\xi})^\top {\bm{y}} : \bm{T}\bm{x} + \bm{W} \bm{y} \geq \bm{h}(\bm{\xi})\right\}
        \\
        &
        =
        \sup_{\bm{\xi} \in \Xi} \bm{c}(\bm{\xi})^\top \bm{x} 
        +
        \inf_{\bm{y} \in \mathcal{Y}} \left\{\bm{d}(\one)^\top {\bm{y}} : \bm{T}\bm{x} + \bm{W} \bm{y} \geq \bm{h}(\one)\right\}
        \\
        &
        \leq
        \sup_{\bm{\xi} \in \{0, 1\}^{n_p}} \bm{c}(\bm{\xi})^\top \bm{x} 
        +
        \bm{d}(\one)^\top {\bm{y}}^0.
    \end{align*}
    As per the condition of Theorem~\ref{theorem:optimal_multiplier_general}, since $\bm{x}$ satisfies $\sup \big\{\mathcal{Q}(\bm{x}, \bm{\xi}) : \bm{\xi} \in \Xi \big\} < +\infty$,
    we also must have that the feasible region of problem
    $\mathcal{Q}(\bm{x}, \one)$
    is non-empty;
    that is, ${\bm{y}}^0$ is guaranteed to exist.
    
    The last statement follows from the max-min inequality:
    \begin{align*}
        \sup_{\bm{\xi} \in \Xi} \mathcal{Q}(\bm{x}, \bm{\xi})
        &
        \leq
        \inf_{\bm{y} \in \mathcal{Y}} \left\{ \sup_{\bm{\xi} \in \Xi} \left\{\bm{c}(\bm{\xi})^\top \bm{x}  + \bm{d}(\bm{\xi})^\top {\bm{y}} \right\} : \bm{T}\bm{x} + \bm{W} \bm{y} \geq \max_{\bm{\xi} \in \Xi} \bm{h}(\bm{\xi})\right\}
        \\
        &
        \leq
        \sup_{\bm{\xi} \in \{0, 1\}^{n_p}} \left\{\bm{c}(\bm{\xi})^\top \bm{x}  + \bm{d}(\bm{\xi})^\top {\bm{y}}^0 \right\},
    \end{align*}
    where $\bm{y}^0 \in \mathcal{Y}(\bm{x})$.
\end{proof}

We close with some remarks about Theorems~\ref{theorem:optimal_multiplier_general} and~\ref{theorem:closed_form_lu}.
First, the condition under which Theorem~\ref{theorem:optimal_multiplier_general} is established is neither weaker nor stronger than relatively complete recourse.
It is equivalent to stating that if a first-stage decision is feasible for all $\bm{\xi} \in \Xi$, then it is also feasible for all $\bm{\xi} \in \{0, 1\}^{n_p}$.
This condition is satisfied in many applications, including those where slack variables are explicitly added to ensure problem feasibility.
Second, all of the conditions stated in Theorem~\ref{theorem:closed_form_lu} are easily verifiable and largely problem-independent.
Nevertheless, although the closed-form expressions in Theorem~\ref{theorem:closed_form_lu} remain valid under its stated conditions,
it is important to choose
small values of $u(\bm{x})$ and large values of $\ell(\bm{x})$ %
as this will lead to smaller values of the optimal Lagrange multiplier.
This will be important for numerical reasons that we shall elaborate in the next section.

\section{Computation of Worst-Case Parameter Realizations}\label{sec:algorithm}

We now use the analytical results from the previous section to design practical methods for computing worst-case parameter realizations that can be integrated in the Benders and column-and-constraint generation algorithms even when the relatively complete recourse assumption may not be satisfied.
We also discuss extensions to mixed-integer second-stage decisions.

Theorems~\ref{theorem:necessary_optimality_condition} and~\ref{theorem:optimal_multiplier_general} characterize the optimal Lagrange multiplier only in the context of feasible first-stage decisions.
To reliably identify a potentially infeasible first-stage decision $\bm{x} \in \mathcal{X}$ for which the second-stage problem becomes infeasible for some $\bm{\xi} \in \Xi$, we introduce the constraint violation functions $\mathcal{S}: \mathcal{X} \times \Xi \to \mathbb{R}$ and $\mathcal{S}_\mathcal{I}: \mathcal{X} \times \Xi \to \mathbb{R}$
which minimize constraint infeasibility in the problems $\mathcal{L}$ and $\mathcal{L}_\mathcal{I}$, respectively.
\begin{align}
    \mathcal{S}(\bm{x}, \bm{\xi})
    &=\left[
    \begin{aligned}
        \mathop{\text{minimize}}_{\bm{y}, \bm{z}, \bm{\sigma}} \;\; &
        \one^\top \bm{\sigma} + \phi(\bm{z}, \bm{\xi}) \\
        \text{subject to} \;\; & \bm{y} \in \mathcal{Y}, \;\; \bm{z} \in [0, 1]^{n_p}, \;\; \bm{\sigma} \in \mathbb{R}^m_{+} \\
        & \bm{T}\bm{x} + \bm{W} \bm{y} + \bm{\sigma} \geq \bm{h}(\bm{z}).
    \end{aligned}
    \right]
    \label{eq:slack_general} \\
    \mathcal{S}_\mathcal{I}(\bm{x}, \bm{\xi})
    &=\left[
    \begin{aligned}
        \mathop{\text{minimize}}_{\bm{y}, \bm{\sigma}} \;\; &
        \one^\top \bm{\sigma} +  \phi_\mathcal{I}(\bm{x}, \bm{y}, \bm{\xi}) \\
        \text{subject to} \;\; & \bm{y} \in \mathcal{Y}, \;\; \bm{\sigma} \in \mathbb{R}^m_{+} \\
        & \bm{g}(\bm{x}, \bm{y}) + \bm{\sigma} \geq \bm{0}.
    \end{aligned}
    \right]
    \label{eq:slack_indicator}
\end{align}
We immediately note the following properties of these functions.
\begin{lemma}[Properties of the constraint violation functions]\label{lemma:constraint_violation_function_properties}
    For any $\bm{x} \in \mathcal{X}$ and $\bm{\xi} \in \Xi$, the functions
    $\mathcal{S}$ and $\mathcal{S}_\mathcal{I}$
    satisfy the following properties.
    \begin{enumerate}
        \item $\mathcal{S}(\bm{x}, \bm{\xi}) \geq 0$ and $\mathcal{S}(\bm{x}, \bm{\xi}) = 0$ if and only if $\mathcal{Q}(\bm{x}, \bm{\xi}) < +\infty$.
        \item $\mathcal{S}_\mathcal{I}(\bm{x}, \bm{\xi}) \geq 0$ and $\mathcal{S}_\mathcal{I}(\bm{x}, \bm{\xi}) = 0$ if and only if $\mathcal{Q}_\mathcal{I}(\bm{x}, \bm{\xi}) < +\infty$.
    \end{enumerate}
\end{lemma}
\begin{proof}
    First, since $\mathcal{Y} \neq \emptyset$, both minimization problems \eqref{eq:slack_general} and \eqref{eq:slack_indicator} are always feasible.
    Also, since the penalty functions $\phi$ and $\phi_\mathcal{I}$ are always non-negative, the objective functions of both problems are non-negative by construction.
    Finally, note that if $\mathcal{S}(\bm{x}, \bm{\xi}) = 0$,
    then any optimal solution $(\bm{y}, \bm{z}, \bm{\sigma})$ satisfies
    $\bm{\sigma} = \bm{0}$, implying that $(\bm{y}, \bm{z})$ is feasible in the second-stage problem $\mathcal{Q}(\bm{x}, \bm{\xi})$.
    Similarly, for any feasible solution of the latter problem, we can construct a feasible solution in $\mathcal{S}(\bm{x}, \bm{\xi})$ with objective value $0$, by setting $\bm{\sigma} = 0$.
    Similar arguments can be made for $\mathcal{S}_\mathcal{I}$.
\end{proof}

Using the constraint violation functions, we can compute worst-case parameter realizations even when the relatively complete recourse assumption may not be satisfied.
Algorithm~\ref{algo:wc:closed_form} describes the computations for \ref{eq:two_stage_ro_general} when the sufficient conditions of Theorem~\ref{theorem:closed_form_lu} are satisfied, 
whereas Algorithms~\ref{algo:wc:general} and \ref{algo:wc:indicator} describe the computations for general instances of \ref{eq:two_stage_ro_general} and \ref{eq:two_stage_ro_indicator}, respectively.
All algorithms first compute the worst-case constraint violations.
If this is equal to $0$ (with actual implementations using a small non-zero tolerance),
then Lemma~\ref{lemma:constraint_violation_function_properties} ensures that $\bm{x} \in \mathcal{X}$
is a feasible first-stage decision.
The algorithms then proceed to compute worst-case parameter realizations over the Lagrangian dual functions.
We note that the computation of the worst-case constraint violations can be skipped if the problem is known to satisfy the relatively complete recourse assumption.

Algorithm~\ref{algo:wc:closed_form} applies Theorem~\ref{theorem:optimal_multiplier_general} to compute a closed-form expression for the optimal Lagrange multiplier when $\bm{x}$ is feasible.
Indeed, in this case, the conditions of Theorem~\ref{theorem:optimal_multiplier_general} are satisfied and $\hat{\lambda}$ in line~\ref{algo:wc:closed_form:lambda_closed_form} is an optimal multiplier.
The conditions in Theorem~\ref{theorem:closed_form_lu} and the corresponding closed-form expressions that are used for computing the optimal multiplier can all be verified and computed efficiently, and even analytically in several applications.
We note, however, that whenever possible, problem-specific knowledge that can yield smaller values of $u(\bm{x})$ (i.e., tighter upper bounds on the worst-case objective value of $\bm{x}$) and larger values of $\ell(\bm{x})$ (i.e., tighter lower bounds on the best-case objective value of $\bm{x}$) should be used, since it will lead to smaller values of $\hat{\lambda}$ and a computationally less expensive maximization (over $\bm{\xi} \in \Xi$) of the Lagrangian function evaluated at $\hat{\lambda}$ in line~\ref{algo:wc:closed_form:optimal_xi_computation}.
Note that if $\bm{x}$ is infeasible, then any positive value of $\hat{\lambda}$ is an optimal multiplier and line~\ref{algo:wc:closed_form:feasible_xi_computation} arbitrarily sets $\hat{\lambda} = 1$.
 
\begin{algorithm}[!htbp]
    \caption{Computation of worst-case parameter realization for \ref{eq:two_stage_ro_general} when the conditions in Theorem~\ref{theorem:closed_form_lu} are satisfied}
    \label{algo:wc:closed_form}
    \begin{algorithmic}[1]
        \renewcommand{\algorithmicrequire}{\textbf{Input:}}
        \renewcommand{\algorithmicensure}{\textbf{Output:}}
        \REQUIRE $\bm{x} \in \mathcal{X}$
        \ENSURE $\hat{\bm{\xi}} \in \mathop{\arg\max}\limits_{\bm{\xi} \in \Xi} \mathcal{Q}(\bm{x}, \bm{\xi})$ and $\hat{\lambda} \in \mathop{\arg\max}\limits_{\lambda \in \mathbb{R}_{+}} \bigg\{\sup\limits_{\bm{\xi} \in \Xi} \mathcal{L}(\bm{x}, \bm{\xi}, \lambda)\bigg\}$
        \STATE Compute 
        $
        \hat{\bm{\xi}} \in \mathop{\arg\max}\limits_{\bm{\xi} \in \Xi} \mathcal{S}(\bm{x}, \bm{\xi})
        $
        and set $\hat{\lambda} = 1$
        \label{algo:wc:closed_form:feasible_xi_computation}
        \IF {$\mathcal{S}(\bm{x}, \hat{\bm{\xi}}) = 0$}
        \STATE Compute $\ell(\bm{x})$ and $u(\bm{x})$ as per Theorem~\ref{theorem:optimal_multiplier_general} and set $\hat{\lambda} = u(\bm{x}) - \ell(\bm{x})$
        \label{algo:wc:closed_form:lambda_closed_form}
        \STATE Compute
        $
        \hat{\bm{\xi}} \in \mathop{\arg\max}\limits_{\bm{\xi} \in \Xi} \mathcal{L}(\bm{x}, \bm{\xi}, \hat{\lambda})
        $
        \label{algo:wc:closed_form:optimal_xi_computation}
        \ENDIF
    \end{algorithmic}
\end{algorithm}

Algorithms~\ref{algo:wc:general} and~\ref{algo:wc:indicator} exploit Lemma~\ref{lemma:wc_lagrangian_function_properties} and Theorem~\ref{theorem:necessary_optimality_condition} to compute a multiplier for \ref{eq:two_stage_ro_general} and \ref{eq:two_stage_ro_indicator}, respectively, when $\bm{x} \in \mathcal{X}$ is a feasible first-stage decision. %
Since the worst-case Lagrangian functions are monotonically non-decreasing in $\lambda$,
the algorithms verify if the necessary conditions in Theorem~\ref{theorem:necessary_optimality_condition} are satisfied for exponentially increasing trial values of $\hat{\lambda}$.
Indeed, the factor 2 appearing in line~\ref{algo:wc:general:lambda_update} is arbitrary, and any value greater than $1$ will also work.
We highlight that if $\bar{\lambda}$ is the smallest possible optimal multiplier, then the algorithm terminates after $\mathcal{O}(\log(\max\{\bar{\lambda}, \lambda^0\}/\lambda^0))$
trial evaluations of $\hat{\lambda}$.
Note that computing $\hat{\bm{y}}, \hat{\bm{z}}$ in line~\ref{algo:wc:general:optimal_y_computation} requires solving a deterministic linear program if the second-stage decisions are continuous, or a deterministic mixed-integer program otherwise.
Since Theorem~\ref{theorem:necessary_optimality_condition} is only a necessary condition, it is possible that the final multiplier $\hat{\lambda}$ and parameter realization $\hat{\bm{\xi}}$ are not optimal and hence,
$\mathcal{L}(\bm{x}, \hat{\bm{\xi}}, \hat{\lambda}) < \sup_{\bm{\xi} \in \Xi} \mathcal{Q}(\bm{x}, \bm{\xi})$.
We shall discuss a practical workaround in the next subsection.

\begin{algorithm}[!htbp]
    \caption{Computation of worst-case parameter realization for general \ref{eq:two_stage_ro_general}}
    \label{algo:wc:general}
    \begin{algorithmic}[1]
        \renewcommand{\algorithmicrequire}{\textbf{Input:}}
        \renewcommand{\algorithmicensure}{\textbf{Output:}}
        \REQUIRE $\bm{x} \in \mathcal{X}$, $\lambda^0 > 0$
        \ENSURE Either $\hat{\bm{\xi}} \in \Xi: \mathcal{Q}(\bm{x}, \hat{\bm{\xi}}) = +\infty$, $\hat{\lambda} = \lambda^0$ or $\hat{\bm{\xi}}, \hat{\lambda}$ satisfying conditions of Theorem~\ref{theorem:necessary_optimality_condition}
        \STATE Compute 
        $
        \hat{\bm{\xi}} \in \mathop{\arg\max}\limits_{\bm{\xi} \in \Xi} \mathcal{S}(\bm{x}, \bm{\xi})
        $
        and set $\hat{\lambda} = \lambda^0$
        \label{algo:wc:general:feasible_xi_computation}
        \IF {$\mathcal{S}(\bm{x}, \hat{\bm{\xi}}) = 0$}
        \STATE Set $\hat{\lambda} = \lambda^0/2$
        \REPEAT
        \STATE Update $\hat{\lambda} \gets 2\hat{\lambda}$
        \label{algo:wc:general:lambda_update}
        \STATE Compute 
        $
        \hat{\bm{\xi}} \in \mathop{\arg\max}\limits_{\bm{\xi} \in \Xi} \mathcal{L}(\bm{x}, \bm{\xi}, \hat{\lambda})
        $
        \label{algo:wc:general:optimal_xi_computation}
        \STATE Compute 
        $
        (\hat{\bm{y}}, \hat{\bm{z}}) \in \mathop{\arg\min}\limits_{(\bm{y}, \bm{z}) \in \mathcal{Y} \times [0, 1]^{n_p}} \left\{
        \bm{d}(\hat{\bm{\xi}})^\top \bm{y} + \hat{\lambda} \phi(\bm{z}, \hat{\bm{\xi}})
        :
        \bm{T}\bm{x} + \bm{W} \bm{y} \geq \bm{h}(\bm{z})
        \right\}
        $
        \label{algo:wc:general:optimal_y_computation}
        \UNTIL {$\phi(\hat{\bm{z}}, \hat{\bm{\xi}}) = 0$}
        \ENDIF
    \end{algorithmic}
\end{algorithm}

\begin{algorithm}[!htbp]
    \caption{Computation of worst-case parameter realization for general \ref{eq:two_stage_ro_indicator}}
    \label{algo:wc:indicator}
    \begin{algorithmic}[1]
        \renewcommand{\algorithmicrequire}{\textbf{Input:}}
        \renewcommand{\algorithmicensure}{\textbf{Output:}}
        \REQUIRE $\bm{x} \in \mathcal{X}$, $\lambda^0 > 0$
        \ENSURE Either $\hat{\bm{\xi}} \in \Xi: \mathcal{Q}_\mathcal{I}(\bm{x}, \hat{\bm{\xi}}) = +\infty$, $\hat{\lambda} = \lambda^0$ or $\hat{\bm{\xi}}, \hat{\lambda}$ satisfying conditions of Theorem~\ref{theorem:necessary_optimality_condition}
        \STATE Compute 
        $
        \hat{\bm{\xi}} \in \mathop{\arg\max}\limits_{\bm{\xi} \in \Xi} \mathcal{S}_\mathcal{I}(\bm{x}, \bm{\xi})
        $
        and set $\hat{\lambda} = \lambda^0$
        \label{algo:wc:indicator:feasible_xi_computation}
        \IF {$\mathcal{S}_\mathcal{I}(\bm{x}, \hat{\bm{\xi}}) = 0$}
        \STATE Set $\hat{\lambda} = \lambda^0/2$
        \REPEAT
        \STATE Update $\hat{\lambda} \gets 2\hat{\lambda}$
        \label{algo:wc:indicator:lambda_update}
        \STATE Compute 
        $
        \hat{\bm{\xi}} \in \mathop{\arg\max}\limits_{\bm{\xi} \in \Xi} \mathcal{L}_\mathcal{I}(\bm{x}, \bm{\xi}, \hat{\lambda})
        $
        \label{algo:wc:indicator:optimal_xi_computation}
        \STATE Compute 
        $
        \hat{\bm{y}} \in \mathop{\arg\min}\limits_{\bm{y} \in \mathcal{Y}} \left\{
        \bm{d}(\hat{\bm{\xi}})^\top \bm{y} + \hat{\lambda} \phi_\mathcal{I}(\bm{x}, \bm{y}, \hat{\bm{\xi}})
        :
        \bm{g}(\bm{x}, \bm{y}) \geq \bm{0}
        \right\}
        $
        \label{algo:wc:indicator:optimal_y_computation}
        \UNTIL {$\phi_\mathcal{I}(\bm{x}, \hat{\bm{y}}, \hat{\bm{\xi}}) = 0$}
        \ENDIF
    \end{algorithmic}
\end{algorithm}

Algorithms \ref{algo:wc:general} and \ref{algo:wc:indicator} appear quite similar to simply computing the maximum of the second-stage value functions over the uncertain parameters.
However, their main advantage is that unlike the latter, which require the solution of mixed-integer bilinear problems (refer to Section~\ref{sec:formulation:literature:bilinear}),
all intermediate optimization problems in lines \ref{algo:wc:closed_form:feasible_xi_computation}, \ref{algo:wc:closed_form:optimal_xi_computation} of Algorithm \ref{algo:wc:closed_form}, and in lines \ref{algo:wc:indicator:feasible_xi_computation}, \ref{algo:wc:indicator:optimal_xi_computation} of Algorithms \ref{algo:wc:general} and \ref{algo:wc:indicator} are relatively easier to solve.
To illustrate this, when the second-stage decisions are continuous ($\mathcal{Y} = \mathbb{R}^{n_2}_{+}$),
it can be readily verified using LP duality that these optimization problems can be reformulated as follows.
Here, we suppose that $\bm{h}(\bm{z}) = \bm{h}_0 + \bm{H}\bm{z}$
and $\bm{g}(\bm{x}, \bm{y}) = \bm{T} \bm{x} + \bm{W} \bm{y} - \bm{h}_0$.
\begin{gather}
    \sup_{\bm{\xi} \in \Xi} \mathcal{S}(\bm{x}, \bm{\xi})
    =
    \left[
    \begin{aligned}
        \mathop{\text{maximize}}_{\bm{\xi} \in \Xi, \bm{\mu} \in \mathbb{R}^m_{+}, \bm{\beta} \in \mathbb{R}^{n_p}_{+}} &
        (\bm{h}_0 - \bm{T}\bm{x})^\top \bm{\mu} +
        \one^\top (\bm{\xi} - \bm{\beta}) \\
        \text{subject to} \;\quad & \bm{W}^\top \bm{\mu} \leq \bm{0}, \;\; \bm{\mu} \leq \bm{e} \\
        & 2 \bm{\xi} - \bm{H}^\top \bm{\mu} - \bm{\beta} \geq \one.
    \end{aligned}
    \right]
    \label{eq:worst_case_problem_general_slack_duality} \tag{$\mathcal{WC}_\mathcal{S}$}
    \\
    \sup_{\bm{\xi} \in \Xi} \mathcal{L}(\bm{x}, \bm{\xi}, \lambda)
    =
    \left[
    \begin{aligned}
        \mathop{\text{maximize}}_{\bm{\xi} \in \Xi, \bm{\mu} \in \mathbb{R}^m_{+}, \bm{\beta} \in \mathbb{R}^{n_p}_{+}}
        &
        \bm{c}(\bm{\xi})^\top\bm{x} +
        (\bm{h}_0 - \bm{T}\bm{x})^\top \bm{\mu} +
        \one^\top (\lambda \bm{\xi} - \bm{\beta}) \\
        \text{subject to} \;\quad & \bm{W}^\top \bm{\mu} \leq \bm{d}(\bm{\xi}) \\
        & 2 \lambda\bm{\xi} - \bm{H}^\top \bm{\mu} - \bm{\beta} \geq \lambda\one.
    \end{aligned}
    \right]
    \label{eq:worst_case_problem_general_lagrangian_duality} \tag{$\mathcal{WC}_\mathcal{L}$}
    \notag
\end{gather}
\begin{gather}
    \sup_{\bm{\xi} \in \Xi} \mathcal{S}_\mathcal{I}(\bm{x}, \bm{\xi})
    =
    \left[
    \begin{aligned}
        \mathop{\text{maximize}}_{\bm{\xi} \in \Xi, \bm{\mu} \in \mathbb{R}^m_{+}} \;\; &
        (\bm{h}_0 - \bm{T}\bm{x})^\top ( \bm{\mu} - \bm{\delta}(\eye, \bm{\xi}) ) \\
        \text{subject to} \;\; & \bm{W}^\top \bm{\mu} \leq \bm{\delta}(\bm{W}, \bm{\xi}), \;\; \bm{\mu} \leq \bm{e}.
    \end{aligned}
    \right]
    \label{eq:worst_case_problem_indicator_slack_duality} \tag{$\mathcal{WC}_{\mathcal{S}_\mathcal{I}}$}
    \\
    \sup_{\bm{\xi} \in \Xi} \mathcal{L}_\mathcal{I}(\bm{x}, \bm{\xi}, \lambda)
    =
    \left[
    \begin{aligned}
        \mathop{\text{maximize}}_{\bm{\xi} \in \Xi, \bm{\mu} \in \mathbb{R}^m_{+}} \;\; &
        \bm{c}(\bm{\xi})^\top\bm{x} +
        (\bm{h}_0 - \bm{T}\bm{x})^\top ( \bm{\mu} - \lambda \bm{\delta}(\eye, \bm{\xi}) ) \\
        \text{subject to} \;\; & \bm{W}^\top \bm{\mu} \leq \bm{d}(\bm{\xi}) + \lambda \bm{\delta}(\bm{W}, \bm{\xi}).
    \end{aligned}
    \right]
    \label{eq:worst_case_problem_indicator_lagrangian_duality} \tag{$\mathcal{WC}_{\mathcal{L}_\mathcal{I}}$}
    \\
    \bm{\delta}(\bm{A}, \bm{\xi}) = \sum_{j \in [n_p]} \sum_{i \in \mathcal{I}_j^1} \xi_j \bm{A}^\top \one_i + \sum_{j \in [n_p]}\sum_{i \in \mathcal{I}_j^0} (1 - \xi_j) \bm{A}^\top \one_i, \;\; \bm{A} \in \mathbb{R}^{m \times n}, \; n \in \mathbb{Z}_{+}.
    \notag
\end{gather}

We note that all of the above problems \eqref{eq:worst_case_problem_general_slack_duality}, \eqref{eq:worst_case_problem_general_lagrangian_duality}, \eqref{eq:worst_case_problem_indicator_slack_duality} and \eqref{eq:worst_case_problem_indicator_lagrangian_duality}
are already in MILP form.
This is in contrast to \eqref{eq:worst_case_problem_duality} and \eqref{eq:worst_case_problem_kkt},
which are bilinear and need further reformulation before they can be cast as MILP problems.
On the one hand, this results in smaller optimization problems to be solved; see Table~\ref{table:formulation_sizes}.
On the other hand, it does not require any arbitrary big-M upper bounds on dual variables that can lead to suboptimal decisions. %
We note however that, as mentioned in the description of Algorithm~\ref{algo:wc:general},
it is crucial to choose
tight upper bounds $u(\bm{x})$ on the worst-case objective value of $\bm{x}$
and tight lower bounds $\ell(\bm{x})$ on the best-case objective value of $\bm{x}$
when computing the optimal multiplier in \ref{eq:two_stage_ro_general}.
This will impact numerical performance when solving \eqref{eq:worst_case_problem_general_lagrangian_duality}.

\subsection{Integration in Benders decomposition algorithm}

Algorithm~\ref{algo:benders:indicator} presents a Benders decomposition scheme to solve \ref{eq:two_stage_ro_indicator} when the second-stage decisions are continuous ($\mathcal{Y} = \mathbb{R}^{n_2}_{+}$).
Note that relatively complete recourse is not assumed.
The algorithm iteratively refines inner approximations of
$
\sup_{\lambda \geq 0}
\sup_{\bm{\xi} \in \Xi}
\mathcal{L}_\mathcal{I}(\bm{x}, \bm{\xi}, \lambda)
$
and
$
\sup_{\bm{\xi} \in \Xi}
\mathcal{S}_\mathcal{I}(\bm{x}, \bm{\xi}, \lambda)
$.
Indeed, the former is reflected in the objective function of the lower bounding problem~\eqref{eq:benders_indicator_lb_update},
where the set $\mathcal{O}$
enumerates the extreme points of the feasible region of problem~\eqref{eq:worst_case_problem_indicator_lagrangian_duality}.
Similarly, note that we have:
\[
\inf_{\bm{x} \in \mathcal{X}} \sup_{\bm{\xi} \in \Xi} \mathcal{L}_\mathcal{I}(\bm{x}, \bm{\xi}, \lambda)\\
=
\inf_{\bm{x} \in \mathcal{X}} \left\{
\max_{\bm{\xi} \in \Xi} \mathcal{L}_\mathcal{I}(\bm{x}, \bm{\xi}, \lambda) :
\mathcal{S}_\mathcal{I}(\bm{x}, \bm{\xi}) \leq 0  \;\; \forall \bm{\xi} \in \Xi
\right\}.
\]
In other words, the constraints of the lower bounding problem~\eqref{eq:benders_indicator_lb_update}, written for all points in $\mathcal{F}$, enumerate the extreme points of the feasible region of problem~\eqref{eq:worst_case_problem_indicator_slack_duality}.
In Benders terminology, the sets $\mathcal{O}$ and $\mathcal{F}$ corresponding to feasibility and optimality cuts, respectively.

\begin{algorithm}[!htb]
    \caption{Benders decomposition scheme along with Algorithm~\ref{algo:wc:indicator} to solve \ref{eq:two_stage_ro_indicator} when $\mathcal{Y} = \mathbb{R}^{n_2}_{+}$ and $\bm{g}(\bm{x}, \bm{y}) = \bm{T} \bm{x} + \bm{W} \bm{y} - \bm{h}_0$}
    \label{algo:benders:indicator}
    \begin{algorithmic}[1]
        \renewcommand{\algorithmicrequire}{\textbf{Input:}}
        \renewcommand{\algorithmicensure}{\textbf{Output:}}
        \REQUIRE Problem \ref{eq:two_stage_ro_indicator}, tolerance $\epsilon > 0$
        \ENSURE $\bm{x}^\star$ (an optimal solution of \ref{eq:two_stage_ro_indicator} if it is feasible or $\emptyset$ otherwise)
        \STATE Initialize $LB = -\infty$, $UB = +\infty$, $\bm{x}^\star = \emptyset$, $\mathcal{O} = \emptyset$, $\mathcal{F} = \emptyset$, ${\lambda} = 1$.
        \REPEAT
        \STATE Set $LB$ as the optimal objective value of the problem:
        \begin{equation}\label{eq:benders_indicator_lb_update}
            \begin{aligned}
                \mathop{\text{minimize}}_{\bm{x} \in \mathcal{X}} \;\; & %
                \sup_{(\bm{\xi}, \bm{\mu}) \in \mathcal{O}} \left\{\bm{c}(\bm{\xi})^\top\bm{x} +
                ( \bm{\mu} - {{\lambda}} \bm{\delta}(\eye, \bm{\xi}) )^\top (\bm{h}_0 - \bm{T}\bm{x}) \right\}  \\ %
                \text{subject to} \;\;& ( \bm{\mu} - \bm{\delta}(\eye, \bm{\xi}) )^\top (\bm{h}_0 - \bm{T}\bm{x}) \leq 0, \; (\bm{\xi}, \bm{\mu}) \in \mathcal{F}.
            \end{aligned}
        \end{equation}
        \label{algo:benders:indicator:lb-update}
        \IF {$LB < +\infty$}
        \STATE Set $\hat{\bm{x}}$ as an optimal solution of problem \eqref{eq:benders_indicator_lb_update}
        \STATE \multiline{Compute $\hat{\bm{\xi}}$, $\hat{\lambda}$ by using Algorithm~\ref{algo:wc:indicator} (with inputs $\hat{\bm{x}}$, $\lambda$) and executing its lines~\ref{algo:wc:indicator:feasible_xi_computation} and~\ref{algo:wc:indicator:optimal_xi_computation} by computing optimal solutions $(\hat{\bm{\xi}}, \hat{\bm{\mu}})$ of problems~\ref{eq:worst_case_problem_indicator_slack_duality} (at $\bm{x} = \hat{\bm{x}}$) and~\ref{eq:worst_case_problem_indicator_lagrangian_duality} (at $\bm{x} = \hat{\bm{x}}$ and $\lambda = \hat{\lambda}$), respectively.}
        \IF {$\mathcal{S}_\mathcal{I}(\hat{\bm{x}}, \hat{\bm{\xi}}) > 0$}
        \STATE Update $\mathcal{F} \gets \mathcal{F} \cup \{(\hat{\bm{\xi}}, \hat{\bm{\mu}})\}$
        \ELSE
        \STATE Update $\mathcal{O}\gets \mathcal{O} \cup \{(\hat{\bm{\xi}}, \hat{\bm{\mu}})\}$.
        \STATE Update $\lambda \gets \hat{\lambda}$
        \IIf {$UB > \mathcal{L}_\mathcal{I}(\hat{\bm{x}}, \hat{\bm{\xi}}, \hat{\lambda})$}
        update
        $UB \gets \mathcal{L}_\mathcal{I}(\hat{\bm{x}}, \hat{\bm{\xi}}, \hat{\lambda})$ and
        $\bm{x}^\star \gets \hat{\bm{x}}$.
        \EndIIf
        \ENDIF
        \ELSE
        \STATE Set $\hat{\bm{x}} = \emptyset$
        \ENDIF
        \UNTIL {$UB - LB \leq \epsilon$ or $LB = +\infty$}
    \end{algorithmic}
\end{algorithm}

It is crucial to note that the algorithm implicitly exploits the strong duality result in Theorem~\ref{theorem:strong_duality_two_stage_problem}.
Indeed, since we have:
\[
\inf_{\bm{x} \in \mathcal{X}} \sup_{\bm{\xi} \in \Xi} \mathcal{Q}_\mathcal{I}(\bm{x}, \bm{\xi}, \lambda)
\geq
\inf_{\bm{x} \in \mathcal{X}} \sup_{\bm{\xi} \in \Xi} \mathcal{L}_\mathcal{I}(\bm{x}, \bm{\xi}, \lambda),
\;\; \lambda \geq 0,
\]
problem~\eqref{eq:benders_indicator_lb_update} provides a rigorous lower bound on the optimal value of \ref{eq:two_stage_ro_indicator}, even though $(\hat{\bm{\xi}}, \hat{\bm{\mu}})$ may correspond to extreme points of \eqref{eq:worst_case_problem_indicator_lagrangian_duality} for potentially different values of $\lambda$.
The algorithm terminates finitely since Algorithm~\ref{algo:wc:indicator} also terminates in a finite number of iterations, only a finite number of trial values of $\lambda$ are enumerated, and because the feasible regions of \eqref{eq:worst_case_problem_indicator_lagrangian_duality} and \eqref{eq:worst_case_problem_indicator_slack_duality} have a finite number of extreme points for any $\lambda \geq 0$.
Note that if the problem is known to satisfy relatively complete recourse, then $\mathcal{F} = \emptyset$ is guaranteed in all iterations of the algorithm.

Since Algorithm~\ref{algo:wc:indicator} computes a Lagrange multiplier that only satisfies the necessary optimality conditions in Theorem~\ref{theorem:necessary_optimality_condition}, it is possible that Algorithm~\ref{algo:benders:indicator} outputs a solution whose worst-case objective value is estimated to be less than the optimal value of \ref{eq:two_stage_ro_indicator}.
We can verify if this happens to be the case at the end of algorithm by solving a mixed-integer bilinear variant of~\eqref{eq:worst_case_problem_indicator_lagrangian_duality}
where $\lambda \geq 0$ is also a decision variable.
Here, the bilinear products $\lambda \bm{\xi}$ can be reformulated by adding variables $\bm{\rho} \in \mathbb{R}_{+}^{n_p}$ and indicator constraints:
\begin{alignat*}{3}
    \xi_j = 0 &\implies \rho_j = 0, &\; j \in [n_p], \\
    \xi_j = 1 &\implies \rho_j = \lambda, &\; j \in [n_p].
\end{alignat*}
If the optimal value of this problem is larger than predicted ($UB$), then we can use its optimal solution $\lambda$ to initialize another run of the algorithm with an appropriately updated $UB$.
In doing so, we can retain the feasibility and optimality sets, $\mathcal{F}$ and $\mathcal{O}$, and do not have to re-initialize them to be empty sets.
We emphasize that if Algorithm~\ref{theorem:necessary_optimality_condition} outputs the true optimal multiplier $\lambda$, then the bilinear problem has to be solved only at most once (because of the strong duality result in Theorem~\ref{theorem:strong_duality_two_stage_problem}).

Finally, we note that the corresponding Benders scheme for \ref{eq:two_stage_ro_general} is similar with the difference that it enumerates the extreme points of \eqref{eq:worst_case_problem_general_lagrangian_duality} and \eqref{eq:worst_case_problem_general_slack_duality} instead.
Also, the worst-case parameter realizations can be computed using Algorithm~\ref{algo:wc:closed_form} whenever the conditions therein are satisfied.

\subsection{Integration in column-and-constraint generation algorithm}

Algorithm~\ref{algo:ccg:indicator} presents a column-and-constraint generation scheme for solving problem \ref{eq:two_stage_ro_indicator} without assuming relatively complete recourse.
It is quite similar to Algorithm~\ref{algo:ccg} for solving problem \ref{eq:two_stage_ro_general}.
The main difference is in the formulation of the lower bounding problem~\eqref{eq:ccg_indicator_lb_update}, where we note the absence of indicator constraints, since all constraints are unconditionally known for each parameter realization recorded in the set $\mathcal{R}$.
We also note that old values of the Lagrange multipliers $\lambda$ are used as inputs to Algorithm~\ref{algo:wc:indicator} so that $\lambda$ is non-decreasing across algorithmic iterations.
The algorithm terminates finitely since Algorithm~\ref{algo:wc:indicator} also terminates in a finite number of iterations and $\Xi$ is a finite set.
The corresponding column-and-constraint generation scheme for \ref{eq:two_stage_ro_general} is similar with the difference that the worst-case parameter realizations are computed using Algorithm~\ref{algo:wc:general}, or using Algorithm~\ref{algo:wc:closed_form} whenever applicable.

\begin{algorithm}[!htbp]
    \caption{Column-and-constraint generation scheme with Algorithm~\ref{algo:wc:indicator} to solve \ref{eq:two_stage_ro_indicator}}
    \label{algo:ccg:indicator}
    \begin{algorithmic}[1]
        \renewcommand{\algorithmicrequire}{\textbf{Input:}}
        \renewcommand{\algorithmicensure}{\textbf{Output:}}
        \REQUIRE Problem \ref{eq:two_stage_ro_indicator}, tolerance $\epsilon > 0$
        \ENSURE $\bm{x}^\star$ (an optimal solution of \ref{eq:two_stage_ro_indicator} if it is feasible or $\emptyset$ otherwise)
        \STATE Initialize $LB = -\infty$, $UB = +\infty$, $\bm{x}^\star = \emptyset$, $\mathcal{R} = \emptyset$, ${\lambda} = 1$.
        \REPEAT
        \STATE Set $LB$ as the optimal objective value of the problem:
        \begin{equation}\label{eq:ccg_indicator_lb_update}
            \inf_{\bm{x}, \eta, \bm{y}} \left\{
            \eta :
            \begin{aligned}
                & \bm{x} \in \mathcal{X}, \;\; \eta \in \mathbb{R}, \;\; \bm{y}^{(\bm{\xi})} \in \mathcal{Y}, \; \bm{\xi} \in \mathcal{R}, \\
                & \eta \geq \bm{c}(\bm{\xi})^\top\bm{x} + \bm{d}(\bm{\xi})^\top\bm{y}^{(\bm{\xi})}, \; \bm{\xi} \in \mathcal{R}, \\
                & \bm{g}\big(\bm{x}, \bm{y}^{(\bm{\xi})} \big) \geq \bm{0}, \; \bm{\xi} \in \mathcal{R}, \\
                & g_i\big(\bm{x}, \bm{y}^{(\bm{\xi})} \big) = 0, \;\; i \in \mathcal{I}_j^0, \; j \in [n_p]: \xi_j = 0, \; \bm{\xi} \in \mathcal{R}, \\
                & g_i\big(\bm{x}, \bm{y}^{(\bm{\xi})} \big) = 0, \;\; i \in \mathcal{I}_j^1, \; j \in [n_p]: \xi_j = 1, \; \bm{\xi} \in \mathcal{R}.
            \end{aligned}
            \right\}
        \end{equation}
        \label{algo:ccg:indicator:lb-update}
        \IF {$LB < +\infty$}
        \STATE Set $\hat{\bm{x}}$ as an optimal (projected) solution of problem \eqref{eq:ccg_indicator_lb_update}
        \STATE Compute $\hat{\bm{\xi}}$, $\hat{\lambda}$ by using Algorithm~\ref{algo:wc:indicator} (with inputs $\hat{\bm{x}}$, $\lambda$)
        \label{algo:ccg:indicator:worst-case-computation}
        \STATE Update $\mathcal{R} \gets \mathcal{R} \cup \{\hat{\bm{\xi}}\}$ and $\lambda \gets \hat{\lambda}$
        \IIf {$UB > \mathcal{L}_\mathcal{I}(\hat{\bm{x}}, \hat{\bm{\xi}}, \hat{\lambda})$}
        update
        $UB \gets \mathcal{L}_\mathcal{I}(\hat{\bm{x}}, \hat{\bm{\xi}}, \hat{\lambda})$ and
        $\bm{x}^\star \gets \hat{\bm{x}}$.
        \EndIIf
        \ELSE
        \STATE Set $\hat{\bm{x}} = \emptyset$
        \ENDIF
        \UNTIL {$UB - LB \leq \epsilon$ or $LB = +\infty$}
    \end{algorithmic}
\end{algorithm}

We re-emphasize a few key points from the previous subsection that apply here as well.
First, Algorithm~\ref{algo:ccg:indicator} rigorously identifies if a solution (or the problem instance itself) is infeasible.
Second, it computes rigorous lower bounds $LB$ on the optimal value of \ref{eq:two_stage_ro_indicator}.
Finally, it is possible that the scheme terminates at a solution whose worst-case objective value is estimated to be less than the optimal value of \ref{eq:two_stage_ro_indicator}, since Algorithm~\ref{algo:wc:indicator} computes a Lagrange multiplier that only satisfies the necessary optimality conditions in Theorem~\ref{theorem:necessary_optimality_condition}.
At the end of the algorithm, we can verify if this happens to be the case by solving the true problem $\mathcal{Q}_\mathcal{I}(\bm{x}, \bm{\xi})$.
If it happens to be larger than $UB$, we can keep increasing $\lambda$ (at an exponential rate) in Algorithm~\ref{algo:wc:indicator} until we observe an increase in the value of $\mathcal{L}_\mathcal{I}(\bm{x}, \bm{\xi}, \lambda)$ and the necessary optimality conditions are satisfied.
We can then continue the column-and-constraint generation Algorithm~\ref{algo:ccg:indicator} after updating $UB$ but retaining $LB$ and $\mathcal{R}$ without re-intializing them.
Note that if Algorithm~\ref{theorem:necessary_optimality_condition} outputs the true optimal multiplier $\lambda$, then the true problem $\mathcal{Q}_\mathcal{I}(\bm{x}, \bm{\xi})$ has to be solved only at most once (because of the strong duality result in Theorem~\ref{theorem:strong_duality_two_stage_problem}).
Indeed, this was the case in all our experiments, which makes us suspect that the necessary conditions in Theorem~\ref{theorem:necessary_optimality_condition} are almost sufficient (although one can construct counter-examples).

When the second-stage decisions are continuous ($\mathcal{Y} = \mathbb{R}^{n_2}_{+}$), 
the intermediate optimization problems in lines \ref{algo:wc:indicator:feasible_xi_computation} and \ref{algo:wc:indicator:optimal_xi_computation} of Algorithm~\ref{algo:wc:indicator} can be computed by solving problems \ref{eq:worst_case_problem_indicator_slack_duality} and \ref{eq:worst_case_problem_indicator_lagrangian_duality}, similar to the Benders scheme.
When the second-stage decisions are mixed-integer, \cite{zeng2013solving} have proposed an inner-level column-and-constraint generation algorithm under the assumption of extended relatively complete recourse, which we shall define shortly.
In the following, we briefly describe how our method can be integrated in this inner-level algorithm and without requiring this assumption.
We only present it for the general formulation \ref{eq:two_stage_ro_general} assuming that the conditions of Theorem~\ref{theorem:closed_form_lu} are satisfied so that Algorithm~\ref{algo:wc:closed_form} is applicable.
The extension toward instances of \ref{eq:two_stage_ro_general} and \ref{eq:two_stage_ro_indicator} without these conditions is straightforward but requires additional notation, and so we omit it in the interest of brevity.

Consider the following instance of problem \ref{eq:two_stage_ro_general} with both continuous ($\bm{y}_\mathrm{c}$) and discrete ($\bm{y}_\mathrm{d}$) second-stage decisions.
\begin{equation}\label{eq:two_stage_ro_discrete}
    \begin{aligned}
        &\inf_{\bm{x} \in \mathcal{X}} \sup_{\bm{\xi} \in \Xi} \, \mathcal{Q}(\bm{x}, \bm{\xi}), \\
        & %
        \mathcal{Q}(\bm{x}, \bm{\xi}) = 
        \left[
        \begin{aligned}
            \mathop{\text{minimize}}_{\bm{y} \in \mathcal{Y}} \;\; & \bm{c}(\bm{\xi})^\top \bm{x}  + \bm{d}_\mathrm{c}(\bm{\xi})^\top \bm{y}_\mathrm{c}  + \bm{d}_\mathrm{d}(\bm{\xi})^\top \bm{y}_\mathrm{d} \\
            \text{subject to} \; & \bm{T}\bm{x} + \bm{W}_\mathrm{c} \bm{y}_\mathrm{c} + \bm{W}_\mathrm{d} \bm{y}_\mathrm{d} \geq \bm{h}(\bm{\xi}) \\
            & \bm{y} = (\bm{y}_\mathrm{c}, \bm{y}_\mathrm{d}) \in \mathbb{R}^{nc_2} \times \mathcal{Y}_\mathrm{d}
        \end{aligned}
        \right],
    \end{aligned}
\end{equation}
where $\mathcal{Y} = \mathbb{R}^{nc_2} \times \mathcal{Y}_\mathrm{d}$ and $\mathcal{Y}_\mathrm{d} \subseteq \mathbb{Z}^{nd_2}$ is a compact MILP representable set.

Before we present the algorithm for computing worst-case parameter realizations,
we note that the Lagrangian and constraint violation functions are exactly the same as in the case of problem \ref{eq:two_stage_ro_general} that we presented in Section~\ref{sec:reformulation}.
For example, the former is defined as follows:
\begin{align*}
    \mathcal{L}(\bm{x}, \bm{\xi}, \lambda)
    &=\left[
    \begin{aligned}
        \mathop{\text{minimize}}_{\bm{y} \in \mathcal{Y}, \bm{z} \in \mathbb{R}^{n_p}_{+}} \;\; &
        \bm{c}(\bm{\xi})^\top \bm{x}  + \bm{d}_\mathrm{c}(\bm{\xi})^\top \bm{y}_\mathrm{c}  + \bm{d}_\mathrm{d}(\bm{\xi})^\top \bm{y}_\mathrm{d} + \lambda \phi(\bm{z}, \bm{\xi}) \\
        \text{subject to} \;\; & \bm{T}\bm{x} + \bm{W}_\mathrm{c} \bm{y}_\mathrm{c} + \bm{W}_\mathrm{d} \bm{y}_\mathrm{d} \geq \bm{h}(\bm{z}), \;\; \bm{z} \leq \bm{e}.
    \end{aligned}
    \right]
\end{align*}
with $\mathcal{S}(\bm{x}, \bm{\xi})$ defined similarly.
The key observation in~\cite{zeng2013solving} is to note that the worst-case values of the Lagrangian and constraint violation functions (over $\bm{\xi} \in \Xi$) can be reformulated as semi-infinite programs.
For example, we can reformulate $\sup_{\bm{\xi} \in \Xi} \mathcal{L}(\bm{x}, \bm{\xi}, \lambda)$ as follows:
\begin{gather*}
    \sup_{\bm{\xi} \in \Xi, \eta \in \mathbb{R}} \left\{
    \eta :
    \eta \leq \left[
    \begin{aligned}
        & \bm{c}(\bm{\xi})^\top \bm{x}  + \bm{d}_\mathrm{d}(\bm{\xi})^\top \bm{y}_\mathrm{d} \\
        & + \inf_{(\bm{y}_\mathrm{c}, \bm{z}) \in \Pi(\bm{x}, \bm{y}_\mathrm{d})} \left\{ \bm{d}_\mathrm{c}(\bm{\xi})^\top \bm{y}_\mathrm{c}  + \lambda \phi(\bm{z}, \bm{\xi}) \right\}
    \end{aligned}
    \right], \bm{y}_\mathrm{d} \in  \mathcal{Y}_\mathrm{d}
    \right\}
    \\
    =\sup_{\bm{\xi} \in \Xi, \eta \in \mathbb{R}} \left\{
    \eta :
    \eta \leq \left[
    \begin{aligned}
        & \bm{c}(\bm{\xi})^\top \bm{x}  + \bm{d}_\mathrm{d}(\bm{\xi})^\top \bm{y}_\mathrm{d} \\
        & + \sup_{(\bm{\mu}, \bm{\beta}) \in \Delta(\bm{\xi}, \lambda, 1)} \left\{ \big[\bm{h}_0 - \bm{T}\bm{x} - \bm{W}_\mathrm{d} \bm{y}_\mathrm{d}\big]^\top \bm{\mu}  + \one^\top (\lambda \bm{\xi} - \bm{\beta}) \right\}
    \end{aligned}
    \right], \bm{y}_\mathrm{d} \in  \mathcal{Y}_\mathrm{d}
    \right\}
    \\
    \Pi(\bm{x}, \bm{y}_\mathrm{d}) = \big\{
    (\bm{y}_\mathrm{c}, \bm{z}) \in \mathbb{R}_{+}^{nc_2} \times [0, 1]^{n_p} :
    \bm{W}_\mathrm{c} \bm{y}_\mathrm{c}  - \bm{H} \bm{z} \geq \bm{h}_0 - \bm{T}\bm{x} - \bm{W}_\mathrm{d} \bm{y}_\mathrm{d}
    \big\} \\
    \Delta(\bm{\xi}, \lambda, \kappa) = \big\{
    (\bm{\mu}, \bm{\beta}) \in \mathbb{R}^m_{+} \times \mathbb{R}^{n_p}_{+}:
    \bm{W}^\top \bm{\mu} \leq \kappa \bm{d}_\mathrm{c}(\bm{\xi}), \; 2 \lambda\bm{\xi} - \bm{H}^\top \bm{\mu} - \bm{\beta} \geq \lambda\one
    \big\}
\end{gather*}
and $\sup_{\bm{\xi} \in \Xi}  \mathcal{S}(\bm{x}, \bm{\xi})$ can be reformulated similarly.

Algorithm~\ref{algo:ccg:inner} presents the computation of the worst-case parameter realization in the presence of mixed-integer second-stage decisions $\bm{y}_\mathrm{d}$ in problem~\eqref{eq:two_stage_ro_discrete} assuming that it satisfies the conditions in Theorem~\ref{theorem:closed_form_lu}.
We note that it does not rely on the extended relatively complete recourse assumption, which requires the second-stage problem to be feasible for any fixed values of the first-stage decisions $\bm{x} \in \mathcal{X}$, uncertain parameters $\bm{\xi} \in \Xi$ and discrete-valued second-stage decisions $\bm{y}_\mathrm{d} \in \mathcal{Y}_\mathrm{d}$.
The key idea of the algorithm is to solve the above semi-infinite reformulation by enumerating $\mathcal{Y}_\mathrm{d}$ in the set $\mathcal{D}$.
Indeed, it adds a full set of dual variables $(\bm{\mu}^{(\bm{y}_\mathrm{d})}, \bm{\beta}^{(\bm{y}_\mathrm{d})})$ and associated constraints for each identified $\bm{y}_\mathrm{d} \in \mathcal{D}$ and thus obtains upper bounds on $\sup_{\bm{\xi} \in \Xi}  \mathcal{S}(\bm{x}, \bm{\xi})$ and $\sup_{\bm{\xi} \in \Xi}  \mathcal{L}(\bm{x}, \bm{\xi}, \lambda)$ in lines \ref{algo:ccg:inner:feasibility:ub-update} and \ref{algo:ccg:inner:optimality:ub-update}, respectively.
Note that the optimization problems in these lines can be solved with any MILP solver.
The algorithm terminates in a finite number of iterations because $\mathcal{Y}_\mathrm{d}$ is a finite set.

\begin{algorithm}[!htb]
    \caption{Computation of worst-case parameter realization for \eqref{eq:two_stage_ro_discrete} using the inner-level column-and-constraint generation scheme and Algorithm~\ref{algo:wc:closed_form}}
    \label{algo:ccg:inner}
    \begin{algorithmic}[1]
        \renewcommand{\algorithmicrequire}{\textbf{Input:}}
        \renewcommand{\algorithmicensure}{\textbf{Output:}}
        \REQUIRE $\bm{x} \in \mathcal{X}$ %
        \ENSURE $\hat{\bm{\xi}} \in \mathop{\arg\max}\limits_{\bm{\xi} \in \Xi} \mathcal{Q}(\bm{x}, \bm{\xi})$ and $\hat{\lambda} \in \mathop{\arg\max}\limits_{\lambda \in \mathbb{R}_{+}} \bigg\{\sup\limits_{\bm{\xi} \in \Xi} \mathcal{L}(\bm{x}, \bm{\xi}, \lambda)\bigg\}$
        \STATE Initialize $\hat{\bm{\xi}} \in \Xi$ (arbitrary), $LB = -\infty$, $UB = +\infty$, $\mathcal{D} = \emptyset$, $\hat{\lambda} = 1$
        \REPEAT
        \STATE Set $UB$ and $\hat{\bm{\xi}}$ as the optimal objective value and (projected) solution of the problem:
        \begin{equation*}%
            \inf_{\bm{x}, \eta, \bm{y}} \left\{
            \eta :
            \begin{aligned}
                & \bm{\xi} \in \Xi, \;\; \eta \in \mathbb{R}, \;\; (\bm{\mu}^{(\bm{y}_\mathrm{d})}, \bm{\beta}^{(\bm{y}_\mathrm{d})}) \in \Delta(\bm{\xi}, 1, 0), \;\; \bm{\mu}^{(\bm{y}_\mathrm{d})} \leq \one, \; \bm{y}_\mathrm{d} \in \mathcal{D} \\
                & \eta \leq \big(\bm{h}_0 - \bm{T}\bm{x} - \bm{W}_\mathrm{d} \bm{y}_\mathrm{d}\big)^\top \bm{\mu}^{(\bm{y}_\mathrm{d})}  + \one^\top (\bm{\xi} - \bm{\beta}^{(\bm{y}_\mathrm{d})}), \; \bm{y}_\mathrm{d} \in \mathcal{D} %
            \end{aligned}
            \right\}
        \end{equation*}
        \label{algo:ccg:inner:feasibility:ub-update}
        \STATE Set 
        $
        (\hat{\bm{y}}, \hat{\bm{z}}, \hat{\bm{\sigma}}) \in \mathop{\arg\min}\limits_{(\bm{y}, \bm{z}, \bm{\sigma}) \in \mathcal{Y} \times [0, 1]^{n_p} \times \mathbb{R}^m_{+}}
        \big\{\one^\top \bm{\sigma} + \phi(\bm{z}, \hat{\bm{\xi}}) : \bm{T}\bm{x} + \bm{W}_\mathrm{c} \bm{y}_\mathrm{c} + \bm{W}_\mathrm{d} \bm{y}_\mathrm{d} + \bm{\sigma} \geq \bm{h}(\bm{z}) \big\}
        $
        \label{algo:ccg:inner:feasibility:lb-update}
        \STATE Update $\mathcal{D} \gets \mathcal{D} \cup \{\hat{\bm{y}}_\mathrm{d}\}$.
        \STATE Update $LB = \one^\top \hat{\bm{\sigma}} + \phi(\hat{\bm{z}}, \hat{\bm{\xi}})$
        \UNTIL {$LB > 0$ or $UB = 0$}
        \IF {$UB = 0$}
        \STATE Compute $\ell(\bm{x})$ and $u(\bm{x})$ as per Theorem~\ref{theorem:optimal_multiplier_general}
        \STATE Set $\hat{\lambda} = u(\bm{x}) - \ell(\bm{x})$ and $LB = -\infty$
        \REPEAT
        \STATE Set $UB$ and $\tilde{\bm{\xi}}$ as the optimal objective value and (projected) solution of the problem: 
        \begin{equation*}%
            \inf_{\bm{x}, \eta, \bm{y}} \left\{
            \eta :
            \begin{aligned}
                & \bm{\xi} \in \Xi, \;\; \eta \in \mathbb{R}, \;\; (\bm{\mu}^{(\bm{y}_\mathrm{d})}, \bm{\beta}^{(\bm{y}_\mathrm{d})}) \in \Delta(\bm{\xi}, \hat{\lambda}, 1), \; \bm{y}_\mathrm{d} \in \mathcal{D} \\
                & \eta \leq \bm{c}(\bm{\xi})^\top\bm{x} + \bm{d}_\mathrm{d}(\bm{\xi})^\top \bm{y}_\mathrm{d} + \big(\bm{h}_0 - \bm{T}\bm{x} - \bm{W}_\mathrm{d} \bm{y}_\mathrm{d}\big)^\top \bm{\mu}^{(\bm{y}_\mathrm{d})}  + \one^\top (\hat{\lambda} \bm{\xi} - \bm{\beta}^{(\bm{y}_\mathrm{d})}), \; \bm{y}_\mathrm{d} \in \mathcal{D} %
            \end{aligned}
            \right\}
        \end{equation*}
        \label{algo:ccg:inner:optimality:ub-update}
        \STATE Compute 
        $
        \hat{\bm{y}} \in \mathop{\arg\min}\limits_{\bm{y} \in \mathcal{Y}}
        \big\{\bm{d}_\mathrm{c}(\tilde{\bm{\xi}})^\top \bm{y}_\mathrm{c}  + \bm{d}_\mathrm{d}(\tilde{\bm{\xi}})^\top \bm{y}_\mathrm{d} : \bm{T}\bm{x} + \bm{W}_\mathrm{c} \bm{y}_\mathrm{c} + \bm{W}_\mathrm{d} \bm{y}_\mathrm{d} \geq \bm{h}(\tilde{\bm{\xi}}) \big\}
        $
        \STATE Update $\mathcal{D} \gets \mathcal{D} \cup \{\hat{\bm{y}}_\mathrm{d}\}$.
        \IIf {$LB < \bm{d}_\mathrm{c}(\tilde{\bm{\xi}})^\top \hat{\bm{y}}_\mathrm{c}  + \bm{d}_\mathrm{d}(\tilde{\bm{\xi}})^\top \hat{\bm{y}}_\mathrm{d}$}
        $LB \gets \bm{d}_\mathrm{c}(\tilde{\bm{\xi}})^\top \hat{\bm{y}}_\mathrm{c}  + \bm{d}_\mathrm{d}(\tilde{\bm{\xi}})^\top \hat{\bm{y}}_\mathrm{d}$,
        $\hat{\bm{\xi}} \gets \tilde{\bm{\xi}}$
        \EndIIf
        \label{algo:ccg:inner:optimality:lb-update}
        \UNTIL {$UB - LB \leq \epsilon$}
        \ENDIF
    \end{algorithmic}
\end{algorithm}

\section{Numerical Experiments}\label{sec:experiments}

The primary goal of our numerical experiments is to measure the computational benefits of using the Lagrangian formulation over traditional methods in computing worst-case parameter realizations, whereas the secondary goal is to assess its impacts on the run time of the overall exact method.
To that end, we consider three different applications: network design with uncertain edge failures, facility location with random facility disruptions, and staff rostering under demand uncertainty.
Collectively, they represent instances of both \ref{eq:two_stage_ro_general} and \ref{eq:two_stage_ro_indicator}, feature the presence and absence of relatively complete recourse as well as continuous and mixed-integer second-stage decisions.
Notably, all of these applications have been previously studied in the literature and problem-specific column-and-constraint generation methods have been developed for their solution.

The experiments were implemented in Julia using JuMP~\cite{DunningHuchetteLubin2017} as the modeling language and Gurobi~9.1~\cite{gurobi} (with default options) as the LP and MILP solver.
All runs were conducted on an Intel~Xeon~2.3 GHz computer, with a limit of eight cores per run.
Our code is available at \url{https://github.com/AnirudhSubramanyam/RobustOptLagrangianDual.jl}.

\subsection{Network design with uncertain edge failures}
We first consider the network design problem that was studied in~\cite{matthews2019designing}.
Consider a network with nodes $V$ and edges $E$.
Each node $i \in V$ has degree $D_i$ and it is associated with a (positive or negative) demand $b_i \in \mathbb{R}$ for a single commodity.
The origin and destination nodes of each edge $e \in E$ are denoted by $o(e) \in V$ and $d(e) \in V$, respectively,
whereas its pre-installed capacity is denoted by $p_e \in \mathbb{R}_{+}$.
Each edge may fail randomly and the uncertain parameter $\xi_e \in \{0, 1\}$ indicates whether the edge $e$ has failed.
The goal is to decide (before any failures occur) the amount of additional capacity $x_e \in \mathbb{R}_{+}$ that must be installed for each edge $e \in E$,
so that all nodal demands can be satisfied using feasible forward and reverse edge flows, denoted $y^f_e \in \mathbb{R}_{+}$ and $y^b_e \in \mathbb{R}_{+}$, respectively, under any postulated realization of the edge failures $\bm{\xi} \in \Xi$.
This problem can be formulated as the following instance of \ref{eq:two_stage_ro_indicator}.
\begin{equation*}
    \begin{aligned}
        &\inf_{\bm{x} \in \mathbb{R}_{+}^{|E|}} \sup_{\bm{\xi} \in \Xi} \, \mathcal{Q}_\mathcal{I}(\bm{x}, \bm{\xi}), \\
        & %
        \mathcal{Q}_\mathcal{I}(\bm{x}, \bm{\xi}) = 
        \left[\begin{aligned}
            \mathop{\text{minimize}}_{\bm{y}^f, \bm{y}^b \in \mathbb{R}^{|E|}_{+}} \;\; & (\bm{c}^x)^\top \bm{x}  + (\bm{c}^y)^\top (\bm{y}^f + \bm{y}^b)  \\
            \text{subject to} \;\; &  \smashoperator{\sum_{e \in E: d(e) = i}} \left(y^f_e - y^b_e\right) + \smashoperator{\sum_{e \in E: o(e) = i}} \left(y^b_e - y^f_e\right) = b_i, \;\; i \in V \\
            & \bm{y}^f \leq \bm{x} + \bm{p}, \;\; \bm{y}^b \leq \bm{x} + \bm{p}, \\
            & \xi_e = 1 \implies y^f_e = 0, \;\; e \in E \\
            & \xi_e = 1 \implies y^b_e = 0, \;\; e \in E
        \end{aligned}\right],
    \end{aligned}
\end{equation*} 
where for each edge $e \in E$, $c_e^x \in \mathbb{R}$ and $c_e^y \in \mathbb{R}$ represent the per-unit capacity installation and transportation costs, respectively.
The uncertainty set is parameterized by a budget of uncertainty $k \geq 0$, and defined as follows:
\begin{equation*}
    \Xi = \left\{
    \bm{\xi} \in \{0, 1\}^{|E|}: \one^\top \bm{\xi} \leq k, \;
    {\sum_{e \in E: d(e) = i}} \xi_e + {\sum_{e \in E: o(e) = i}} \xi_e \leq D_i - 1, \; i \in V_T
    \right\},
\end{equation*}
where $V_T \subseteq V$ denotes the subset of terminal nodes; that is, those with non-zero demands.
Thus, the uncertainty set models the simultaneous failure of up to $k$ edges at any moment.
We refer readers to \cite{matthews2019designing} for the motivation behind the choice of this uncertainty set, the larger context of the network design problem, as well as additional details about the problem formulation.

In \cite{matthews2019designing}, the authors use a big-M representation of the indicator constraints as follows:
$y^f_e, y^b_e \leq (x_e + p_e)(1 - \xi_e)$ for each $e \in E$.
Nevertheless, we retain the indicators since our method can exploit that structure.
Also, we note that this problem does not satisfy the relatively complete recourse assumption.

We conduct experiments on the same set of instances as \cite{matthews2019designing} with the caveat that we do not preprocess the original network superstructure as described in Appendix~B of that paper.
Also, we set $c_e^x/c_e^y = 1$ for all $e \in E$, whereas \cite{matthews2019designing} study a range of different values for this ratio (including the one that we consider).
Apart from these differences, we benchmark our method against the one proposed in \cite{matthews2019designing}, which is a variant of the original column-and-constraint generation algorithm proposed in \cite{zhao2012exact}.
Specifically, \cite{matthews2019designing} modify the original algorithm to account for the lack of relatively complete recourse and they use the duality-based mixed-integer bilinear formulation \eqref{eq:worst_case_problem_duality} for computing worst-case parameter realizations; the latter is further reformulated using additional variables and indicator constraints (see Section~\ref{sec:formulation:literature:bilinear}) and solved using an MILP solver.
We benchmark (our implementation) of this method against Algorithm~\ref{algo:ccg:indicator}, which uses Algorithm~\ref{algo:wc:indicator} for computing worst-case parameter realizations.
In both cases, we set a total time limit of 1~hour.

Table~\ref{table:network} summarizes and compares the computational performance of the existing algorithm in \cite{matthews2019designing} against our proposed method for three representative test instances, since all other instances are solved in less than two minutes on average.
In this table, the columns \# It., $t_\text{tot}$ and $t^\text{avg}_\text{wc}$ report the total number of iterations of the column-and-constraint generation algorithm, the total run time of the algorithm, and the average time to compute worst-case parameter realizations, respectively.
The last column $t^\text{speedup}_\text{wc}$ reports the average speedup with respect to the time for computing worst-case realizations and it is simply the ratio of the corresponding $t^\text{avg}_\text{wc}$ values.
We note that in all instances, the largest reported $k$ value corresponds to the smallest uncertainty set for which that instance becomes infeasible, and the corresponding column $t_\text{tot}$ reports the time to prove infeasibility.
Figure~\ref{figure:network} graphically shows the same quantities but averages them across all instances for each $k$.

\begin{table}[!hp]
    \centering
    \caption{Computational comparison of the algorithm in \cite{matthews2019designing} and the proposed method for the network design problem.}
    \label{table:network}
    \begin{tabularx}{\textwidth}{r*{6}{R}r}
        \toprule
        & \multicolumn{3}{c}{Algorithm in \cite{matthews2019designing}} & \multicolumn{3}{c}{Proposed method} & \\
        \cmidrule(r){2-4}\cmidrule(l){5-7}
        $k$ & \# It. & $t_\text{tot}$ (s) & $t^\text{avg}_\text{wc}$ (s) & \# It. & $t_\text{tot}$ (s) & $t^\text{avg}_\text{wc}$ (s) & $t^\text{speedup}_\text{wc}$ \\
        \midrule
        \multicolumn{8}{c}{Instance dfn-bwin}\\
        1 &      10 &     0.4 &    0.03 &      10 &     0.2 &    0.01 &     2.5 \\
        2 &      32 &     3.0 &    0.08 &      33 &     1.3 &    0.03 &     2.8 \\
        3 &      64 &    23.9 &    0.33 &      65 &     5.3 &    0.06 &     5.2 \\
        4 &      73 &    67.7 &    0.86 &      80 &    10.5 &    0.11 &     7.9 \\
        5 &      82 &   124.6 &    1.46 &      86 &    13.6 &    0.13 &    10.8 \\
        6 &      73 &   565.8 &    7.33 &      64 &    12.3 &    0.17 &    42.2 \\
        7 &      67 &   271.3 &    3.96 &      63 &     9.5 &    0.13 &    30.0 \\
        8 &      71 &   270.1 &    3.74 &      69 &     7.2 &    0.09 &    40.9 \\
        9 &      86 &  1462.0 &   16.60 &      88 &    17.7 &    0.18 &    93.4 \\
        10 &     102 &  1873.9 &   18.17 &      87 &    32.7 &    0.35 &    52.2 \\
        11 &     117 &  TL$^*$ &   30.49 &     125 &   100.4 &    0.77 &    39.5 \\
        12 &     106 &  TL$^*$ &   33.63 &      99 &   119.9 &    1.17 &    28.7 \\
        13 &      92 &  TL$^*$ &   38.69 &      85 &   253.6 &    2.90 &    13.3 \\
        14 &      74 &  TL$^*$ &   47.98 &      66 &   118.4 &    1.73 &    27.7 \\
        15 &      83 &  TL$^*$ &   42.84 &      73 &    58.4 &    0.78 &    55.1 \\
        16 &      21 &     0.8 &    0.02 &      13 &     0.4 &    0.01 &     1.6 \\
        \midrule
        \multicolumn{8}{c}{Instance dfn-gwin}\\
        1 &      12 &     1.0 &    0.03 &      12 &     0.8 &    0.01 &     3.1 \\
        2 &      35 &     3.7 &    0.08 &      32 &     1.8 &    0.03 &     2.9 \\
        3 &      47 &     7.8 &    0.14 &      60 &     4.6 &    0.05 &     2.9 \\
        4 &      82 &    62.7 &    0.70 &      77 &     9.6 &    0.09 &     8.1 \\
        5 &      81 &   161.5 &    1.90 &      83 &    12.1 &    0.11 &    17.0 \\
        6 &      83 &   731.7 &    8.39 &      80 &    18.5 &    0.19 &    45.3 \\
        7 &      93 &   826.4 &    8.59 &      90 &    25.7 &    0.24 &    36.2 \\
        8 &     118 &   658.0 &    5.46 &     105 &    12.4 &    0.09 &    60.2 \\
        9 &      98 &   851.6 &    8.50 &     102 &    13.3 &    0.10 &    84.3 \\
        10 &      17 &     1.4 &    0.03 &      16 &     1.1 &    0.01 &     2.4 \\
        \midrule
        \multicolumn{8}{c}{Instance di-yuan}\\
        1 &      14 &     0.5 &    0.03 &      15 &     0.3 &    0.01 &     4.5 \\
        2 &      37 &     4.2 &    0.10 &      42 &     2.1 &    0.03 &     3.6 \\
        3 &      49 &    11.0 &    0.20 &      51 &     3.2 &    0.04 &     5.3 \\
        4 &      72 &    19.5 &    0.25 &      80 &     6.3 &    0.05 &     4.9 \\
        5 &      80 &    38.9 &    0.46 &      86 &     7.2 &    0.05 &     8.5 \\
        6 &     100 &    54.3 &    0.52 &     103 &     6.5 &    0.04 &    13.7 \\
        7 &     137 &    97.4 &    0.68 &     134 &    10.1 &    0.04 &    15.6 \\
        8 &     120 &   231.6 &    1.89 &     138 &    12.3 &    0.05 &    35.2 \\
        9 &     133 &   528.9 &    3.92 &     124 &    14.0 &    0.08 &    51.3 \\
        10 &     129 &  1231.7 &    9.45 &     117 &    22.3 &    0.15 &    63.7 \\
        11 &     119 &  2518.6 &   20.97 &     112 &    43.9 &    0.32 &    66.5 \\
        12 &     119 &    10.6 &    0.06 &     124 &     5.7 &    0.02 &     2.6 \\
        \midrule
        Avg &    74.4$^\dagger$ &   385.3$^\dagger$ &    8.38 &    74.0$^\dagger$ &    10.4$^\dagger$ &    0.27 &    30.6 \\
        \bottomrule
        \multicolumn{8}{l}{\footnotesize$^*$ Reached time limit of 1~hour without providing a feasible solution.} \\
        \multicolumn{8}{l}{\footnotesize$^\dagger$ Averaged across instances where both methods terminated within the time limit.}
    \end{tabularx}
\end{table}

\begin{figure}[!htb]
    \centering
    \begin{subfigure}[b]{0.49\linewidth}
        \includegraphics[height=0.9\linewidth]{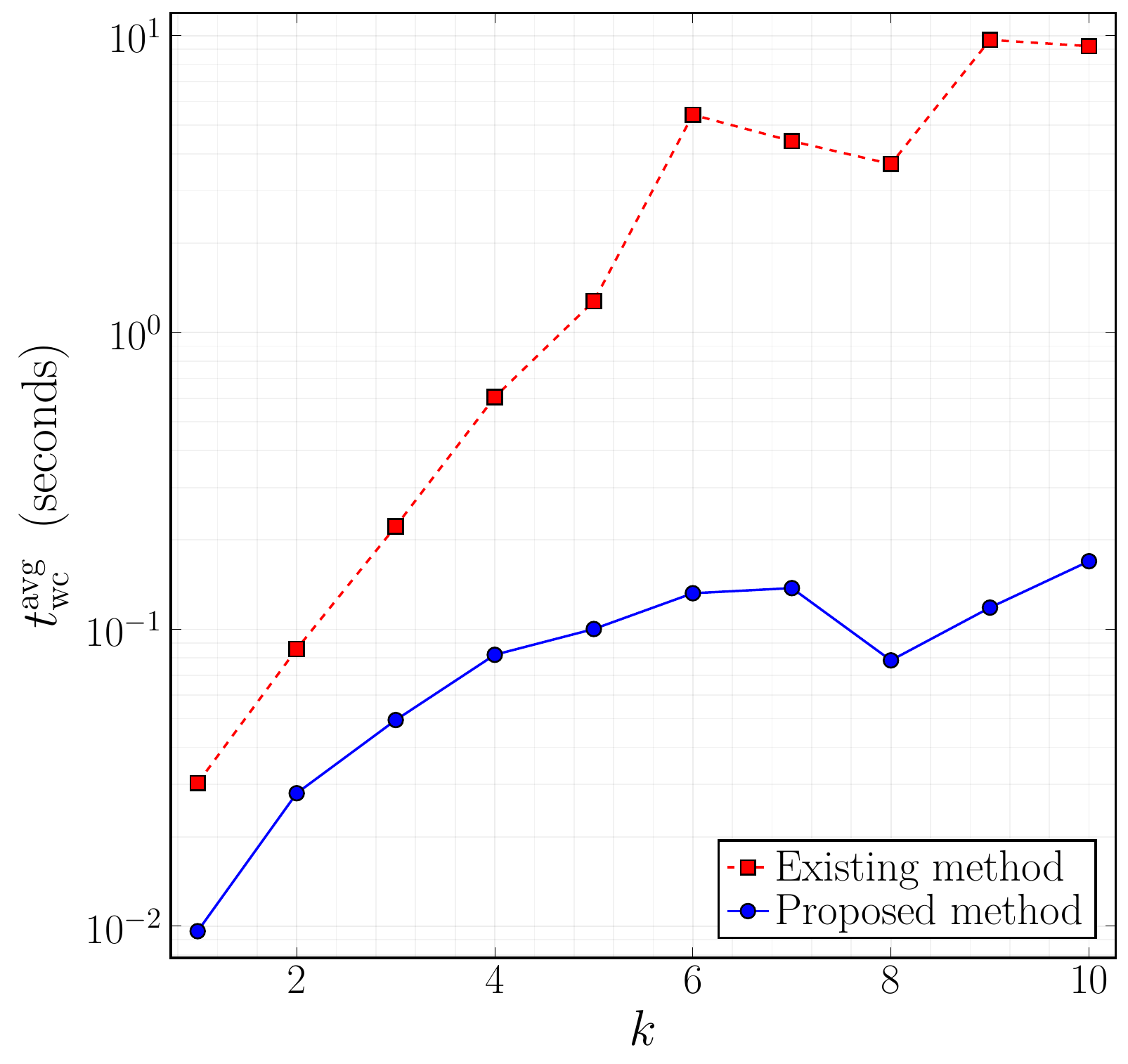}
        \caption{Average time for worst-case realizations}\label{figure:network:twc}
    \end{subfigure}\hfil
    \begin{subfigure}[b]{0.49\linewidth}
        \includegraphics[height=0.9\linewidth]{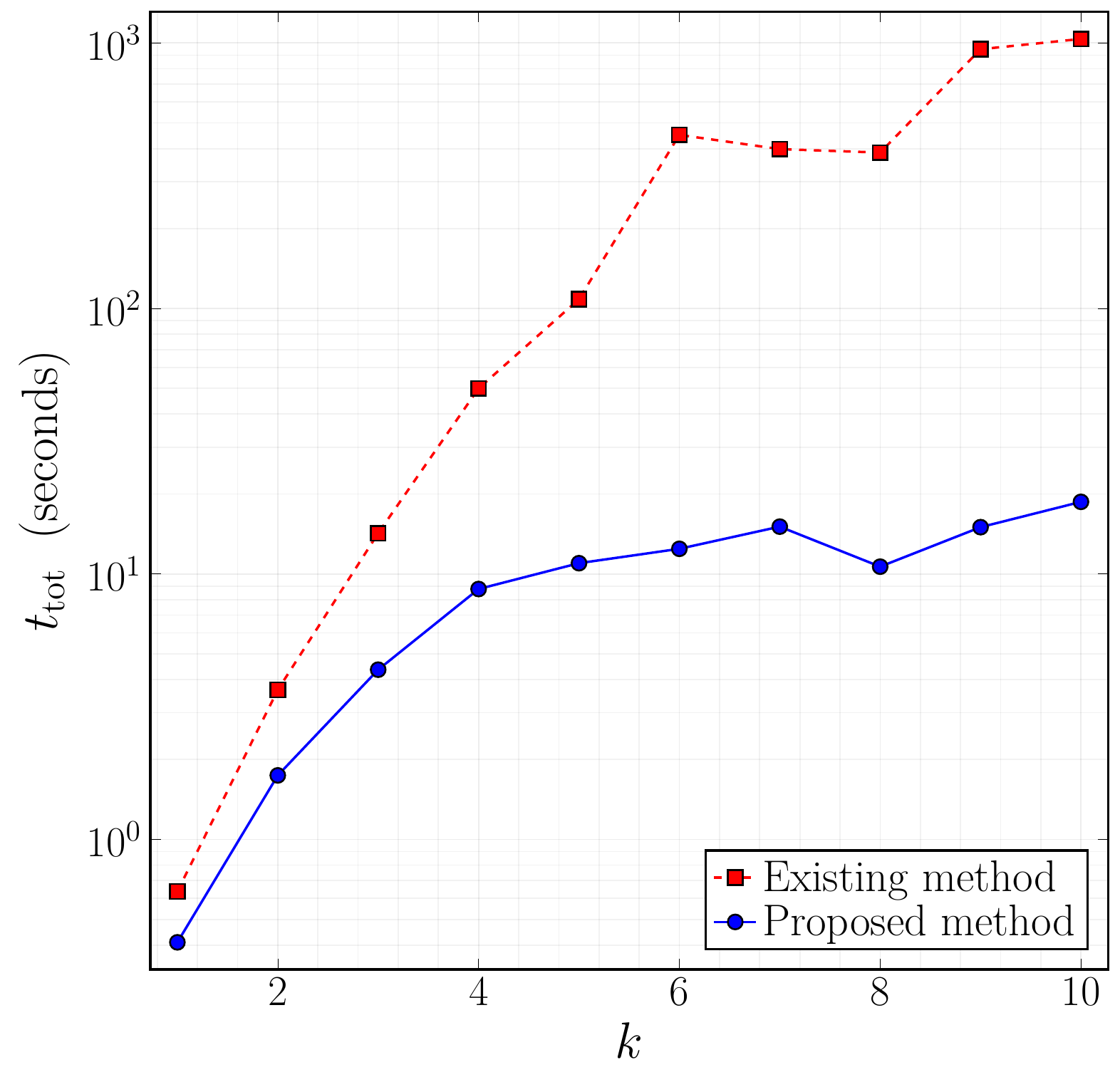}
        \caption{Total run time of algorithm}\label{figure:network:total}
    \end{subfigure}
    \caption{Comparison of computational times of the existing method in \cite{matthews2019designing} against the proposed method for the network design problem.}
    \label{figure:network}
\end{figure}

Table~\ref{table:network} and Figure~\ref{figure:network} show that the proposed method can offer an average speedup of more than a factor of $30$ for computing worst-case parameter realizations.
Since the majority of the total run time is spent for this purpose, this directly translates to a speedup in the total run time of the algorithm as well.
In particular, we observe that the proposed method can also compute optimal solutions (within 2~minutes) for cases where the existing method can not find a feasible solution in the time limit of 1~hour.
Finally, we can observe from Table~\ref{table:network} that that the two methods do not necessarily execute the same number of column-and-constraint generation iterations.
The main reason for this difference is that whenever a candidate first-stage decision $\bm{x}$ is infeasible, the parameter realization that is optimal in the worst-case constraint violation formulation~\eqref{eq:worst_case_problem_indicator_slack_duality} is not necessarily the same as that identified by the method of \cite{matthews2019designing}.
In fact, even when $\bm{x}$ is feasible, the two methods may compute different worst-case parameter realizations if the second-stage problem $\mathcal{Q}_I(\bm{x}, \bm{\xi})$ does not happen to have a unique optimal solution.

\subsection{Facility location with random facility disruptions}
We now consider the facility location problem that was studied in~\cite{cheng2021robust}.
Let $I$ and $J$ denote sets of customers and facility locations, respectively.
Each customer $i \in I$ is associated with some demand $h_i \in \mathbb{R}_{+}$ and a per-unit penalty cost $p_i \in \mathbb{R}_{+}$ for failure to satisfy its demand.
Each location $j \in J$ is associated with a fixed cost $c_j \in \mathbb{R}_{+}$ and capacity $C_j \in \mathbb{R}_{+}$ if a facility is built at that location.
Finally, $d_{ij} \in \mathbb{R}_{+}$ represents the per-unit transportation cost if a facility built at location $j \in J$ satisfies some portion of the demand of customer $i \in I$.
Once built, facilities may experience random disruptions and the uncertain parameter $\xi_j \in \{0, 1\}$ indicates whether facility $j \in J$ has been disrupted.
The goal is to decide (before any disruptions occur) the subset of locations where facilities must be built,
so that the sum of total transportation and penalty costs is minimized in the worst-case disruption scenario.
This problem can be formulated as the following instance of \ref{eq:two_stage_ro_indicator}.
\begin{equation*}%
    \begin{aligned}
        &\inf_{\bm{x} \in \{0, 1\}^{|J|}} \sup_{\bm{\xi} \in \Xi} \, \mathcal{Q}_\mathcal{I}(\bm{x}, \bm{\xi}), \\
        & %
        \mathcal{Q}_\mathcal{I}(\bm{x}, \bm{\xi}) = 
        \left[\begin{aligned}
            \mathop{\text{minimize}}_{\bm{y} \in \mathbb{R}_{+}^{|I| \times |J|}, \bm{u} \in \mathbb{R}^{|I|}_{+}} & \bm{c}^\top \bm{x}  + \bm{d}^\top \bm{y} + \bm{p}^\top \bm{u}  \\
            \text{subject to} \quad &  \sum_{j \in J} y_{ij} + u_i \geq h_i, \;\; i \in I \\
            &  \sum_{i \in I} y_{ij} \leq C_j x_j, \;\; j \in J \\
            & \xi_j = 1 \implies \sum_{i \in I} y_{ij} = 0, \;\; j \in J
        \end{aligned}\right],
    \end{aligned}
\end{equation*}
where the uncertainty set is parameterized by a budget of uncertainty $k \geq 0$, and defined as follows:
\begin{equation*}
    \Xi = \left\{
    \bm{\xi} \in \{0, 1\}^{|J|}: \one^\top \bm{\xi} \leq k
    \right\}.
\end{equation*}
Thus the set models the simultaneous disruption of up to $k$ facilities.
We note that \cite{chen2009uncertain} model the indicator constraints in the second-stage problem as follows:
$\sum_{i \in I} y_{ij} \leq C_j x_j(1 - \xi_j)$ for each $j \in J$;
but we retain the above structure since our method can exploit it.
We refer readers to \cite{cheng2021robust} for additional details about the problem and model formulation.

We conduct experiments on the same set of instances as \cite{cheng2021robust} with the exception that we do not consider uncertainty in customer demands, which that work also addresses.
They develop a column-and-constraint generation method for solving the problem.
Similar to \cite{zhao2012exact}, they propose to use the KKT-based mixed-integer bilinear formulation \eqref{eq:worst_case_problem_kkt} for computing worst-case parameter realizations; the latter is further reformulated by linearizing the bilinear terms using problem-specific big-M coefficients and solved using an MILP solver.
We benchmark (our implementation) of this method against Algorithm~\ref{algo:ccg:indicator}, which uses Algorithm~\ref{algo:wc:indicator} for computing worst-case parameter realizations.
Since the problem satisfies relatively complete recourse by construction, we skip the computation of the constraint violation functions.
In all cases, we set a total time limit of 2~hours.

Table~\ref{table:facility:instance} summarizes and compares the computational performance of the existing algorithm in \cite{cheng2021robust} against our proposed method, averaged across all test instances with uncertainty set parameter $k = 4$.
The columns have the same interpretation as in Table~\ref{table:network}.
Table~\ref{table:facility:budget} summarizes the computational performance across the same test instances as a function of the uncertainty set parameter $k$, and also includes the average optimality gap, which we define as $100\% \times (UB - LB)/UB$.
Finally, Figure~\ref{figure:network} graphically shows the same quantities but averages them as a function of the instance size; that is, the number of customers $|I|$.

\begin{table}[!htb]
    \centering
    \caption{Computational comparison of the algorithm in \cite{cheng2021robust} and the proposed method for the facility location problem ($k = 4$).}
    \label{table:facility:instance}
    \begin{tabularx}{\textwidth}{r*{6}{R}r}
        \toprule
        & \multicolumn{3}{c}{Algorithm in \cite{cheng2021robust}} & \multicolumn{3}{c}{Proposed method} & \\
        \cmidrule(r){2-4}\cmidrule(l){5-7}
        $(|J|, |I|)$ & \# It. & $t_\text{tot}$ (s) & $t^\text{avg}_\text{wc}$ (s) & \# It. & $t_\text{tot}$ (s) & $t^\text{avg}_\text{wc}$ (s) & $t^\text{speedup}_\text{wc}$ \\
        \midrule
        (10,10) &      16 &     8.4 &    0.46 &      16 &     1.4 &    0.01 &    38.9 \\
        (10,15) &      17 &    19.1 &    1.04 &      17 &     2.1 &    0.02 &    46.4 \\
        (10,20) &       7 &    12.2 &    1.71 &       7 &     0.6 &    0.02 &    98.4 \\
        (10,25) &       9 &    20.9 &    2.26 &       9 &     1.0 &    0.02 &   123.6 \\
        (10,30) &       8 &    58.0 &    7.19 &       8 &     0.9 &    0.02 &   360.2 \\
        (10,35) &       7 &    68.5 &    9.75 &       7 &     0.8 &    0.02 &   496.6 \\
        (10,40) &       9 &   110.3 &   12.20 &       9 &     0.9 &    0.02 &   557.2 \\
        (10,45) &       8 &   157.1 &   19.58 &       8 &     1.1 &    0.03 &   705.2 \\
        (10,49) &       8 &   146.9 &   18.28 &       8 &     1.4 &    0.03 &   586.3 \\
        \midrule
        (15,15) &      27 &    79.1 &    2.39 &      27 &    15.0 &    0.03 &    83.4 \\
        (15,20) &      23 &   104.0 &    4.06 &      23 &    11.9 &    0.03 &   152.0 \\
        (15,25) &      41 &   253.7 &    4.27 &      41 &    84.4 &    0.04 &   104.8 \\
        (15,30) &      20 &   168.8 &    7.93 &      20 &    11.1 &    0.03 &   235.2 \\
        (15,35) &      40 &   541.4 &   11.04 &      40 &   100.2 &    0.05 &   244.6 \\
        (15,40) &      46 &   844.8 &   15.14 &      46 &   157.9 &    0.06 &   253.8 \\
        (15,45) &      32 &   608.8 &   16.55 &      32 &    81.2 &    0.06 &   296.9 \\
        (15,49) &      55 &  1341.6 &   19.63 &      55 &   268.4 &    0.07 &   282.0 \\
        \midrule
        (20,20) &     130 &  5510.2 &    4.33 &     130 &  4624.7 &    0.04 &    98.3 \\
        (20,25) &      98 &  4023.5 &   13.44 &      98 &  2628.2 &    0.06 &   216.4 \\
        (20,30) &     108 &  6090.0 &   11.26 &     108 &  4963.4 &    0.06 &   198.7 \\
        (20,35) &     105 &      TL &   14.45 &     108 &  6030.6 &    0.07 &   219.8 \\
        (20,40) &     100 &      TL &   17.94 &     100 &  5454.4 &    0.07 &   272.5 \\
        (20,45) &      93 &  6531.9 &   17.52 &      93 &  5216.0 &    0.07 &   238.8 \\
        (20,49) &      76 &      TL &   24.57 &      84 &      TL &    0.08 &   323.8 \\
        \midrule
        (25,25) &      91 &      TL &    8.63 &      95 &      TL &    0.05 &   170.2 \\
        (25,30) &      84 &      TL &   12.46 &      91 &      TL &    0.05 &   237.4 \\
        (25,35) &      81 &      TL &   11.98 &      86 &      TL &    0.05 &   219.2 \\
        (25,40) &      77 &      TL &   18.63 &      85 &      TL &    0.06 &   310.2 \\
        (25,45) &      69 &      TL &   27.52 &      76 &      TL &    0.07 &   385.2 \\
        (25,49) &      68 &      TL &   31.06 &      77 &      TL &    0.07 &   422.4 \\
        \midrule
        (30,30) &      59 &      TL &   14.24 &      61 &      TL &    0.05 &   267.5 \\
        (30,35) &      50 &      TL &   17.40 &      53 &      TL &    0.06 &   306.7 \\
        (30,40) &      45 &      TL &   18.43 &      46 &      TL &    0.06 &   290.6 \\
        (30,45) &      43 &      TL &   25.60 &      46 &      TL &    0.06 &   446.8 \\
        (30,49) &      40 &      TL &   25.95 &      43 &      TL &    0.06 &   424.9 \\
        \midrule
        Avg     &    38.2$^\dagger$ &  1271.4$^\dagger$ &   13.40 &    38.2$^\dagger$ &   865.4$^\dagger$ &    0.05 &   284.4 \\
        \bottomrule
        \multicolumn{8}{l}{\footnotesize$^\dagger$ Averaged across instances where both methods terminated within the time limit.}
    \end{tabularx}
\end{table}

\begin{table}[!htb]
    \centering
    \caption{Computational comparison of the algorithm in \cite{cheng2021robust} and the proposed method for the facility location problem.}
    \label{table:facility:budget}
    \begin{tabularx}{\textwidth}{rr*{3}{R}r*{3}{R}r}
        \toprule
        & \multicolumn{4}{c}{Algorithm in \cite{cheng2021robust}} & \multicolumn{4}{c}{Proposed method} & \\
        \cmidrule(r){2-5}\cmidrule(l){6-9}
        $k$ & Opt & Gap (\%) & $t_\text{tot}^\dagger$ (s) & $t^\text{avg}_\text{wc}$ (s) & Opt & Gap (\%) & $t_\text{tot}^\dagger$ (s) & $t^\text{avg}_\text{wc}$ (s) & $t^\text{speedup}_\text{wc}$ \\
        \midrule
        1 &  35/35  &    0.00 &    15.7 &    1.33 &  35/35  &    0.00 &     6.3 &    0.01 &   114.6 \\
        2 &  35/35  &    0.00 &   296.9 &    4.55 &  35/35  &    0.00 &   202.6 &    0.02 &   207.3 \\
        3 &  28/35  &    1.08 &  1073.9 &    9.33 &  29/35  &    0.87 &   816.8 &    0.04 &   244.2 \\
        4 &  21/35  &    5.47 &  1271.4 &   13.40 &  23/35  &    5.19 &   865.4 &    0.05 &   284.4 \\
        \midrule
          & 119/140  &    1.64 &   569.0 &    7.15 & 122/140  &    1.51 &   406.4 &    0.03 &   240.6 \\
        \bottomrule
        \multicolumn{10}{l}{\footnotesize$^\dagger$ Averaged across instances where both methods terminated within the time limit.}
    \end{tabularx}
\end{table}

\begin{figure}[!htb]
    \centering
    \begin{subfigure}[b]{0.49\linewidth}
        \includegraphics[height=0.9\linewidth]{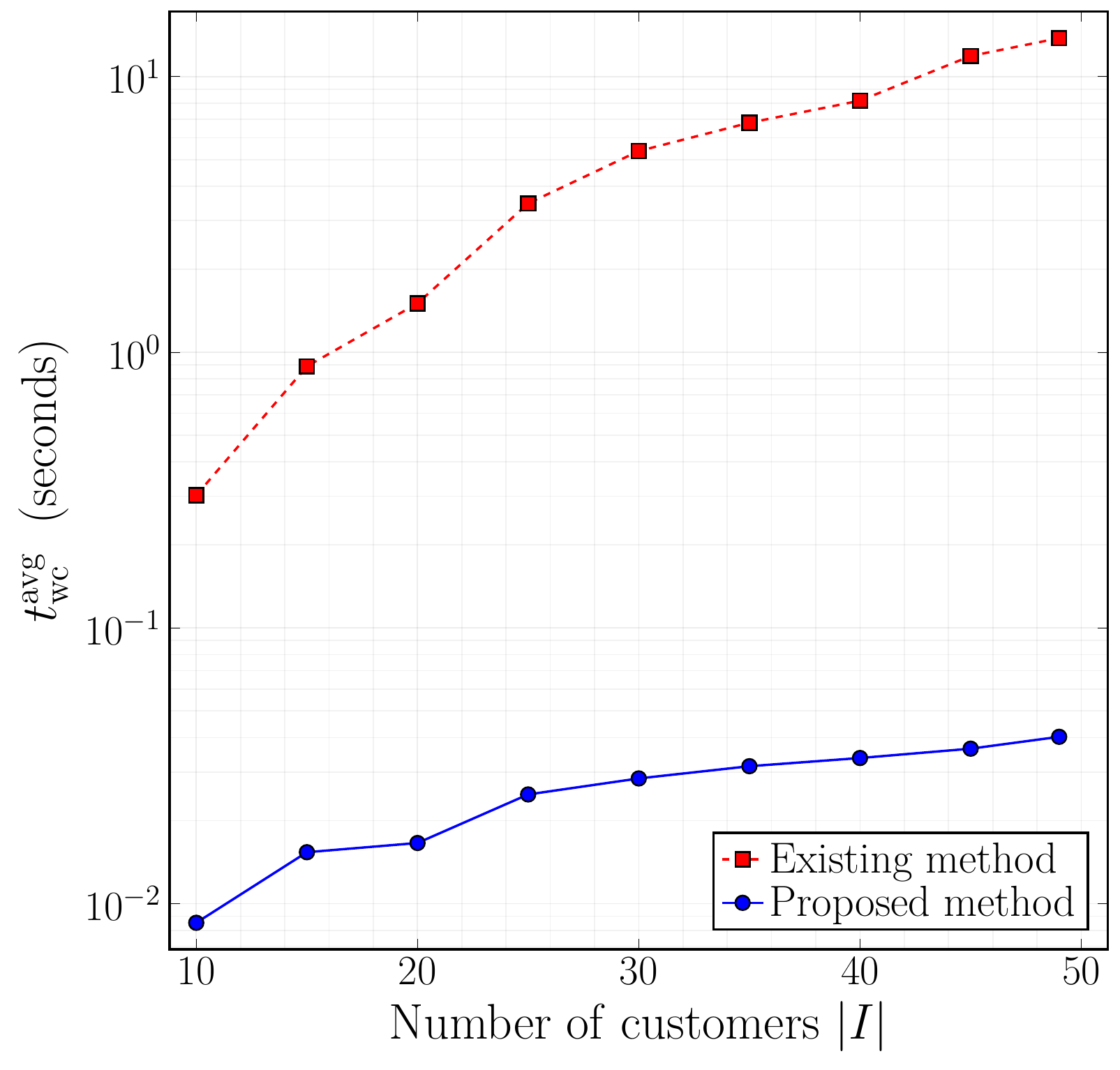}
        \caption{Average time for worst-case realizations}\label{figure:facility:twc}
    \end{subfigure}\hfil
    \begin{subfigure}[b]{0.49\linewidth}
        \includegraphics[height=0.9\linewidth]{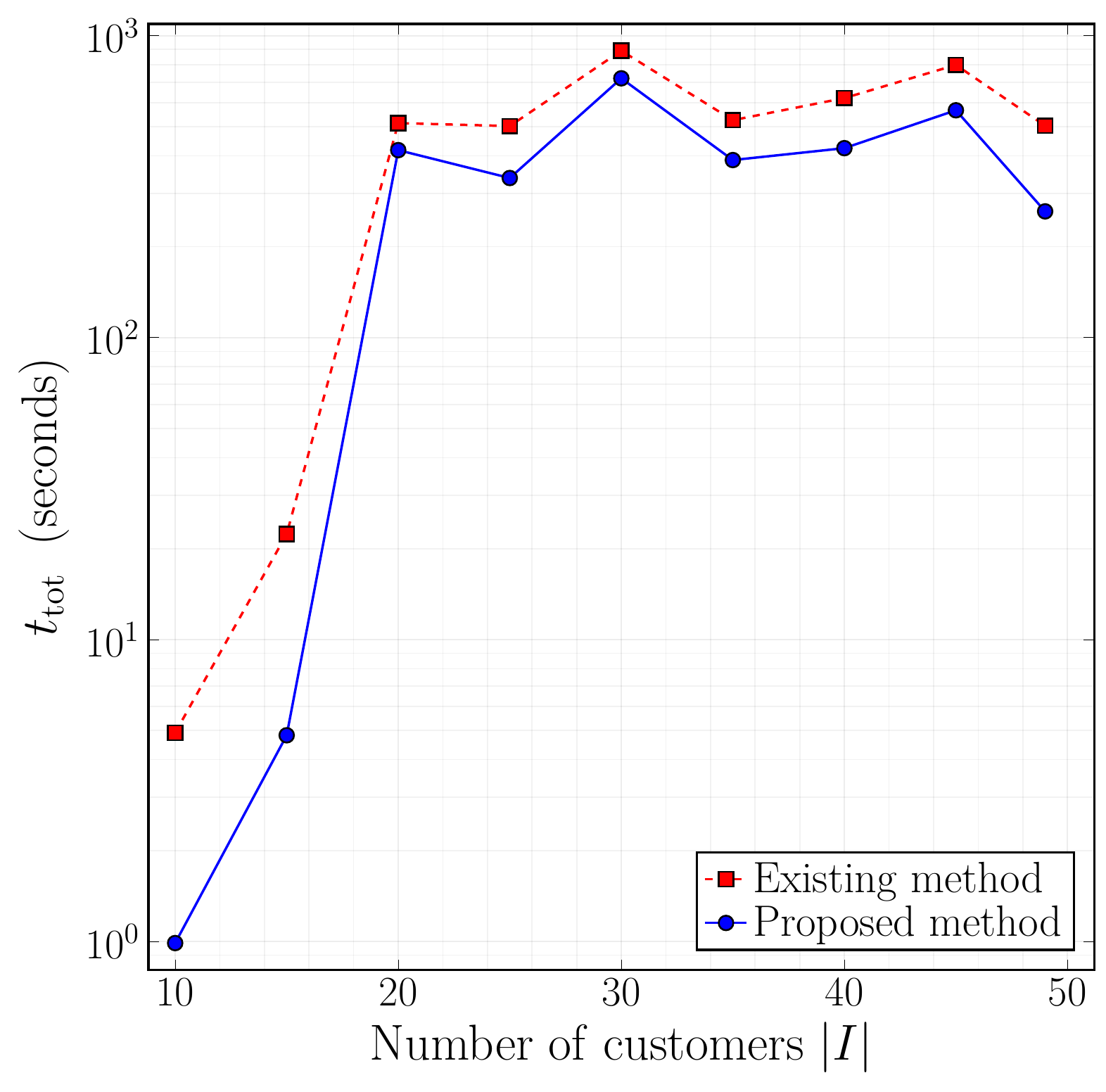}
        \caption{Total run time of algorithm}\label{figure:facility:total}
    \end{subfigure}
    \caption{Comparison of computational times of the existing method in \cite{cheng2021robust} against the proposed method for the facility location problem.}
    \label{figure:facility}
\end{figure}

Tables~\ref{table:facility:instance} and~\ref{table:facility:budget}, and Figure~\ref{figure:facility:twc} show that the proposed method can offer an average speedup of more than a factor of $240$ for computing worst-case parameter realizations.
However, unlike the network design application, the majority of the total run time is spent in solving the lower-bounding problems \eqref{eq:ccg_lb_update} and \eqref{eq:ccg_indicator_lb_update}; therefore, the speedup in computing worst-case parameter realizations translates only partially to a speedup in the total run time of the algorithm.
This is particularly evident for larger instances as shown in Figure~\ref{figure:facility:total}.
Nevertheless, the decrease in total run time is an encouraging sign.
Interestingly, and in contrast to the network design application, Table~\ref{table:facility:instance} shows that whenever both methods terminate in the time limit, they execute the same number of column-and-constraint generation iterations.
This is likely because the second-stage problem $\mathcal{Q}_I(\bm{x}, \bm{\xi})$ always has a unique optimal solution; indeed, in this case, the KKT-based \eqref{eq:worst_case_problem_kkt} and Lagrangian duality-based \eqref{eq:worst_case_problem_indicator_lagrangian_duality} formulations (the latter corresponding to the optimal Lagrange multiplier) are guaranteed to identify the same worst-case parameter realization in every iteration of the algorithm.

Finally, we also compare the computational performance of the Benders decomposition scheme described in Algorithm~\ref{algo:benders:indicator} with the column-and-generation method.
Similar to the latter, the former uses Algorithm~\ref{algo:wc:indicator} for computing worst-case parameter realizations.
Therefore, any differences in their performance can be attributed primarily to differences in their lower-bounding master problems.

Table~\ref{table:facility:benders} summarizes and compares the two methods across the same set of test instances as those reported in Table~\ref{table:facility:budget}.
We observe that for small values of the uncertainty budget parameter $k \in \{1, 2, 3\}$, the column-and-constraint generation algorithm outperforms the Benders decomposition scheme.
However, for $k = 4$, the latter slightly outperforms the former, featuring smaller solution times and optimality gaps.
This difference in performance is because the number of worst-case parameter realizations grows exponentially with $k$.
Each of these adds additional binary variables to the column-and-constraint generation master problem, making its solution significantly more difficult with each iteration.
In contrast, the Benders decomposition scheme adds only a single constraint to its master problem in every iteration.
Although this leads to weaker lower bounds and a larger number of iterations in general (as is also evidenced from the table), this trade-off appears to favor the Benders decomposition scheme for large values of $k$.

\begin{table}[!htb]
    \centering
    \caption{Computational comparison of the Benders decomposition and column-and-constraint generation algorithms for the facility location problem.}
    \label{table:facility:benders}
    \begin{tabularx}{\textwidth}{rr*{3}{R}r*{3}{R}}
        \toprule
        & \multicolumn{4}{c}{Benders decomposition} & \multicolumn{4}{c}{Column-and-constraint generation} \\
        \cmidrule(r){2-5}\cmidrule(l){6-9}
        $k$ & Opt & \# It.$^\dagger$ & Gap (\%) & $t_\text{tot}^\dagger$ (s) & Opt & \# It.$^\dagger$ & Gap (\%) & $t_\text{tot}^\dagger$ (s) \\
        \midrule
        1 &  27/35  &   356.9 &    1.80 &   474.7 &  35/35  &     5.7 &    0.00 &     1.6\\
        2 &  28/35  &   517.2 &    1.98 &   671.2 &  35/35  &    13.0 &    0.00 &    22.6\\
        3 &  24/35  &   362.0 &    2.23 &   223.7 &  29/35  &    25.8 &    0.87 &   222.9\\
        4 &  24/35  &   507.8 &    3.57 &   556.5 &  23/35  &    43.9 &    5.19 &  1289.5\\
        \midrule
          & 103/140 &   436.8 &    2.39 &   488.7 & 122/140 &    21.0 &    1.51 &   349.8\\
        \bottomrule
        \multicolumn{9}{l}{\footnotesize$^\dagger$ Averaged across instances where both methods terminated within the time limit.}
    \end{tabularx}
\end{table}

\subsection{Staff rostering under demand uncertainty}
We finally examine the staff rostering problem that was studied in~\cite{zeng2013solving}.
Consider a service organization that has to assign its regular staff members to shifts to meet unknown service demands, but may also occasionally have to hire part-time staff to meet random demand surges.
Let $I$ and $J$ denote the number of regular and part-time staff, respectively;
and let $T$ and $N$ denote the number of shifts (or time periods) in the scheduling horizon and number of work-hours per shift per staff member, respectively.
Each shift $t \in [T]$ is associated with an uncertain demand $d_t \in [d_t^0, d_t^0 + \bar{d}_t]$, where $d_t^0, \bar{d}_t \in \mathbb{R}_{+}$ are the forecasts of the nominal demand and its deviation; and a penalty cost $M_t \in \mathbb{R}_{+}$ per unit of unmet demand.
The set of possible demand realizations is modeled using a cardinality-constrained uncertainty set with parameter $k \geq 0$:
\begin{equation*}
    \mathcal{D} = \left\{
    \bm{d} \in \mathbb{R}_{+}^T: \exists \bm{\xi} \in \Xi: \bm{d} = \bm{d}^0 + \bar{\bm{d}} \circ \bm{\xi}
    \right\}, \text{ and }
    \Xi = \left\{
    \bm{\xi} \in \{0, 1\}^T : \one^\top \bm{\xi} \leq k
    \right\}.
\end{equation*}
Furthermore, each regular staff member $i \in [I]$ has a fixed wage cost $c_{it} \in \mathbb{R}_{+}$, and a minimum $l_i \in \mathbb{R}_{+}$ and maximum $u_i \in \mathbb{R}_{+}$ number of shifts that they can work over the planning horizon.
Similarly, each part-time staff member $j \in [J]$ has a fixed and hourly wage cost of $f_{jt} \in \mathbb{R}_{+}$ and $h_{jt} \in \mathbb{R}_{+}$, respectively, and minimum $a_j \in \mathbb{R}_{+}$ and maximum $b_j \in \mathbb{R}_{+}$ number of shifts they can work over the planning horizon.
The goal is to allocate (before observing any demand) the allocation of regular staff members to shifts,
so that the total costs of hiring any additional part-time staff members and penalty costs for unserved demands is minimized under the worst-case demand realization.
This problem can be formulated as the following instance of \ref{eq:two_stage_ro_general}.
\begin{equation*}
    \begin{aligned}
        &\inf_{\bm{x} \in \mathcal{X}} \sup_{\bm{\xi} \in \Xi} \, \mathcal{Q}(\bm{x}, \bm{\xi}), \\
        & %
        \mathcal{Q}(\bm{x}, \bm{\xi}) = 
        \left[\begin{aligned}
            \mathop{\text{minimize}}_{(\bm{y}, \bm{v}) \in \mathcal{Y}, \bm{w} \in \mathbb{R}_{+}^T} & \bm{c}^\top \bm{x}  + \bm{f}^\top \bm{y} + \bm{d}^\top \bm{u} + \bm{M}^\top \bm{w}  \\
            \text{subject to} \;\;\; &  N \sum_{i \in [I]} x_{it} + \sum_{j \in [J]} v_{jt} + w_t \geq d_t^0 + \bar{d}_t \xi_t, \;\; t \in [T]
        \end{aligned}\right],
    \end{aligned}
\end{equation*}
where the feasible sets of first- and second-stage decisions are as follows:
\begin{gather*}
    \mathcal{X} = \left\{
    \bm{x} \in \{0, 1\}^{I \times T}: \begin{aligned}
        & x_{it} + x_{i,t+1} + x_{i,t+2} \leq 2, \;\; i \in [I], \; t\in [T-2] \\
        & l_i \leq \sum_{t \in [T]} x_{it} \leq u_i, \;\; i \in [I]
    \end{aligned}
    \right\}, \\
    \mathcal{Y} = \left\{
    (\bm{y}, \bm{v}) \in \{0, 1\}^{J \times T} \times \mathbb{R}_{+}^{J \times T}: \begin{aligned}
        & y_{jt} + y_{j,t+1} \leq 1, \;\; j \in [J], \; t\in [T-1] \\
        & a_j \leq \sum_{t \in [T]} y_{jt} \leq b_j, \;\; j \in [J] \\
        & v_{jt} \leq N y_{jt}, \;\; j \in [J], \; t \in [T]
    \end{aligned}
    \right\}.
\end{gather*}
The first constraints in $\mathcal{X}$ and $\mathcal{Y}$ model the fact that regular and part-time staff members cannot work more than 3 and 2 consecutive shifts, respectively.
Note that the second-stage decision $y_{jt}$ is binary and it indicates whether part-time staff member $j \in [J]$ is hired in a particular shift $t \in [T]$ whereas the second-stage decision $v_{jt}$ is continuous and it record the number of hours part-time staff member $j$ is hired to work for in shift $t$.
Therefore, this is an example of a problem featuring mixed-integer second-stage recourse decisions.
We refer readers to \cite{zeng2013solving} for additional details about the model formulation.

We use the same experimental setup as \cite{zeng2013solving} and randomly generate 10 instances, each of size $(I, J, T) = (12, 3, 21)$.
However, since these are solved relatively fast, we also consider larger instances of twice the size: $(I, J, T) = (24, 6, 42)$.
In doing so, we also scale the parameters $\bm{l}, \bm{u}$, $\bm{a}, \bm{b}$, $\bm{d}^0$, and $\bar{\bm{d}}$ by a factor of two, but keep all cost parameters unchanged.
In \cite{zeng2013solving}, the problem is solved using a two-level column-and-constraint generation method, where the upper bounding problem of the inner-level algorithm (see line~\ref{algo:ccg:inner:optimality:ub-update} of Algorithm~\ref{algo:ccg:inner}) is solved using a KKT-based mixed-integer bilinear formulation \eqref{eq:worst_case_problem_kkt} of the second-stage problem (for fixed values of $\bm{y}$).
The latter is then further reformulated by linearizing the bilinear terms using problem-specific big-M coefficients and solved using an MILP solver.
We benchmark (our implementation) of this method, where the inner-level column-and-constraint generation is solved using Algorithm~\ref{algo:ccg:inner} instead.
Since the problem satisfies relatively complete recourse by construction, we skip the computation of the worst-case constraint violations.
Also, since the second set of conditions in Theorem~\ref{theorem:closed_form_lu} are satisfied,
we directly compute the optimal Lagrange multiplier using Theorem~\ref{theorem:optimal_multiplier_general}.
In all cases, we set a total time limit of 2~hours.

\begin{table}[!htb]
    \centering
    \caption{Computational comparison of the algorithm in \cite{zeng2013solving} and the proposed method for the staff rostering problem.}
    \label{table:rostering}
    \begin{tabularx}{\textwidth}{cr*{3}{R}r*{3}{R}R}
        \toprule
        & \multicolumn{4}{c}{Algorithm in \cite{zeng2013solving}} & \multicolumn{4}{c}{Proposed method} & \\
        \cmidrule(r){2-5}\cmidrule(l){6-9}
        $(T, k)$ & Opt & \# It. & Gap (\%) & $t_\text{tot}^\dagger$ (s) & Opt & \# It. & Gap (\%) & $t_\text{tot}^\dagger$ (s) & $t^{\text{speedup},\dagger}_\text{tot}$ \\
        \midrule
        (12, 3) & 10/10  &    71.3 &    0.00 &     9.0 & 10/10  &    71.8 &    0.00 &     9.3 &     1.0 \\
        (12, 6) & 10/10  &   145.4 &    0.00 &    20.5 & 10/10  &   147.4 &    0.00 &    23.4 &     0.9 \\
        (12, 9) & 10/10  &   119.9 &    0.00 &   548.3 & 10/10  &   117.5 &    0.00 &   300.1 &     1.8 \\
        \midrule
                & 30/30  &   112.2 &    0.00 &   192.6 & 30/30  &   112.2 &    0.00 &   111.0 &     1.7 \\
        \midrule
        (24, 3) &  4/10  &   679.5 &    0.62 &   808.8 &  5/10  &   679.8 &    0.51 &   512.6 &     1.6 \\
        (24, 6) &  1/10  &   556.0 &    2.17 &   436.7 &  2/10  &   506.0 &    2.15 &   183.3 &     2.4 \\
        (24, 9) &  1/10  &   780.0 &    2.61 &  1125.5 &  1/10  &   771.0 &    2.69 &   488.6 &     2.3 \\
        \midrule
                &  6/30  &   675.7 &    1.80 &   799.5 &  8/30  &   666.0 &    1.78 &   453.8 &     1.8 \\
        \bottomrule
        \multicolumn{10}{l}{\footnotesize$^\dagger$ Averaged across instances where both methods terminated within the time limit.}
    \end{tabularx}
\end{table}

Table~\ref{table:rostering} summarizes and compares the computational performance of the algorithm in \cite{zeng2013solving} against our proposed method.
The columns \# It. and Gap (\%) report, respectively, the total number of iterations in the inner-level column-and-constraint generation algorithm (summed across all outer-level iterations) and the optimality gap defined as $100\% \times (UB - LB)/UB$.
Thus, the former quantity is equal to the number of times that the upper bounding problem in line~\ref{algo:ccg:inner:optimality:ub-update} of Algorithm~\ref{algo:ccg:inner} is solved.
We observe that although the larger class of instances with $T = 24$ shifts is difficult to solve using either method, our proposed algorithm is able to solve more instances and with a smaller optimality gap, with an average speedup (in total run time) of roughly~1.75.
Finally, although not shown in this table, the number of outer-level iterations for $T=24$ never exceeded 10 for either method and for any value of $k$.
This is because the solution of these instances is limited more by the lower bounding problem \eqref{eq:ccg_lb_update} in the column-and-constraint generation algorithm, than by the computation of worst-case parameter realizations.

\section{Conclusions}\label{sec:conclusions}

This paper presented new Lagrangian dual formulations for two-stage robust optimization problems with binary-valued uncertain data.
Their crucial features are the use of apropriately defined penalty functions that move all uncertainties from the constraints to the objective function using a finite scalar-valued Lagrange multiplier.
The proposed formulations satisfy strong duality properties even in the presence of discrete first- and second-stage decisions, uncertain objective functions, and in the absence of relatively complete recourse.
We showed how they can be used to efficiently calculate worst-case parameter realizations that can be readily integrated in existing high-level exact algorithms.
Notably, they circumvent many of the numerical issues that plague existing methods, including smaller MILP formulations and a lack of reliance on decision-independent bounds on dual variables.
The proposed methods can exploit problem structure in applications where the binary parameters switch on or off constraints, and they can potentially enable existing algorithms that address uncertainty only in the objective function to solve also problems with uncertainty-affected constraints.
Numerical experiments on a diverse set of network design, facility location and staff rostering test instances showed that they can provide computational speedups over existing methods that have been tailored for those specific applications.

We believe that this work can be extended along several directions.
First, we need to build extensions of the proposed Lagrangian dual formulations and associated exact methods for problems featuring random recourse or mixed-integer uncertainties
(including single-stage problems with mixed-integer decisions).
Looking beyond exact algorithms, it would also be interesting to assess whether the proposed ideas can be useful in other contexts as well, including but not limited to approximations schemes for multi-stage robust optimization problems, and worst-case uncertainty quantification for nonlinear models.

\bibliographystyle{plain}      %
\bibliography{bibliography}   %

\noindent\fbox{\parbox{0.95\textwidth}{
    The submitted manuscript has been created by UChicago Argonne, LLC, Operator of Argonne National Laboratory (``Argonne''). Argonne, a U.S. Department of Energy Office of Science laboratory, is operated under Contract No. DE-AC02-06CH11357. The U.S. Government retains for itself, and others acting on its behalf, a paid-up nonexclusive, irrevocable worldwide license in said article to reproduce, prepare derivative works, distribute copies to the public, and perform publicly and display publicly, by or on behalf of the Government. The Department of Energy will provide public access to these results of federally sponsored research in accordance with the DOE Public Access Plan (http://energy.gov/downloads/doe-public-access-plan).}
}

\end{document}